\documentclass[a4paper,11pt,reqno]{amsart}

\usepackage[utf8]{inputenc}
\usepackage[english]{babel}
\usepackage{amsmath, mathtools}
\usepackage{amssymb}
\usepackage{amsfonts}
\usepackage{amsthm}
\usepackage{bbm}
\usepackage{enumitem}
\usepackage{xcolor}
\usepackage[colorlinks=false]{hyperref}

\usepackage{ esint }

\newtheorem{theorem}{Theorem}[section]

\newtheorem{lemma}[theorem]{Lemma}
\newtheorem{proposition}[theorem]{Proposition}

\theoremstyle{definition}

\newtheorem{remark}[theorem]{Remark}

\numberwithin{equation}{section}

\newcommand{\R}{\mathbb{R}}
\newcommand{\N}{\mathbb{N}}

\newcommand{\eps}{\varepsilon}
\newcommand{\e}{\varepsilon}

\newcommand{\A}{\mathcal{A}}

\newcommand{\hs}{{\mathcal H}}
\newcommand{\ds}{\displaystyle}

\newcommand{\Sph}{{\mathbb S}}
\def\diam{{\rm diam}}

\newcommand{\ie}{{; \it i.e., }}

\newcommand{\Rd}{{\R}^d}
\newcommand{\Sd}{{\mathbb{S}}^{d-1}}
\newcommand{\dx}{\, \mathrm{d} x}

\newcommand{\hd}{\mathcal{H}^{d-1}}
\newcommand{\Ld}{{\mathcal{L}}^d}

\newcommand{\BBB}{\color{black}} 
\newcommand{\UUU}{\color{blue}}

\newcommand{\EEE}{\color{black}}

\usepackage{latexsym,amsfonts}
\usepackage{mathrsfs}
\usepackage{centernot}

\numberwithin{equation}{section}

\begin{document}

\title[Free-discontinuity problems  with $p(\cdot)$-growth]{Integral representation and $\Gamma$-convergence for  free-discontinuity problems  with $p(\cdot)$-growth}

\author{Giovanni Scilla}
\address[Giovanni Scilla]{Department of Mathematics and Applications ``Renato Caccioppoli'', University of Naples ``Federico II'', Via Cintia, Monte S. Angelo
80126 Naples, Italy}
\email[Giovanni Scilla]{giovanni.scilla@unina.it}

\author{Francesco Solombrino}
\address[Francesco Solombrino]{Department of Mathematics and Applications ``Renato Caccioppoli'', University of Naples ``Federico II'', Via Cintia, Monte S. Angelo
80126 Naples, Italy}
\email{francesco.solombrino@unina.it}

\author{Bianca Stroffolini}
\address[Bianca Stroffolini]{Department of Mathematics and Applications ``Renato Caccioppoli'', University of Naples ``Federico II'', Via Cintia, Monte S. Angelo
80126 Naples, Italy}
\email{bstroffo@unina.it}



\begin{abstract}
 An integral representation result for free-discontinuity energies defined on
the space $GSBV^{p(\cdot)}$ of generalized special functions of bounded variation with variable exponent is proved, under the assumption of log-H\"older continuity for the variable exponent $p(x)$. Our analysis is based on a variable exponent version of the global method for relaxation devised in \cite{BFLM} for a constant exponent. 
We prove $\Gamma$-convergence of sequences of energies of the same type, we identify the limit integrands in terms of asymptotic cell formulas and prove a non-interaction property between bulk and surface contributions.
\end{abstract}

\maketitle


\section{Introduction}
Originally introduced in the setting of image restoration \cite{MS}, free-discontinuity functionals are by now ubiquitous   in the mathematical modeling of elastic solids with surface discontinuities, including phenomena as fracture, damage,   or material voids.
If $u$ is the variable of the problem (representing, e.g., the output image or the deformation of the solid), these problems are characterized  by the competition between a ``bulk'' energy, usually taking the form of a variational integral
\begin{equation}\label{introeq:bulk}
\int_\Omega f(x, u(x), \nabla u(x))\,\mathrm{d}x
\end{equation}
where $\Omega$ is a reference configuration, and a ``surface" energy of the form
\begin{equation}\label{introeq:surface}
\int_{J_u} g\big(x,  u^+(x), u^-(x), \nu_u(x)\big) \, {\rm d} \mathcal{H}^{d-1} (x) 
\end{equation}
where $J_u$ is the set of discontinuities of $u$ with normal $\nu_u$. This latter term is, for instance, accounting for the ``cost''  of an interface in the image (enforcing optimal segmentation), or for the energy spent to produce a crack (\cite{Griffith:1921, FrancfortMarigo:1998}).
If one imposes a $p$-growth assumption of the form
\[
c|\xi|^p \le f(x,u, \xi)\le C(1+|\xi|^p)
\]
with $p>1$ on the bulk integrand $f$, and $g \geq \alpha >0$, then the existence of minimizers is guaranteed in the space of Generalized Special functions of Bounded Variation ($GSBV$) whenever $f$ is quasiconvex and $g$ $BV$-elliptic, see \cite{Ambrosio-Fusco-Pallara:2000}. In particular, compactness of minimizing sequences with respect to the convergence in measure can be recovered by {\sc Ambrosio}'s results (\cite{A89, Amb}) if some lower order fidelity terms are included in the problem, or some boundary data are considered (see \cite{Fr}). 

A wide attention has been also paid, over the last two decades, to the theory of variational limits of free-discontinuity functionals, with  applications  in various contexts, such as homogenization, dimension reduction, or  atomistic-to-continuum approximations.  Starting from the first results  on the subspace $SBV$ of special functions of bounded variation \cite{BFLM, BraChP96, Braides-Defranceschi-Vitali}  and on piecewise constant functions \cite{AmbrosioBraides}, this analysis has been further improved to deal with
functionals  and variational limits  on   $GSBV^p$  (generalized   special  functions of bounded variation  with $p$-integrable bulk density), see, e.g., \cite{BacBraZep18, BacCicRuf19, barfoc,  BarLazZep16, CDMSZ, focgelpon07, Fr}. Most of these results are based on the so-called \emph{global method for relaxation}, which has been    developed by {\sc Bouchitt\'{e}, Fonseca, Leoni, and Mascarenhas}       in \cite{BFLM, BFM}.  This very powerful method is essentially based on comparing asymptotic Dirichlet problems
on small balls with different boundary data depending on the local properties of the functions and allows one to characterize limit energy densities in terms of cell formulas.  Recently, it has also been used for analyzing the limit behavior of free-discontinuity problems  in the space $GSBD$ of generalized functions of bounded deformations,  involving the symmetric gradient, see for example \cite{CCS, ChCri, CFI, CFS}, with applications to crack energies in linear elastic materials.

The topic of the present paper are, instead, free discontinuity problems in variable exponent spaces. These spaces were originally considered by the Russian school, see  \cite{SH} and the Czech one  \cite{KR}. Subsequently, motivated by models for the behavior of composite materials, {\sc Zhikov}  initiated the so-called theory of variational integrals with non standard growth in the mid 80's. Since then, the subject of variable exponent spaces has undergone a large interest, both from the standpoint of regularity theory (see \cite{Z2} for the scalar case and \cite{CosM, AM}  for the vectorial one) and  in view of applications ranging from electrorheological fluids  right up to homogenization, see \cite{RZ, R, Zh, ZhKO}, and the references in \cite{CF, DHHR}. Motivated by the aforementioned applications, in \cite{CM} {\sc Coscia and Mucci} analyzed the $\Gamma$-convergence of variational integrals of the form \eqref{introeq:bulk} with a $p(x)$-growth condition
\begin{equation}\label{introeq:growth}
c|\xi|^{p(x)} \le f(x,u, \xi)\le C(1+|\xi|^{p(x)})\,,
\end{equation}
where $p(x)\geq 1+\delta>1$ is a variable exponent, in the Sobolev space $W^{1, p(\cdot)}(\Omega; \R^m)$. They proved that the $\Gamma$-limit of these energies is still an integral functional of the same type and growth,  under a key assumption on the modulus of continuity of the variable exponent, the so-called \emph{log-H\"older continuity}, see \eqref{eq:loghold} below. In some sense, this condition says that we can freeze the exponent on small balls around a point, as pointed out in \cite[Lemma~3.2]{Dieni} (see also Lemma \ref{lem:dieni3.2} below). As such, it is particularly suitable for blow-up methods: for instance, in \cite{ABF} it allows the authors to prove the singular part of the measure representation of relaxed functionals with growth \eqref{introeq:growth} disappears. More in general, log-H\"older continuity plays a central role in the theory of functionals with $p(x)$-growth, as  {\sc Zhikov} proved in \cite{Z2} that such functionals exhibit the Lavrentiev phenomenon if it is violated.

In recent years, variational problems in spaces of functions of bounded variation with variable integrability exponent on the gradient have been proposed, especially in the setting of image restoration.  In the pioneering paper \cite{CLR} {\sc Chen, Levine and Rao} proposed  for the first time a model   considering a kind of intermediate regime between the $TV$ model and the isotropic diffusion away from the edges (see also \cite{Hasto} for a related model, \cite{Li} for simulations, and \cite{Hasto2} for a $\Gamma$-convergence result). Observe that in these models, the value $p(x)=1$ is allowed. A related, but different, point of view takes instead into account the coupling of a {\it strictly superlinear} bulk energy \eqref{introeq:bulk} under the growth conditions \eqref{introeq:growth} with a surface energy \eqref{introeq:surface}, which can be seen as a variable-exponent version of Mumford-Shah-type functionals\footnote{Indeed, for such functionals, also in the case of a constant integrability exponent it is customary to assume $p>1$, in order to get the scale separation effect we describe later on.}. This kind of functionals will constitute the object of the present paper. From an analytical point of view, they were considered in \cite{DCLV}. There, provided the bulk integrand is quasiconvex and the exponent is log-H\"older continuous, a lower semicontinuity result for sequences with bounded energy has been proved, which entails well-posedness of such variational problems in the subspace $SBV^{p(\cdot)}$ of $SBV$ functions with $p(\cdot)$-integrable gradients (again, if some lower order terms are added to the problem in order to apply Ambrosio's compactness Theorem).

\noindent \emph{Description of our results.} This leads us to the purpose of the present paper.  Our focus is to study the $\Gamma$-convergence (with respect to the convergence in measure) for functionals $\mathcal{F}_j\colon GSBV^{p(\cdot)}(\Omega;\R^m) \to [0,+\infty)$ of the form 
\begin{align}\label{eq:energyintro}
\mathcal{F}_j(u) = \int_{\Omega} f_j\big(x, \nabla u(x) \big) \, {\rm d}x +  \int_{J_u} g_j\big(x,  [u](x),  \nu_u(x)\big) \, {\rm d} \mathcal{H}^{d-1} (x) 
\end{align} 
for each $u \in GSBV^{p(\cdot)}(\Omega;\R^m)$, where $[u](x):=u^+(x)-u^-(x)$. The variable exponent $p(\cdot)$ is assumed to be log-H\"older continuous, with $p(x) \geq p^- > 1$ for all $x$ (see \ref{assP1}).	
We assume that the bulk integrands $f_j$ satisfy \eqref{introeq:growth} uniformly in $j$, while the surface integrands $g_j$
satisfy
\[
0<\alpha\le g_j(x, \zeta, \nu)\le \beta\,.
\]  
Under a fairly general set of assumptions, devised in \cite{CDMSZ}, we are able to show that the $\Gamma$-limit is again an integral functional of the same form (Theorem \ref{th: gamma}). Furthermore, as shown in Section \ref{sec:identification},  due to the assumption $p(x) \geq p^- > 1$ a {\it separation of scales} effect takes place, exactly as in the case of a constant exponent: bulk and surface effects decouple in the limit.  Namely, the bulk limit density $f_\infty$ is completely determined by taking the $\Gamma$-limit of the functionals \eqref{introeq:bulk} in the Sobolev space $W^{1,p(\cdot)}$, while the surface limit density $g_\infty$ can be recovered from the sole $g_j$'s via an asymptotic cell formula on piecewise constant functions, that is $GSBV$ functions whose gradient is a.e.\ equal to $0$ .

As we mentioned, for the proof of Theorem \ref{th: gamma} we follow quite closely the  {\it global method for relaxation} of \cite{BFLM}. The main point is recovering an integral representation for functionals 
$$\mathcal{F}\colon GSBV^{p(\cdot)}(\Omega;\R^m) \times \mathcal{B}(\Omega) \to  [0,+\infty)$$ 
(here $\mathcal{B}(\Omega)$  denote the Borel subsets of $\Omega$) that satisfy the standard abstract conditions to be Borel measure in the second argument, lower semicontinuity with respect to the convergence in measure, and  local in the first argument. In addition, we require a coercivity and control condition of variable exponent type:  there exist $0 < \alpha  < \beta $ such that for any $u \in GSBV^{p(\cdot)}(\Omega;\R^m)$ and $B \in \mathcal{B}(\Omega)$  we have
\begin{small}
$$
\alpha \bigg(\int_{ B  } |\nabla u|^{p(x)}  \dx    +   \mathcal{H}^{d-1}(J_u \cap B)\bigg) \le \mathcal{F}(u,B) \le \beta \bigg(\int_{ B } (1 + |\nabla u|^{p(x)})    \dx  +   \mathcal{H}^{d-1}(J_u \cap B)\bigg).
$$ 
\end{small}
The result is proved in Theorem \ref{thm: int-representation-gsbv}. The proof strategy recovers the integral bulk and surface densities as blow-up limits of cell minimization formulas, as a consequence of the  estimates in Lemmas \ref{lem:lemma2fonseca} and \ref{lem:Lemma3fon}. In particular, in this latter the interplay between the asymptotic estimates and the variable exponent setting causes some nontrivial difficulties, which are overcome by means of assumption \eqref{eq:loghold}. It allows us to estimate the asymptotic distance between a suitable modification of $u$ and its blow-up at jump points in some variable exponent space, keeping bounded some constants which depend on the oscillation of $p(\cdot)$ in a small ball around the blow-up point (see equation \eqref{eq:crucial}).
The log-H\"older continuity assumption plays also a crucial role in Theorem \ref{thm: fsob}, where separation of scales for the bulk energy is shown. There, a Lusin-type approximation for $SBV$ functions is used to reduce the asymptotic minimization problems defining the cell formula for the bulk energy to the (variable exponent) Sobolev setting. Again, via \eqref{eq:loghold} it is possible to estimate the rest term coming from this approximation (see equations \eqref{eq:stimapk}--\eqref{eq:7.22quater}).

Our results can be also adapted to the case where the surface integrands $g_j$'s satisfy a more general growth condition, as in \cite{CDMSZ}, namely
\[
\alpha\le g_j(x,\zeta, \nu)\le \beta(1+|\zeta|)\,.
\]
This can be done by first establishing  the integral representation in the  $SBV^{p(\cdot)}$ case for functionals which satisfy
\begin{small}
\[
\begin{split}
\alpha \bigg(\int_{ B  } |\nabla u|^{p(x)}  \dx    +   \int_{J_u \cap B}(1+|[u]|)\,\mathrm{d}\mathcal{H}^{d-1}\bigg) &\le \mathcal{F}(u,B)
\\
& \le \beta \bigg(\int_{ B } (1 + |\nabla u|^{p(x)})    \dx  +  \int_{J_u \cap B}(1+|[u]|)\,\mathrm{d}\mathcal{H}^{d-1}\bigg).
\end{split}
\]
\end{small}
The analysis can be reconducted to this setting by a \emph{perturbation trick}:  one considers a small perturbation of the functional, depending on the jump opening, to represent functionals on  $SBV^{p(\cdot)}$.  Then, by letting the perturbation parameter vanish and by truncating functions suitably,  the representation can be extended to  $GSBV^{p(\cdot)}$. In order to do this, one can follow quite closely the arguments in \cite{CDMSZ}, with some minor changes due to the variable exponent setting: for the sake of completeness and self-containedness, statements and proofs are given in Appendix \ref{sec:appendix}.

\emph{Outline of the paper.} The paper is structured as follows. In Section~\ref{sec: prel} we fix the basic notation and recall some basic facts about Lebesgue spaces with variable exponent (Section~\ref{sec:variableexp}). Then, in Section~\ref{sec:gsbvp(x)} we introduce the space $GSBV^{p(\cdot)}$, and prove some regularity and compactness properties useful in the sequel. Section~\ref{sec: main} is entirely devoted to the proof of the integral representation result in $GSBV^{p(\cdot)}$. Specifically, in Section~\ref{sec:fundamental} we prove a fundamental estimate, which is a key tool for the global method, Section~\ref{sec: global method}. The proofs of the necessary blow-up properties are postponed to Sections~\ref{sec:bulk} and \ref{sec:surf}. In Section~\ref{sec:gammaconv} we prove a $\Gamma$-convergence result for sequences of free-discontinuity functionals defined on $GSBV^{p(\cdot)}$. The identification of the $\Gamma$-limit is contained in Section~\ref{sec:identification}. Eventually, in Appendix~\ref{sec:appendix}, we develop the analysis of Sections~\ref{sec: main} and \ref{sec:gammaconv} for free-discontinuity energies with a weaker growth condition from above in the surface term.

\section{Basic notation and preliminaries}\label{sec: prel}

 We start with some basic notation.   Let $\Omega \subset \R^d$  be  open, bounded with Lipschitz boundary. Let $\mathcal{A}(\Omega)$ be the  family of open subsets of $\Omega$,  and denote by   $\mathcal{B}(\Omega)$  the family of Borel sets contained in $\Omega$.  For every $x\in \Rd$ and $\eps>0$ we indicate by $B_\eps(x) \subset \Rd$ the open ball with center $x$ and radius $\eps$. If $x=0$, we will often use the shorthand $B_\eps$. For $x$, $y\in \Rd$, we use the notation $x\cdot y$ for the scalar product and $|x|$ for the  Euclidean  norm.   Moreover, we let  $\Sd:=\{x \in \Rd \colon |x|=1\}$, we denote  by $\mathbb{R}^{m \times d}$ the set of $m\times d$ matrices and by $\R^d_0$ the set $\R^d\backslash\{0\}$. The $m$-dimensional Lebesgue measure of the unit ball in $\R^m$ is indicated by $\gamma_m$ for every $m \in \N$.   We denote by $\Ld$ and $\mathcal{H}^k$ the $d$-dimensional Lebesgue measure and the $k$-dimensional Hausdorff measure, respectively.  For $A \subset \R^d$, $\eps>0$, and $x_0 \in \R^d$ we set
\begin{align}\label{eq: shift-not}
A_{\eps,x_0} := x_0 + \eps (A - x_0).    
\end{align} 
The closure of $A$ is denoted by $\overline{A}$. The diameter of $A$ is indicated by ${\rm diam}(A)$.  Given two sets $A_1,A_2 \subset \R^d$, we denote their symmetric difference by $A_1 \triangle A_2$.   We write $\chi_A$ for the  characteristic  function of any $A\subset  \R^d$, which is 1 on $A$ and 0 otherwise.  If $A$ is a set of finite perimeter, we denote its essential boundary by $\partial^* A$,   see  \cite[Definition 3.60]{Ambrosio-Fusco-Pallara:2000}.   
The notation $L^0(E; \R^m)$ will be used for the space of Lebesgue measurable function from some measurable set $E\subset \R^n$ to $\R^m$, endowed with the convergence in measure.

\subsection{Variable exponent Lebesgue spaces.}\label{sec:variableexp}

We briefly recall the notions of variable exponents and variable exponent Lebesgue spaces. We refer the reader to \cite{DHHR} for a comprehensive treatment of the topic. 

A measurable function $p:\Omega\to[1,+\infty)$ will be called a \emph{variable exponent}. Correspondingly, for every $A\subset\Omega$ we define
\begin{equation*}
p^+_A:=\mathop{{\rm ess}\,\sup}_{x\in A} p(x)\,\mbox{\,\, and \,\,}p^-_A:=\mathop{{\rm ess}\,\inf}_{x\in A} p(x)\,,
\end{equation*}
while $p^+_\Omega$ and $p^-_\Omega$ will be denoted by $p^+$ and $p^-$, respectively.

For a measurable function $u:\Omega\to\R^m$ we define the \emph{modular} as
\begin{equation*}
\varrho_{p(\cdot)}(u):=\int_\Omega|u(x)|^{p(x)}\,\mathrm{d}x
\end{equation*}
and the (Luxembourg) \emph{norm}
\begin{equation*}
\|u\|_{L^{p(\cdot)}(\Omega)}:=\inf\{\lambda>0:\,\, \varrho_{p(\cdot)}(u/\lambda)\leq1\}\,.
\end{equation*}
The \emph{variable exponent Lebesgue space} $L^{p(\cdot)}(\Omega)$ is defined as the set of measurable functions $u$ such that $\varrho_{p(\cdot)}(u/\lambda)<+\infty$ for some $\lambda>0$. In the case $p^+<+\infty$, $L^{p(\cdot)}(\Omega)$ coincides with the set of functions such that $\varrho_{p(\cdot)}(u)$ is finite. It can be checked that $\|\cdot\|_{L^{p(\cdot)}(\Omega)}$ is a norm on $L^{p(\cdot)}(\Omega)$. Moreover, if $p^+<+\infty$, it holds that
\begin{equation}
\varrho_{p(\cdot)}(u)^\frac{1}{p^+} \leq \|u\|_{L^{p(\cdot)}(\Omega)} \leq \varrho_{p(\cdot)}(u)^\frac{1}{p^-}
\label{eq:normineq}
\end{equation}
if $\|u\|_{L^{p(\cdot)}(\Omega)}>1$, while an analogous inequality holds by exchanging the role of $p^-$ and $p^+$ if $0\leq \|u\|_{L^{p(\cdot)}(\Omega)}\leq 1$. Another useful property of the modular, in the case $p^+<+\infty$, is the following one: 
\begin{equation}
\min\{\lambda^{p^+}, \lambda^{p^-}\}\varrho_{p(\cdot)}(u) \leq \varrho_{p(\cdot)}(\lambda u) \leq \max\{\lambda^{p^+}, \lambda^{p^-}\}\varrho_{p(\cdot)}(u)
\label{eq:normineq2}
\end{equation}
for all $\lambda>0$.

We say that a function $p:\Omega\to\R$ is \emph{log-H\"older continuous} on $\Omega$ if
\begin{equation}
\exists C>0 \mbox{ \,\, such that \,\, } |p(x)-p(y)|\leq \frac{C}{-\log |x-y|}\,,\quad \forall x,y\in\Omega\,,\, |x-y|\leq\frac{1}{2}\,.
\label{eq:loghold}
\end{equation}

We recall the following geometric meaning of the $p$ log-H\"older continuity (see, e.g., \cite[Lemma~3.2]{Dieni}).
\begin{lemma}\label{lem:dieni3.2}
Let $p:\Omega\to[1,+\infty)$ be a bounded, continuous variable exponent. The following conditions are equivalent:
\begin{enumerate}
\item[(i)] $p$ is log-H\"older continuous;
\item[(ii)] for all open balls $B$, we have
\begin{equation*}
\mathcal{L}^d(B)^{(p^-_B-p^+_B)}\leq C_1\,.
\end{equation*}
\end{enumerate}
\end{lemma}

The following lemma provides an extension to the variable exponent setting of the well-known embedding property of classical Lebesgue spaces (see, e.g., \cite[Corollary~3.3.4]{DHHR}).

\begin{lemma}\label{embedding}
Let $p,q$ be measurable variable exponents on $\Omega$, and assume that $\mathcal{L}^d(\Omega)<+\infty$. Then $L^{p(\cdot)}(\Omega)\hookrightarrow L^{q(\cdot)}(\Omega)$ if and only if $q(x)\leq p(x)$ for $\mathcal{L}^d$-a.e. $x$ in $\Omega$. The embedding constant is less or equal to the minimum between $2(1+\mathcal{L}^d(\Omega))$ and $2\max\{\mathcal{L}^d(\Omega)^{(\frac{1}{q}-\frac{1}{p})^+}, \mathcal{L}^d(\Omega)^{(\frac{1}{q}-\frac{1}{p})^-}\}$.
\label{lem:embed}
\end{lemma}

The following result generalizes the concept of Lebesgue points to the variable exponent Lebesgue spaces (see, e.g., \cite[Theorem~3.1]{HH}).

\begin{theorem}\label{thm:lebpoint}
Let $\displaystyle p^+:=\mathop{\rm ess\,sup}_{x\in\R^d}p(x)<+\infty$. If $u\in L^{p(\cdot)}(\R^d)$ then
\begin{equation*}
\lim_{\varepsilon\to0} \frac{1}{\varepsilon^d}\int_{B_\varepsilon(x)}|u(y)-u(x)|^{p(y)}\,\mathrm{d}y=0
\end{equation*}
for a.e. $x\in\R^d$.
\end{theorem}

\subsection{The space $GSBV^{p(\cdot)}$. Poincar\'e-type inequality} \label{sec:gsbvp(x)}

We denote by $SBV^{p(\cdot)}(\Omega;\R^m)$ the set of functions $u\in SBV(\Omega;\R^m)$ with $\nabla u\in L^{p(\cdot)}(\Omega;\mathbb{R}^{m \times d})$ and $\mathcal{H}^{d-1}(J_u)<+\infty$. Here, $\nabla u$ denotes the approximate gradient, 
while $J_u$ stands for the (approximate) jump set with corresponding normal $\nu_u$ and one-sided limits $u^+$ and $u^-$.   We say that $u\in GSBV^{p(\cdot)}(\Omega;\R^m)$ if for every $\phi\in C^1(\R^m)$ with the support of $\nabla \phi$ compact, the composition $\phi\circ u$ belongs to $SBV^{p(\cdot)}_{\rm loc}(\Omega;\R^m)$.

From the inclusion $L^{p(\cdot)}(\Omega)\subset L^{p^-}(\Omega)$ and \cite{Amb}, one can also deduce that for $u\in GSBV^{p(\cdot)}(\Omega)$ the \emph{approximate gradient} $\nabla u$ exists $\mathcal{L}^d$-a.e.\ in $\Omega$.

\begin{lemma}[Approximate gradient]\label{lemma: approx-grad}
Let $\Omega \subset \R^d$ be  open, bounded (with  Lipschitz boundary),   let $p:\Omega\to[1,+\infty]$ be a variable exponent, and $u \in GSBV^{p(\cdot)}(\Omega;\R^m)$. Then for $\mathcal{L}^d$-a.e.\   $x_0 \in \Omega$  there exists a matrix  in $\mathbb{R}^{m\times d}$,   denoted by $\nabla u(x_0)$, such that
$$\lim_{\eps \to 0} \  \eps^{-d} \mathcal{L}^d\Big(\Big\{x \in B_\eps(x_0) \colon \,  \frac{|u(x) - u(x_0) - \nabla u(x_0)(x-x_0)|}{|x - x_0|}  > \varrho   \Big\} \Big)  = 0 \text{ for all $\varrho >0$}. $$    
\end{lemma}

In order to state a Poincar\'e-Wirtinger inequality in $GSBV^{p(\cdot)}$, we first fix some notation, following \cite{BFLM, CL}. With given $a=(a_1,\dots, a_m)$, $b=(b_1,\dots,b_m)\in\R^m$, we denote $a\wedge b:=(\min(a_1,b_1),\dots,\min(a_m,b_m))$ and $a\vee b:=(\max(a_1,b_1),\dots,\max(a_m,b_m))$. Let $B$ be a ball in $\R^d$. For every measurable function $u:B\to\R^m$, with $u=(u_1,\dots,u_m)$, we set
\begin{equation*}
u_*(s;B):= ((u_1)_*(s;B),\dots, (u_m)_*(s;B))\,,\quad {\rm med}(u;B):=u_*\left(\frac{1}{2}\mathcal{L}^d(B);B\right)\,,
\end{equation*}
where
\begin{equation*}
(u_i)_*(s;B):=\inf\{t\in\R:\,\, |\{u_i<t\}\cap B|\geq s\} \qquad \mbox{ for } 0\leq s \leq \mathcal{L}^d(B)\,,
\end{equation*}
for $i=1,\dots,m$.

For every $u\in GSBV^{p(\cdot)}(\Omega;\R^m)$ such that
\begin{equation*}
\left(2\gamma_{\rm iso}\mathcal{H}^{d-1}(J_u\cap B)\right)^{\frac{d}{d-1}} \leq \frac{1}{2}\mathcal{L}^d(B)\,,
\end{equation*}
we define
\begin{equation*}
\begin{split}
\tau'(u;B) & := u_*\left(\left(2\gamma_{\rm iso}\mathcal{H}^{d-1}(J_u\cap B)\right)^{\frac{d}{d-1}};B\right)\,, \\
\tau''(u;B) & := u_*\left(\mathcal{L}^d(B)-\left(2\gamma_{\rm iso}\mathcal{H}^{d-1}(J_u\cap B)\right)^{\frac{d}{d-1}};B\right)\,, 
\end{split}
\end{equation*}
and the truncation operator
\begin{equation}
T_Bu(x):= (u(x)\wedge \tau''(u;B)) \vee \tau'(u;B)\,,
\label{eq:truncated}
\end{equation}
where $\gamma_{\rm iso}$ is the dimensional constant in the relative isoperimetric inequality. 

We recall the following Poincar\'e-Wirtinger inequality for $SBV$ functions with small jump set in a ball, which was first proven in the scalar setting in \cite[Theorem~3.1]{DeGCarLea}, and then extended to vector-valued functions in \cite[Theorem~2.5]{CL}.
In the statement below, the case $p\ge d$ is discussed in \cite[Remark 4.15]{Ambrosio-Fusco-Pallara:2000}.

\begin{theorem}
Let $u\in SBV(B;\R^m)$ and assume that
\begin{equation}
\left(2\gamma_{\rm iso}\mathcal{H}^{d-1}(J_u\cap B)\right)^{\frac{d}{d-1}} \leq \frac{1}{2}\mathcal{L}^d(B)\,.
\label{(10)}
\end{equation}
If $1\leq p < d$ then
\begin{equation}
\left(\int_B |T_Bu-{\rm med}(u;B)|^{p^*}\,\mathrm{d}x\right)^\frac{1}{p^*} \leq \frac{2\gamma_{\rm iso}p(d-1)}{d-p} \left(\int_B|\nabla u|^p\,\mathrm{d}x\right)^\frac{1}{p}
\label{(11)}
\end{equation}
and
\begin{equation}
\mathcal{L}^d(\{T_Bu\neq u\}\cap B) \leq 2 \left(2\gamma_{\rm iso}\mathcal{H}^{d-1}(J_u\cap B)\right)^{\frac{d}{d-1}}\,,
\label{(12)}
\end{equation}
where $p^*:=\frac{dp}{d-p}$.

If $p\ge d$, inequality \eqref{(11)} holds with $p^*$ replaced by an arbitrary $q\in [1, +\infty)$.
\label{thm:poincsbv}
\end{theorem}

\begin{remark}\label{rem:gsbv}
More generally, Theorem~\ref{thm:poincsbv} holds for functions in $GSBV(\Omega;\R^m)$ and for balls $B\subset\subset\Omega$, by applying the scalar result in $SBV$ to truncated functions $u_i^M:=M\wedge u_i \vee -M$ for every $i=1,\dots,m$, up to understand $\nabla u$ and $J_u$ in a weaker sense.
\end{remark}

The analogous result in $GSBV^{p(\cdot)}$ is as follows.

\begin{theorem}\label{thm:thm5}
Let $p:\Omega\to (1,+\infty)$ be measurable and such that
\begin{equation}
\mathrm{either}\,p^-\ge d \quad \mathrm{ or }\quad 1<p^-<d\,,\quad p^+< (p^-)^*\,.
\label{eq:assump}
\end{equation}
Let $B\subset \subset\Omega$ and $u\in GSBV^{p(\cdot)}(B;\R^m)$, and assume that \eqref{(10)} holds. Then
\begin{equation}
\|T_Bu-{\rm med}(u;B)\|_{L^{p(\cdot)}(B;\R^m)} \leq c(1+\mathcal{L}^d(B))^2\mathcal{L}^d(B)^{\frac{1}{d}+\frac{1}{p^+}-\frac{1}{p^-}}
\|\nabla{u}\|_{L^{p(\cdot)}(B;\mathbb{R}^{m\times d})}
\label{(11bis)}
\end{equation}
for some constant $c$ depending on $p^-,d$, and
\begin{equation}
\mathcal{L}^d(\{T_Bu\neq u\}\cap B) \leq 2 \left(2\gamma_{\rm iso}\mathcal{H}^{d-1}(J_u\cap B)\right)^{\frac{d}{d-1}}\,.
\label{(12bis)}
\end{equation}
\end{theorem}

\proof
In view of Remark~\ref{rem:gsbv}, we are reduced to prove the validity of \eqref{(11bis)}. For this, it will suffice to write \eqref{(11)} for $p=p^-$, and then the desired inequality will be a consequence of \eqref{eq:assump} and Lemma~\ref{lem:embed}.
\endproof

A first consequence of Theorem~\ref{thm:thm5} is the following compactness result, which can be seen as the $GSBV^{p(\cdot)}$ counterpart of \cite[Theorem~3.5]{DeGCarLea}. Motivated by the blow-up analysis of Lemma~\ref{lem:lemma2fonseca}, we will prove the result for a fixed ball and a uniformly convergent sequence of continuous variable exponents satisfying \eqref{eq:assump} (see also \cite[Theorem~4.1]{DCLV} for a related result under the additional stronger assumption \eqref{eq:loghold}). 

\begin{theorem}\label{thm:3.5}
Let $B\subset\Omega$ be a ball, $(p_j)_{j\in\N}$ be a sequence of variable exponents $p_j:B\to(1,+\infty)$ complying uniformly with \eqref{eq:assump} and converging uniformly to some $\bar{p}:B\to(1,+\infty)$ in $B$. Let $\{u_j\}_{j\in\mathbb{N}}\subset GSBV^{p_j(\cdot)}(B;\R^m)$ be such that
\begin{equation}
\sup_{j\in\mathbb{N}}\int_B |\nabla u_j|^{p_j(y)}\,\mathrm{d}y < +\infty\,,\quad \lim_{j\to+\infty} \mathcal{H}^{d-1}(J_{u_j}\cap B)=0\,.
\label{eq:ebounded}
\end{equation}
Then there exist a function $u_0\in W^{1,\bar{p}(\cdot)}(B;\R^m)$ and a subsequence (not relabeled) of $\{u_j\}$ such that
\begin{equation}
\begin{split}
\int_B |T_B{u}_j-{\rm med}(u_j;B)-u_0|^{p_j(y)}\,\mathrm{d}y\to0\,,\quad {u}_j-{\rm med}(u_j;B) \to u_0 \quad & \mathcal{L}^d-\mbox{a.e. in $B$}\,.
\end{split}
\label{eq:claims}
\end{equation}
\end{theorem}

\begin{proof}
For every $j\in\N$, we set
\begin{equation*}
p^-_j:= \inf_{y\in B} p_j(y)\,,\quad p^+_j:= \sup_{y\in B} p_j(y)\,.
\end{equation*}
Correspondingly, we define
\begin{equation*}
p^-:=\mathop{\lim\inf}_{j\to+\infty} p^-_j\,,\quad p^+:=\mathop{\lim\sup}_{j\to+\infty} p^+_j\,.
\end{equation*}
We set for brevity $\bar{u}_j:=T_B{u}_j-{\rm med}(u_j;B)$. Let $\eta>0$ be fixed such that $p^-_\eta:=p^--\eta>1$ and $p^+_\eta:=p^++\eta<(p^-_\eta)^*$. Note that, for $j$ large enough, we have
\begin{equation*}
p^-_\eta<p^-_j\leq p_j(\cdot) \leq p^+_j<p^+_\eta \quad \mbox{ on $B$}\,.
\end{equation*}
By virtue of \eqref{eq:assump}, \eqref{(11bis)}, the definition of $T_B{u}_j$ and \eqref{eq:ebounded} we have
\begin{equation*}
\sup_{j\in\mathbb{N}} \left(\|\bar{u}_j\|_{L^{p^-_\eta}(B;\R^m)} + \|\nabla\bar{u}_j\|_{L^{p^-_\eta}(B;\mathbb{R}^{m\times d})} +  \mathcal{H}^{d-1}(J_{u_j}\cap B)\right) <+\infty\,.
\end{equation*}
This implies, by \cite[Theorem~2.2]{Amb} that there exists $u_0\in GSBV^{{p^-_\eta}}(B;\R^m)$ and a subsequence (not relabeled) ${u}_j$ such that $\bar{u}_j \to u_0$ in measure and 
\begin{equation}\label{nojump}
\mathcal{H}^{d-1}(J_{u_0}\cap B) \leq \mathop{\lim\inf}_{j\to+\infty} \mathcal{H}^{d-1}(J_{T_B{u}_j}\cap B)=0\,.
\end{equation}
With \eqref{(11)}, since $p^+_\eta<(p^-_\eta)^*$, we get that   $|\bar{u}_j|^{p^+_\eta}$ is equiintegrable, hence $\bar{u}_j$ strongly converges to $u_0$ in $L^{p^+_\eta}(B;\R^m)$. With Lemma \ref{embedding}, and the definition of $\bar{u}_j$ we then get the first assertion in \eqref{eq:claims}.
With \eqref{nojump}, we have $u_0\in W^{1,{p^-_\eta}}(B;\R^m)$. Now, for each $\eta>0$ we further have
\[
\sup_{j\in\mathbb{N}}\int_B |\nabla u_j|^{\bar{p}(y)-\eta}\,\mathrm{d}y \le C< +\infty
\]
by the uniform convergence of $p_j$. With the weak-$L^1$ convergence of $\nabla \bar{u}_j$ to $\nabla u_0$ and Ioffe's Theorem (see \cite{Ioffe}), we get
\[
\int_B |\nabla u_0|^{\bar{p}(y)-\eta}\,\mathrm{d}y \le C
\]
with a bound independent of $\eta$. Applying the monotone convergence Theorem in the set $\{|\nabla u_0|\ge 1\}$ we get $u_0 \in W^{1,\bar{p}(\cdot)}(B;\R^m)$. The second assertion in \eqref{eq:claims} follows from \eqref{(12bis)} and \eqref{eq:ebounded}.
\end{proof}

To conclude this section, we recall the following result on the approximation of $GSBV$ functions with piecewise constant functions (see \cite[Theorem~4.9]{FS}), which can be seen as a piecewise Poincar\'e inequality and essentially relies on the $BV$ coarea formula. We refer the reader for a proof to \cite[Theorem~2.3]{Manuel}, although the argument can be retrieved in previous literature (see, e.g., \cite{Amb, Braides-Defranceschi-Vitali}).
\begin{theorem}\label{thm:fspoincare}
Let $d\geq1$ and $z\in GSBV(\Omega;\R^m)$ with
\begin{equation*}
\|\nabla z\|_{L^1(\Omega;\R^{m\times d})} + \mathcal{H}^{d-1}(J_z) < +\infty\,.
\end{equation*}
Let $D\subset\Omega$ be a Borel set with finite perimeter. Let $\theta>0$ be fixed. Then there exists a partition $(P_l)_{l=1}^\infty$ of $D$, made of sets of finite perimeter, and a piecewise constant function $z_{\rm pc}:=\sum_{l=1}^\infty b_l \chi_{P_l}$ such that
\begin{enumerate}
\item[$(i)$] $\displaystyle\sum_{l=1}^\infty \mathcal{H}^{d-1}((\partial^*P_l\cap D^1)\backslash J_z)\leq \theta$; \\
\item[$(ii)$] $\|z-z_{\rm pc}\|_{L^\infty(D;\R^m)}\leq c \theta^{-1}\|\nabla z\|_{L^1(D;\R^{m\times d})}$,
\end{enumerate}
for a dimensional constant $c=c(d)>0$, where $D^1$ denotes the set of points with density one. If, in addition, the $i$-th component $z^i$ satisfies the bound $\|z^i\|_{L^\infty(D;\R)}\leq M$, then also $\|z^i_{\rm pc}\|_{L^\infty(D;\R)}\leq M$ holds.
\end{theorem}

\section{The integral representation result}\label{sec: main}

In this section we will establish an integral representation result in the  space $GSBV^{p(\cdot)}(\Omega;\R^m)$ for $m \in \N$, where the variable exponent $p:\Omega\to(1,+\infty)$ complies with the following assumptions
\begin{enumerate}[font={\normalfont},label={(${\rm P}_\arabic*$)}]
\item $p^->1$ and $p^+<+\infty$; \label{assP1}
\item $p$ is log-H\"older continuous on $\Omega$ (see \eqref{eq:loghold}). \label{assP2}
\end{enumerate}
We consider functionals  $\mathcal{F}\colon GSBV^{p(\cdot)}(\Omega;\R^m) \times \mathcal{B}(\Omega) \to  [0,+\infty)$ with the following general assumptions: 
\begin{enumerate}[font={\normalfont},label={(${\rm H}_\arabic*$)}]
\item  $\mathcal{F}(u,\cdot)$ is a Borel measure for any $u \in  GSBV^{p(\cdot)}(\Omega;\R^m)$; \label{assH1} 
\item  $\mathcal{F}(\cdot,A)$ is lower semicontinuous with respect to convergence in measure on $\Omega$ for any $A \in \mathcal{A}(\Omega)$; \label{assH2}
\item   $\mathcal{F}(\cdot, A)$ is local for any $A \in \mathcal{A}(\Omega)$, in the sense that if $u,v \in GSBV^{p(\cdot)}(\Omega;\R^m)$ satisfy $u=v$ a.e.\ in $A$, then $\mathcal{F}(u,A) = \mathcal{F}(v,A)$; \label{assH3}
\item  there exist $0 < \alpha  < \beta $ such that for any $u \in GSBV^{p(\cdot)}(\Omega;\R^m)$ and $B \in \mathcal{B}(\Omega)$  we have
$$\alpha \bigg(\int_{ B  } |\nabla u|^{p(x)}  \dx    +   \mathcal{H}^{d-1}(J_u \cap B)\bigg) \le \mathcal{F}(u,B) \le \beta \bigg(\int_{ B } (1 + |\nabla u|^{p(x)})    \dx  +   \mathcal{H}^{d-1}(J_u \cap B)\bigg).$$  \label{assH4}
\end{enumerate}

For every $u \in GSBV^{p(\cdot)}(\Omega;\R^m)$ and $A \in \mathcal{A}(\Omega)$ we define
\begin{equation}\label{eq: general minimizationsgbv}
\mathbf{m}_{\mathcal{F}}(u,A) = \inf_{v \in GSBV^{p(\cdot)}(\Omega;\R^m)} \  \lbrace \mathcal{F}(v,A): \ v = u \ \text{ in a neighborhood of } \partial A \rbrace\,.
\end{equation}
Moreover, for $x_0 \in \Omega$, $u_0 \in \R^m$, and $\xi \in  \mathbb{R}^{m \times d}  $ we introduce the affine functions  $\ell_{x_0,u_0,\xi}\colon \R^d \to \R^m$ by 
\begin{align}\label{eq: elastic competitor}
\ell_{x_0,u_0,\xi}(x) =  u_0 + \xi (x-x_0)\,,
\end{align}
and, for $a,b \in \R^m$, $\nu \in \mathbb{S}^{d-1}$ we define  $u_{x_0,a,b,\nu} \colon \R^d \to \R^m$ by 
\begin{align}\label{eq: jump competitor}
u_{x_0,a,b,\nu}(x) = \begin{cases}  a & \text{if } (x-x_0) \cdot \nu > 0,\\ b & \text{if }  (x-x_0) \cdot \nu < 0. \end{cases} 
\end{align}

The main result of this section is the following integral representation theorem.

\begin{theorem}[Integral representation in $GSBV^{p(\cdot)}$]\label{thm: int-representation-gsbv}
Let $\Omega \subset \R^d$ be open, bounded with Lipschitz boundary, let $m \in \N$. Let $p:\Omega\to(1,+\infty)$ be a variable exponent complying with {\rm\ref{assP1}}-{\rm\ref{assP2}}, and suppose that  $\mathcal{F}\colon GSBV^{p(\cdot)}(\Omega;\R^m)  \times \mathcal{B}(\Omega) \to [0,+\infty)$ satisfies {\rm\ref{assH1}}--{\rm\ref{assH4}}. Then 
$$\mathcal{F}(u,B) = \int_B f\big(x,u(x),\nabla u(x)\big)  \, {\rm d}x +    \int_{J_u\cap  B} g\big(x,u^+(x),u^-(x),\nu_u(x)\big)\,  {\rm d}  \mathcal{H}^{d-1}(x)$$
for all $u \in  GSBV^{p(\cdot)}(\Omega;\R^m)$  and   $B \in \mathcal{B}(\Omega)$, where $f$ is given  by
\begin{align}\label{eq:fdef-gsbv}
f(x_0,u_0,\xi) = \limsup_{\eps \to 0} \frac{\mathbf{m}_{\mathcal{F}}(\ell_{x_0,u_0,\xi},B_\eps(x_0))}{\gamma_d\eps^{d}}
\end{align}
for all $x_0 \in \Omega$, $u_0 \in \R^m$, $\xi \in \mathbb{R}^{m \times d}$,  and $g$ is given by 
\begin{align}\label{eq:gdef-gsbv}
g(x_0,a,b,\nu) = \limsup_{\eps \to 0} \frac{\mathbf{m}_{\mathcal{F}}(u_{x_0,a,b,\nu},B_\eps(x_0))}{\gamma_{d-1}\eps^{d-1}}
\end{align}
for all $ x_0  \in \Omega$,  $a,b \in \R^m$, and $\nu \in \mathbb{S}^{d-1}$.
\end{theorem}

\subsection{Fundamental estimate}\label{sec:fundamental}

In this section we prove an important tool in the proof of the integral representation,  namely  a fundamental estimate in $GSBV^{p(\cdot)}$ for functionals $\mathcal{F}$.

 \begin{lemma}[Fundamental estimate in $GSBV^{p(\cdot)}$]\label{lemma: fundamental estimate}
Let $\Omega \subset \R^d$ be open and bounded, and let $p:\Omega\to(1,+\infty)$ be a variable exponent in $\Omega$ satisfying {\rm\ref{assP1}}.  Let $\eta >0$ and  let $D', D'', E \in \mathcal{A}(\Omega)$ with $D' \subset \subset  D''$, and set $\delta:=\frac{1}{2}{\rm dist}(D',\partial D'')$. For every functional $\mathcal{F}$  satisfying {\rm\ref{assH1}}, {\rm\ref{assH3}}, and {\rm\ref{assH4}} and for every $u \in GSBV^{p(\cdot)}(D';\R^m)$, $v \in GSBV^{p(\cdot)}(E;\R^m)$ there exists a function $\varphi \in C^\infty(\R^d;[0,1])$  such that  $w :=  \varphi u + (1- \varphi)v \in GSBV^{p(\cdot)}(D'\cup E;\R^m)$ satisfies
\begin{align}\label{eq: assertionfund}
{\rm (i)}& \ \ \mathcal{F} ( w, D' \cup E) \le  (1+ \eta)\big(\mathcal{F}(u,D'')  + \mathcal{F}(v, E) \big) + M \int_{F}\left(\frac{|u-v|}{\delta}\right)^{p(x)}\,\mathrm{d}x +\eta\mathcal{L}^d(D' \cup E), \notag \\ 
{\rm (ii)} & \ \  w = u \text{ on } D' \text{ and } w = v \text{ on } E \setminus D'',
\end{align}
where $F:= (D'' \setminus D') \cap E$ and $M=M(D',D'',E,p^+,\eta)>0$ depends only on $D',D'',E,p^+,\eta$, but is independent of $u$ and $v$. Moreover, if for $\eps>0$  and $x_0 \in \R^d$ we have  $D'_{\eps,x_0}, D''_{\eps,x_0}, E_{\eps,x_0} \subset \Omega$,  then 
\begin{align}\label{eq: constant-scal}
 M(D'_{\eps,x_0}, D''_{\eps,x_0},E_{\eps,x_0},p^+,\eta)  = M(D',D'',E,p^+,\eta),
\end{align} 
and the remainder term is
\begin{equation*}
M \int_{F_{\eps,x_0}}\left(\frac{|u-v|}{\delta\varepsilon}\right)^{p(x)}\,\mathrm{d}x\,,
\end{equation*}
where we used the notation introduced in \eqref{eq: shift-not}. 
\end{lemma} 
 
\begin{proof}
We choose $k \in \N$ such that
\begin{align}\label{eq: kdef}
k \ge \max\Big\{\frac{ 3^{p^+-1}\beta  }{\eta \alpha}, \frac{\beta}{\eta}   \Big\}\,,
\end{align}
and for $i=1,\ldots,k$, we set
\begin{equation*}
D_{i+1}:=\left\{x\in D'' :\,\, {\rm dist}(x,D')<\frac{\delta i}{k}\right\}\,.
\end{equation*}
We then have $D_1:=D' \subset \subset D_2 \subset \subset \ldots \subset \subset D_{k+1} \subset \subset D''$. Correspondingly, let $\varphi_i\in C_0^\infty(D_{i+1})$ with $0\leq\varphi_i\leq1$ and $\varphi_i=1$ in a neighborhood $U_i$ of $\overline{D_i}$. Note that $\|\nabla \varphi_i\|_\infty\leq\frac{2k}{\delta}$.

Let $u \in GSBV^{p(\cdot)}(D'';\R^m)$ and  $v \in GSBV^{p(\cdot)}(E;\R^m)$ be such that $u-v \in L^{p(\cdot)}((D'' \setminus D') \cap E;\R^m)$, as otherwise the result is trivial.  We define the function $w_i = \varphi_i u + (1-\varphi_i)v \in GSBV^{p(\cdot)}(D' \cup E;\R^m)$ (this can be easily proved as in \cite[Lemma~2.11]{CM}), where $u$ and $v$ are extended arbitrarily outside $D''$ and $E$, respectively.  Letting $I_i = D'' \cap (D_{i+1} \setminus \overline{D_i})$ we get by {\rm\ref{assH1}}  and {\rm\ref{assH3}} 
\begin{align}\label{eq: wi}
\mathcal{F}(w_i,D' \cup E) &\le \mathcal{F}(u, (D' \cup E) \cap U_i) + \mathcal{F}(v, E \setminus {\rm supp}\,\varphi_i)
+  \mathcal{F}(w_i,I_i) \notag \\
&\le \mathcal{F}(u,D'') + \mathcal{F}(v,E) + \mathcal{F}(w_i,I_i). 
\end{align} 
 For the last term, we compute  using  {\rm\ref{assH4}}   
\begin{align*}
& \mathcal{F}(w_i,I_i)  \le  \beta \int_{I_i} (1+|\nabla w_i|^{p(x)})\, {\rm d}x + \beta \mathcal{H}^{d-1}(J_{w_i} \cap J_i) \\
& \le \beta  \int_{I_i} (1+|\varphi_i \nabla u + (1-\varphi_i)\nabla v + \nabla \varphi_i (u-v) |^{p(x)})\,\mathrm{d}x + \beta \mathcal{H}^{d-1}( (J_{u} \cup J_v) \cap I_i) \\
& \le \beta\mathcal{L}^d(I_i) + 3^{p^+-1}\beta  \int_{I_i} \big(|\nabla u|^{p(x)} + |\nabla v|^{p(x)} + |\nabla \varphi_i|^{p(x)} |u-v|^{p(x)} \big)\,\mathrm{d}x \\
& + \beta \mathcal{H}^{d-1}( J_{u}  \cap I_i) + \beta \mathcal{H}^{d-1}( J_v \cap I_i)  \\
& \le  3^{p^+-1}  \frac{\beta}{\alpha} \big(  \mathcal{F}(u,I_i) +  \mathcal{F}(v,I_i)\big) + (2k)^{p^+}\cdot 3^{p^+-1}\beta \int_{I_i}\left(\frac{|u-v|}{\delta}\right)^{p(x)}\,\mathrm{d}x+ \beta\mathcal{L}^d(I_i).
\end{align*}
Consequently, recalling \eqref{eq: kdef}  and using {\rm\ref{assH1}}  we find $i_0 \in \lbrace 1, \ldots, k \rbrace$ such that
\begin{equation*}
\begin{split}
\mathcal{F}(w_{i_0},I_{i_0}) & \le \frac{1}{k}\sum_{i=1}^k  \mathcal{F}(w_{i},I_i) \\
& \le \eta \big(  \mathcal{F}(u,D'') +  \mathcal{F}(v,E)\big) + M \int_{F}\left(\frac{|u-v|}{\delta}\right)^{p(x)}\,\mathrm{d}x + \eta \mathcal{L}^d(F), 
\end{split}
\end{equation*}
where $M :=  (2k)^{p^+}\cdot 3^{p^+-1}  \beta k^{-1}$. This along with \eqref{eq: wi} concludes the proof of \eqref{eq: assertionfund}  by setting $w = w_{i_0}$.  To see the scaling property \eqref{eq: constant-scal}, it suffices to use the cut-off functions 
$\varphi^\eps_i\in C_0^\infty((D_{i+1})_{\eps,x_0};[0,1])$ $i=1,\ldots,k$, defined by $\varphi_i^\eps (x) =  \varphi_i( x_0 +  \frac{(x-x_0)}{\eps}) $ for $x \in (D_{i+1})_{\eps,x_0}$. This concludes the proof.   
 \end{proof}

\subsection{The global method}\label{sec: global method}

This section is entirely devoted to the proof of Theorem \ref{thm: int-representation-gsbv}. 
As a first step, we show that $\mathcal{F}$ and $\mathbf{m}_{\mathcal{F}}$, defined by \eqref{eq: general minimizationsgbv}, 
have the same  Radon-Nikodym  derivative with respect to the measure
\begin{equation}
\mu:= \mathcal{L}^d\lfloor_{\Omega}\,\, +\,\, \mathcal{H}^{d-1}\lfloor_{J_u \cap \Omega}\,. 
\label{eq: measuremu}
\end{equation}

\begin{lemma}\label{lemma: F=m}
Let $p:\Omega\to(1,+\infty)$ be a variable exponent satisfying {\rm\ref{assP1}}.   
 Suppose that $\mathcal{F}$ satisfies {\rm\ref{assH1}}--{\rm\ref{assH4}}. Let $u \in GSBV^{p(\cdot)}(\Omega;\R^m)$ and  $\mu$ as in \eqref{eq: measuremu}. 
Then for $\mu$-a.e.\ $x_0 \in \Omega$ we have
 $$\lim_{\eps \to 0}\frac{\mathcal{F}(u,B_\eps(x_0))}{\mu(B_\eps(x_0))} =  \lim_{\eps \to 0}\frac{\mathbf{m}_{\mathcal{F}}(u,B_\eps(x_0))}{\mu(B_\eps(x_0))}.$$
\end{lemma}

The proof of this lemma is postponed to the end of this section. 
The second step in the proof of Theorem \ref{thm: int-representation-gsbv} is that, asymptotically as $\eps \to 0$, the minimization problems $\mathbf{m}_{\mathcal{F}}(u,B_\eps(x_0))$ and $\mathbf{m}_{\mathcal{F}}(\bar{u}^{\rm bulk}_{x_0},B_\eps(x_0))$ coincide for $\mathcal{L}^{d}$-a.e.\ $x_0 \in \Omega$, where we write $\bar{u}^{\rm bulk}_{x_0} :=  \ell_{x_0,u(x_0),\nabla u(x_0)}  $ for brevity, see \eqref{eq: elastic competitor}.

\begin{lemma}\label{lemma: sameminbulk}
Let $p:\Omega\to(1,+\infty)$ be a Riemann-integrable variable exponent satisfying {\rm\ref{assP1}}. Suppose that $\mathcal{F}$ satisfies {\rm\ref{assH1}} and  {\rm\ref{assH3}}--{\rm\ref{assH4}}  and let $u \in GSBV^{p(\cdot)}(\Omega;\R^m)$.  Then for $\mathcal{L}^{d}$-a.e.\ $x_0 \in \Omega$  we have
\begin{align}\label{eq: samemin-bulk}
  \lim_{\eps \to 0}\frac{\mathbf{m}_{\mathcal{F}}(u,B_\eps(x_0))}{\gamma_{d}\eps^{d}} =  \limsup_{\eps \to 0}\frac{\mathbf{m}_{\mathcal{F}}(\bar{u}^{\rm bulk}_{x_0},B_\eps(x_0))}{\gamma_{d}\eps^{d}}. 
 \end{align}
\end{lemma}

The final step is that, asymptotically as $\eps \to 0$, the minimization problems $\mathbf{m}_{\mathcal{F}}(u,B_\eps(x_0))$ and $\mathbf{m}_{\mathcal{F}}(\bar{u}^{\rm surf}_{x_0},B_\eps(x_0))$ coincide for $\mathcal{H}^{d-1}$-a.e.\ $x_0 \in J_u$, where we write $\bar{u}^{\rm surf}_{x_0} := u_{x_0,u^+(x_0),u^-(x_0),\nu_u(x_0)}$ for brevity, see \eqref{eq: jump competitor}.

\begin{lemma}\label{lemma: sameminsurf}
Let $p:\Omega\to(1,+\infty)$ be a variable exponent satisfying {\rm\ref{assP1}}-{\rm\ref{assP2}}. Suppose that $\mathcal{F}$ satisfies {\rm\ref{assH1}} and  {\rm\ref{assH3}}--{\rm\ref{assH4}}  and let $u \in GSBV^{p(\cdot)}(\Omega;\R^m)$.  Then for $\mathcal{H}^{d-1}$-a.e.\ $x_0 \in J_u$  we have
\begin{align}\label{eq: samemin-surf}
  \lim_{\eps \to 0}\frac{\mathbf{m}_{\mathcal{F}}(u,B_\eps(x_0))}{\gamma_{d-1}\eps^{d-1}} =  \limsup_{\eps \to 0}\frac{\mathbf{m}_{\mathcal{F}}(\bar{u}^{\rm surf}_{x_0},B_\eps(x_0))}{\gamma_{d-1}\eps^{d-1}}. 
  \end{align}
\end{lemma}

We defer the proof of Lemma~\ref{lemma: sameminbulk} and Lemma~\ref{lemma: sameminsurf} to Section~\ref{sec:bulk} and Section~\ref{sec:surf}, respectively.  Now, we proceed  to prove Theorem~\ref{thm: int-representation-gsbv}.

\begin{proof}[Proof of Theorem \ref{thm: int-representation-gsbv}]
In view of the assumption {\rm\ref{assH4}} on $\mathcal{F}$ and of the Besicovitch derivation theorem (cf.\ \cite[Theorem~2.22]{Ambrosio-Fusco-Pallara:2000}), we need to show that 
\begin{align}
\frac{\mathrm{d}\mathcal{F}(u,\cdot)}{\mathrm{d}\mathcal{L}^{d}}(x_0) & = f\big(x_0,u(x_0),\nabla u(x_0)\big), \quad \mbox{ for $\mathcal{L}^{d}$-a.e.\ $x_0 \in \Omega$,} \label{eq: to show1}\\
\frac{\mathrm{d}\mathcal{F}(u,\cdot)}{\mathrm{d}\mathcal{H}^{d-1}\lfloor_{J_u}}(x_0) & =  g\big(x_0,u^+(x_0),u^-(x_0),\nu_u(x_0)\big), \quad \mbox{ for $\mathcal{H}^{d-1}$-a.e.\ $x_0 \in J_u$, }\label{eq: to show2}
\end{align}
where $f$ and $g$ were defined in \eqref{eq:fdef-gsbv} and \eqref{eq:gdef-gsbv}, respectively. 

By Lemma \ref{lemma: F=m} and the fact that  $\lim_{\eps \to 0} (\gamma_{d}\eps^{d})^{-1}\mu(B_\eps(x_0))=1$ for $\mathcal{L}^{d}$-a.e.\ $x_0 \in \Omega$  we deduce 
 $$\frac{\mathrm{d}\mathcal{F}(u,\cdot)}{\mathrm{d}\mathcal{L}^{d}}(x_0)  = \lim_{\eps \to 0}\frac{\mathcal{F}(u,B_\eps(x_0))}{\mu(B_\eps(x_0))} =  \lim_{\eps \to 0}\frac{\mathbf{m}_{\mathcal{F}}(u,B_\eps(x_0))}{\mu(B_\eps(x_0))}  = \lim_{\eps \to 0}\frac{\mathbf{m}_{\mathcal{F}}(u,B_\eps(x_0))}{\gamma_{d}\eps^{d}} < \infty$$
 for $\mathcal{L}^{d}$-a.e.\ $x_0 \in \Omega$. Then, \eqref{eq: to show1} follows from \eqref{eq:fdef-gsbv} and Lemma \ref{lemma: sameminbulk}.    By Lemma \ref{lemma: F=m} and the fact that  $\lim_{\eps \to 0} (\gamma_{d-1}\eps^{d-1})^{-1}\mu(B_\eps(x_0))=1$ for $\mathcal{H}^{d-1}$-a.e.\ $x_0 \in J_u$  we deduce 
 $$\frac{\mathrm{d}\mathcal{F}(u,\cdot)}{\mathrm{d}\mathcal{H}^{d-1}\lfloor_{J_u}}(x_0)  = \lim_{\eps \to 0}\frac{\mathcal{F}(u,B_\eps(x_0))}{\mu(B_\eps(x_0))} =  \lim_{\eps \to 0}\frac{\mathbf{m}_{\mathcal{F}}(u,B_\eps(x_0))}{\mu(B_\eps(x_0))}  = \lim_{\eps \to 0}\frac{\mathbf{m}_{\mathcal{F}}(u,B_\eps(x_0))}{\gamma_{d-1}\eps^{d-1}} < \infty$$
 for $\mathcal{H}^{d-1}$-a.e.\ $x_0 \in J_u$. Now, \eqref{eq: to show2} follows   from \eqref{eq:gdef-gsbv} and  Lemma \ref{lemma: sameminsurf}.
 \end{proof} 
 \EEE

In the remaining part of the section we prove Lemma~\ref{lemma: F=m}. For this, we need to fix some notation. For $\delta>0$ and $A \in \mathcal{A}(\Omega)$,  we define 
\begin{align}
\mathbf{m}^\delta_{\mathcal{F}}(u,A) = \inf\Big\{ \sum\nolimits_{i=1}^\infty \mathbf{m}_{\mathcal{F}}(u,B_i)\colon & \ B_i \subset A   \text{ pairwise disjoint balls}, \,  \diam(B_i) \le \delta,\notag\\ 
 & \quad   \quad  \quad  \quad  \quad \quad  \quad  \quad \quad \quad \quad \quad \mu\Big(  A \setminus \bigcup\nolimits_{i=1}^\infty B_i\Big) = 0\Big\}\,, \notag
\end{align}
where $\mu$ is defined in \eqref{eq: measuremu}. 
As $\mathbf{m}^\delta_{\mathcal{F}}(u,A) $ is decreasing in $\delta$, we can also introduce
\begin{align}\label{eq: m*=F}
\mathbf{m}^*_{\mathcal{F}}(u,A) :=   \lim_{\delta \to 0 }  \mathbf{m}^\delta_{\mathcal{F}}(u,A).
\end{align}
\EEE

\BBB  In the following lemma, we prove that $\mathcal{F}$ and $\mathbf{m}^*_{\mathcal{F}}$ coincide under our assumptions.   \EEE

\begin{lemma}\label{lemma: F=m*}
Let $p:\Omega\to(1,+\infty)$ be complying with {\rm\ref{assP1}}. {Suppose that $\mathcal{F}$ satisfies {\rm (${\rm H_1}$)}--{\rm (${\rm H_4}$)} and  let  $u \in GSBV^{p(\cdot)}(\Omega;\R^m)$. Then, for all $A \in \mathcal{A}(\Omega)$ there holds $\mathcal{F}(u,A) = \mathbf{m}^*_{\mathcal{F}}(u,A)$. }
\end{lemma}

\begin{proof}
We can follow the argument of \cite[Lemma~4]{BFLM} (see also \cite[Lemma~3.3]{BFM}). We start by proving the inequality $\mathbf{m}_{\mathcal{F}}^*(u,A) \le \mathcal{F}(u,A)$. For each ball $B \subset A$ we have $\mathbf{m}_{\mathcal{F}}(u,B) \le \mathcal{F}(u,B)$ by definition. By {\rm\ref{assH1}}  we  get $\mathbf{m}_{\mathcal{F}}^\delta(u,A) \le \mathcal{F}(u,A)$ for all $\delta>0$, whence the assertion follows taking into account \eqref{eq: m*=F}.

We now address the reverse inequality.  We fix $A\in\mathcal{A}(\Omega)$ and  $\delta >0$. Let $(B^\delta_i)_i$ be balls as in the definition of $\mathbf{m}_{\mathcal{F}}^\delta(u,A)$ such that
\begin{align}\label{eq: to show-estimate1}
\sum\nolimits_{i=1}^\infty \mathbf{m}_{\mathcal{F}}(u,B^\delta_i) \le \mathbf{m}_{\mathcal{F}}^\delta(u,A) + \delta.
\end{align}
By the definition of $\mathbf{m}_{\mathcal{F}}$, we find $v_i^\delta \in GSBV^{p(\cdot)}(B_i^\delta;\R^m)$ such that $v_i^\delta = u$ in a neighborhood of $\partial B_i^\delta$ and 
\begin{align}\label{eq: to show-estimate2}
\mathcal{F}(v_i^\delta,B_i^\delta) \le \mathbf{m}_{\mathcal{F}}(u,B_i^\delta) + \delta \mathcal{L}^d(B_i^\delta).
\end{align}
We define 
\begin{align}\label{eq: def of vdelta}
 v^{\delta,n}  := \sum\nolimits_{i=1}^n v^\delta_i \chi_{B^\delta_i} +  u \chi_{N_0^{\delta,n}}  \quad \text{for $n\in \N$},\quad \quad \quad v^\delta := \sum\nolimits_{i=1}^\infty v^\delta_i \chi_{B^\delta_i} + u \chi_{N_0^\delta}, 
\end{align}
 where  $N_0^{\delta,n} := \Omega \setminus \bigcup_{i=1}^n B^\delta_i$   and $N_0^\delta := \Omega \setminus \bigcup_{i=1}^\infty B^\delta_i$.  
By construction, we have that  each $ v^{\delta,n} $ lies in $GSBV^{p(\cdot)}(\Omega;\R^m)$ and that 
\begin{equation}
 \sup_{n \in \N} \left(\int_\Omega |\nabla v^{\delta,n}(x)|^{p(x)}\,\mathrm{d}x  + \mathcal{H}^{d-1}(J_{ v^{\delta,n} })\right) <+\infty
\label{eq: equibddvdelta}
\end{equation}
by \eqref{eq: to show-estimate1}--\eqref{eq: to show-estimate2} and {\rm\ref{assH4}}. Moreover, $ v^{\delta,n}  \to v^\delta$ pointwise a.e.\ in $\Omega$, and then in measure on $\Omega$. Then,  
\cite[Theorem~2.2]{Amb} combined with the compactness in $L^0$ of $(v^{\delta,n})$ yields $v^\delta \in GSBV^{p^{-}}(\Omega;\R^m)$ and $\nabla v^{\delta,n}\rightharpoonup \nabla v^{\delta}$ weakly in $L^{p^-}(\Omega;\mathbb{R}^{m\times d})$. Now, by Ioffe's Theorem (see \cite{Ioffe}) and \eqref{eq: equibddvdelta} we get
\begin{equation*}
\int_\Omega |\nabla v^{\delta}(x)|^{p(x)}\,\mathrm{d}x \leq \displaystyle\mathop{\lim\inf}_{n\to+\infty} \int_\Omega |\nabla v^{\delta,n}(x)|^{p(x)}\,\mathrm{d}x <+\infty\,,
\end{equation*}
whence $v^\delta \in GSBV^{p{(\cdot)}}(\Omega;\R^m)$. We have

\begin{align}\label{eq: to show-flaviana2}
\mathcal{F}(v^\delta,A) &= \sum\nolimits_{i=1}^\infty\mathcal{F}(v_i^\delta,B_i^\delta)  + \mathcal{F}(u,N_0^\delta \cap A) \le  \sum\nolimits_{i=1}^\infty \big(\mathbf{m}_{\mathcal{F}}(u,B_i^\delta) + \delta \mathcal{L}^d(B_i^\delta)\big) \notag\\
&\le \mathbf{m}_{\mathcal{F}}^\delta(u,A) + \delta(1+\mathcal{L}^d(A)),    
\end{align}
where we also used  the fact that $\mu(N_0^\delta \cap A)  =  \mathcal{F}(u,N_0^\delta \cap A) = 0$ by the definition of $(B^\delta_i)_i$  and {\rm\ref{assH4}}. For later purpose, we also note by {\rm\ref{assH4}} that this implies  
\begin{align}\label{eq: to show-flaviana2-NNN}
\Vert \nabla v^\delta \Vert_{L^{p^-}(A;\mathbb{R}^{m\times d})}  + \mathcal{H}^{d-1}(J_{v^\delta} \cap A) \le c\alpha^{-1} \big(\mathbf{m}_{\mathcal{F}}^\delta(u,A) + \delta(1+\mathcal{L}^d(A)) \big).
\end{align}
We now  claim that 
\begin{equation}
w^\delta:= u-v^\delta \to 0 \quad \mbox{ in measure on $A$.}
\label{eq:convinmeasure} 
\end{equation}
With this, using {\rm\ref{assH2}},  \eqref{eq: m*=F},  and  \eqref{eq: to show-flaviana2} we will get the required inequality  $\mathbf{m}_{\mathcal{F}}^*(u,A) \ge \mathcal{F}(u,A)$ in the limit as $\delta \to 0$.  To prove \eqref{eq:convinmeasure}, we first note that $w^\delta\lfloor_{B^\delta_i} \in GSBV^{p^-}(B^\delta_i;\R^m)$ has trace zero on $\partial B^\delta_i$. 
Then, setting for every $M>0$
\begin{equation*}
w^{\delta,M}:= (-M \vee w^\delta)\wedge M\,,
\end{equation*}
from the classical Poincar\'e inequality we get
\begin{equation*}
\begin{split}
\|w^{\delta,M}\|_{L^1(B^\delta_i;\R^m)}&\leq C\delta |Dw^{\delta,M}|(B^\delta_i)\,, 
\end{split}
\end{equation*}
whence
\begin{equation*}
\|w^{\delta,M}\|_{L^1(A;\R^m)} \leq C\delta |Dw^{\delta,M}|(\cup_{i=1}^\infty B^\delta_i) \leq  C\delta |Dw^{\delta,M}|(A)\,,
\end{equation*}
where $|Dw^{\delta,M}|(A)$ is bounded in view of \eqref{eq: to show-flaviana2-NNN} and the fact that $u\in GSBV(A;\R^m)$, since 
\begin{equation*}
|Dw^{\delta,M}|(A) \leq \int_A|\nabla v^\delta|\,\mathrm{d}x + \int_A|\nabla u|\,\mathrm{d}x + 2M \left(\mathcal{H}^{d-1}(J_{v^\delta}\cap A)+\mathcal{H}^{d-1}(J_{u}\cap A)\right)<+\infty\,.
\end{equation*}
This implies $w^{\delta,M}\to0$ in $L^1(A;\R^m)$, and then in measure on $A$, as $\delta\to0$ for every $M>0$. Now, with fixed $M=1$ and $\varepsilon\in(0,1)$ we have
\begin{equation*}
E_\eps^\delta:=\{x\in A:\,\,|w^\delta(x)|>\eps\} \subseteq \{x\in A:\,\,|w^{\delta,1}(x)|>\varepsilon\}\,,
\end{equation*}
whence $\mathcal{L}^d(E_\eps^\delta)\to0$ as $\delta\to0$, thus proving \eqref{eq:convinmeasure}. The proof is concluded. \EEE
\end{proof}

\begin{proof}[Proof of Lemma \ref{lemma: F=m}.] {We may follow the same argument as in \cite[Proofs of Lemma~5 and Lemma~6]{BFLM}, by exploiting also Lemma~\ref{lemma: F=m*}. We then omit the details.}
\end{proof}

 To conclude the proof of Theorem \ref{thm: int-representation-gsbv}, it remains to prove Lemmas \ref{lemma: sameminbulk} and \ref{lemma: sameminsurf}. This is the subject of the following two sections.

\subsection{The bulk density}\label{sec:bulk}

 This section is devoted to the proof of  Lemma~\ref{lemma: sameminbulk}. With the following lemma, we analyze the blow-up at points with approximate  gradient, which exists for  $\mathcal{L}^d$-a.e.\   point in $\Omega$ by Lemma \ref{lemma: approx-grad}. It is noteworthy that in order to develop the blow-up arguments of this section, it will suffice to consider a Riemann integrable exponent $p$ satisfying \ref{assP1}, as $\mathcal{L}^d$-a.e. $x\in\Omega$ is a continuity point for $p$. On the contrary, the stronger assumption \ref{assP2} will be crucial in Section~\ref{sec:surf} when dealing with the surface scaling.

\begin{lemma}\label{lem:lemma2fonseca}
Let $p:\Omega\to(1,+\infty)$ be a Riemann integrable variable exponent complying with \ref{assP1}. Let $u \in GSBV^{p(\cdot)}(\Omega;\R^m)$. Then for $\mathcal{L}^{d}$-a.e.\ $x_0 \in \Omega$ and $\mathcal{L}^{1}$-a.e.\ $\sigma \in (0,1)$ there exists a sequence $u_\eps\in GSBV^{p(\cdot)}(B_\eps(x_0);\R^m)$ such that
\begin{equation}
\begin{aligned}
& (i)\,\, u_\eps=u \,\, \mbox{in $B_\eps(x_0)\backslash\overline{B_{\sigma\eps}(x_0)}$,}\quad\displaystyle\lim_{\eps\to0}\,{\eps^{-(d+1)}} \mathcal{L}^d(\{u_\eps\neq u\}\cap B_{\eps}(x_0))=0\,; \\
& (ii)\,\, \displaystyle \lim_{\eps \to 0}  \ \eps^{-d} \int_{B_{\sigma\eps}(x_0)} \left(\frac{|u_\eps(x) - u(x_0)-\nabla u(x_0)(x-x_0)|}{\eps} \right)^{p(x)} \, \mathrm{d}x = 0\,;\\
& (iii)\,\, \displaystyle\lim_{\eps\to0} \eps^{-d} \mathcal{H}^{d-1}(J_{u_\eps})=0\,.
\end{aligned}
\label{eq:(16)}
\end{equation}
If, in addition, $u \in SBV^{p(\cdot)}(\Omega;\R^m)$, then $u_\eps$ also satisfies
\begin{equation*}
\begin{aligned}
&(i)'\,\, \lim_{\eps\to0}\eps^{-(d+1)}\int_{B_\eps(x_0)}|u_\eps-u|\,\mathrm{d}x=0\,, \\
&(iii)'\,\, \lim_{\eps\to0}\eps^{-d}\int_{J_{u_\eps}}|[u_\eps]|\,\mathrm{d}\mathcal{H}^{d-1}=0\,.
\end{aligned}
\end{equation*}
\end{lemma}
\proof
It will suffice to treat the scalar case $m=1$. Let $x_0\in\Omega$ be such that 
\begin{align}\label{eq:givenpropertiesbis}
{(a)} & \ \  \lim_{\eps \to 0} \ \eps^{-d} \int_{B_{\eps}(x_0)} \big|\nabla u(x) - \nabla u(x_0)\big|^{p(x)}  \, \mathrm{d}x = 0\,;\notag\\
{(b)} & \ \   \lim_{\eps \to 0} \  \eps^{-d}\,\mathcal{H}^{d-1}(J_{u} \cap B_\eps(x_0))= 0\,; \\
{(c)} & \  \ \lim_{\eps \to 0} \  \eps^{-d} \mathcal{L}^d\Big(\Big\{x \in B_\eps(x_0) \colon \,  \frac{|u(x) - u(x_0) - \nabla u(x_0)(x-x_0)|}{\eps} > \varrho   \Big\} \Big)  = 0 \text{ for all $\varrho >0$}     
\notag\,.
\end{align}
Properties $(a)$ and $(c)$ hold for $\mathcal{L}^d$-a.e.\ $x_0 \in \Omega$ by Theorem~\ref{thm:lebpoint} since $|\nabla u| \in L^{p(\cdot)}(\Omega;\mathbb{R}^{m\times d})$ and by Lemma \ref{lemma: approx-grad}, respectively, while $(b)$ follows from the fact that $J_u$ is countably $\hd$-rectifiable (see, e.g., \cite{Ambrosio-Fusco-Pallara:2000}). We can also assume that $x_0$ is a continuity point for $p(x)$; hence, it is not restrictive to assume that \eqref{eq:assump} holds, up to  replacing $\Omega$ with a fixed neighborhood of $x_0$ where it is satisfied.

We set $\bar{u}_\eps(x):=\frac{u(x)-u(x_0)}{\eps}$, define the truncated functions $T_\eps \bar{u}_\eps:= T_{B_{\eps}(x_0)}\bar{u}_\eps$ as in \eqref{eq:truncated}, and $v_\eps(x) := u(x_0)+\eps T_\eps \bar{u}_\eps(x)$. 

Note that
\begin{equation}
|\nabla v_\eps|\leq |\nabla u| \quad \mbox{ $\mathcal{L}^d$-a.e.}
\label{eq:gradineq}
\end{equation}
and $J_{v_\eps}\subseteq J_u$, $\mathcal{H}^{d-1}(J_{v_\eps}\backslash J_u)=0$. This along with \eqref{eq:givenpropertiesbis}$(b)$ implies \eqref{eq:(16)}$(iii)$.

We notice that \eqref{eq:givenpropertiesbis}$(b)$ implies also \eqref{(10)} for $\eps$ small enough, which combined with \eqref{(12bis)} gives
\begin{equation}
\begin{split}
0\leq \eps^{-d} \int_0^1 \mathcal{H}^{d-1}(\{v_\eps\neq u\}\cap \partial B_{\sigma\eps}(x_0))\,\mathrm{d}\sigma & = \frac{2}{\eps^{d+1}} \mathcal{L}^d(\{v_\eps\neq u\}\cap B_{\eps}(x_0)) \\
 & \leq \frac{4 (2\gamma_{\rm iso})^{\frac{d}{d-1}}}{\eps^{d+1}}\left(\mathcal{H}^{d-1}(J_{u}\cap B_\eps(x_0))\right)^{\frac{d}{d-1}} \\
 & \leq 4 (2\gamma_{\rm iso}\gamma_d)^{\frac{d}{d-1}}C^{\frac{d}{d-1}}\eps^\frac{1}{d-1} \to 0
\end{split}
\label{eq:contocoarea}
\end{equation} 
as $\eps\to0$.

Therefore, for every sequence $\eps\to0$ one can find a subsequence (not relabeled) such that, for $\mathcal{L}^1$-a.e. $\sigma\in(0,1)$,
\begin{equation}
\begin{split}
\mu(\partial B_{\sigma\eps}(x_0)) = \mathcal{H}^{d-1}(\partial B_{\sigma\eps}(x_0)\cap J_{v_\eps})=0\,,  \\
\lim_{\eps\to0} \eps^{-d} \mathcal{H}^{d-1}(\{v_\eps\neq u\}\cap \partial B_{\sigma\eps}(x_0)) = 0\,. 
\end{split}
\label{eq:43}
\end{equation}
Now, we fix a sequence $\eps\to0$ and consider a subsequence (not relabeled) and $\sigma\in(0,1)$ for which \eqref{eq:43} holds. We then define
\begin{equation*}
u_\eps(x) =
\begin{cases}
v_\eps(x) & \mbox{ in } B_{\sigma\eps}(x_0)\,,\\
u(x) & \mbox{ in } B_\eps(x_0)\backslash \overline{B_{\sigma\eps}(x_0)}\,.
\end{cases}
\end{equation*}
From the definition of $u_\eps$ and the argument of \eqref{eq:contocoarea} we get the assertions in \eqref{eq:(16)}$(i)$.
We now prove \eqref{eq:(16)}$(ii)$. We set $\widetilde{u}_\eps(y):=\bar{u}_\eps(x_0+\eps y)$. Then, for $s\in[0,\eps^d]$ we have $(\bar{u}_\eps)_*(s;B_{\sigma\eps}(x_0)) = (\widetilde{u}_\eps)_*(s/\eps^d;B_\sigma)$ and, in turn, 
\begin{equation*}
\tau'(\widetilde{u}_\eps;B_{\sigma}) = \tau'(\bar{u}_\eps;B_{\sigma\eps}(x_0))\,,\,\,\tau''(\widetilde{u}_\eps;B_{\sigma}) = \tau''(\bar{u}_\eps;B_{\sigma\eps}(x_0))\,,\,\,
{\rm med}(\widetilde{u}_\eps;B_\sigma) ={\rm med}(\bar{u}_\eps;B_{\sigma\eps}(x_0))\,.
\end{equation*}
We have
\begin{equation*}
T_\eps \bar{u}_\eps(x_0+\eps y)=T_\sigma\widetilde{u}_\eps(y)\,,
\end{equation*}
so, recalling that $u_\eps=v_\eps$ in $B_{\sigma\eps}(x_0)$, \eqref{eq:(16)}$(ii)$ can be rephrased as
\begin{equation}
\int_{B_\sigma}|T_\sigma\widetilde{u}_\eps(y)-\nabla u(x_0)\cdot y|^{p_\eps(y)}\,\mathrm{d}y\to0\,,
\label{eq:(20)}
\end{equation}
as $\eps\to0$, where we have set
\begin{equation*}
p_\eps(y):=p(x_0+\eps y)\,,\quad y\in B_\sigma\,.
\end{equation*}
From \eqref{eq:givenpropertiesbis}$(a)$-$(b)$ we infer
\begin{equation*}
\int_{B_\sigma} |\nabla \widetilde{u}_\eps|^{p_\eps(y)}\,\mathrm{d}y\leq C\,,\quad \lim_{\eps\to 0} \mathcal{H}^{d-1}(J_{\widetilde{u}_\eps}\cap B_\sigma)=0\,.
\end{equation*}
Then, by virtue of Theorem~\ref{thm:3.5} there exist a function $\widetilde{u}_0\in W^{1,p(x_0)}(B_\sigma;\R)$ and a subsequence (not relabeled) of $\{\widetilde{u}_\eps\}$ such that
\begin{equation*}
\begin{split}
\int_{B_\sigma}|T_\sigma\widetilde{u}_\eps(y)-{\rm med}(\widetilde{u}_\eps;B_\sigma)-\widetilde{u}_0|^{p_\eps(y)}\,\mathrm{d}y\to0\,,\quad \widetilde{u}_\eps-{\rm med}(\widetilde{u}_\eps;B_\sigma) \to \widetilde{u}_0 \quad & \mathcal{L}^d-\mbox{a.e. in $B_\sigma$}\,.
\end{split}
\end{equation*}
\EEE
The assertion \eqref{eq:(20)} will then follow 
once we prove that
\begin{equation}
\lim_{\varepsilon\to0} {\rm med}(\widetilde{u}_\eps;B_\sigma)=0\,.
\label{eq:(21)}
\end{equation}
For this, notice that \eqref{eq:givenpropertiesbis}$(c)$ implies $\widetilde{u}_0(y)=\nabla u(x_0)\cdot y$ for $\mathcal{L}^d$-a.e. $y\in B_\sigma$. The a.e. convergence in measure of $\widetilde{u}_\eps-{\rm med}(\widetilde{u}_\eps;B_\sigma)$ to $\nabla u(x_0)\cdot y$ is now enough to reproduce  the proof of \cite[eq. (21)]{BFLM}, and obtain \eqref{eq:(21)}. We therefore omit the details.

If $u\in SBV^{p(\cdot)}(\Omega;\R^m)$, we may fix $x_0\in\Omega$ such that, in addition to \eqref{eq:givenpropertiesbis}$(a)$, \eqref{eq:givenpropertiesbis}$(c)$ holds in the stronger form
\begin{equation*}
\lim_{\eps \to 0} \  \eps^{-(d+1)}\int_{B_\eps(x_0)}{|u(x) - u(x_0) - \nabla u(x_0)(x-x_0)|}\,\mathrm{d}x =0 \,,
\end{equation*}
and also property
\begin{equation}
\lim_{\eps\to0} \eps^{-d}\int_{J_u\cap B_\eps(x_0)}|[u]|\,\mathrm{d}\mathcal{H}^{d-1}=0 
\label{eq:givenpropertiestris}
\end{equation}
is satisfied. Then as a consequence of Fubini's Theorem, we can fix $\sigma\in(0,1)$ such that $\mathcal{H}^{d-1}(J_u\cap\partial B_{\sigma\eps}(x_0))=0$ and
\begin{equation}
\lim_{\eps\to0} \eps^{-d}\int_{\partial B_{\sigma\eps}(x_0)}{|u(x) - u(x_0) - \nabla u(x_0)(x-x_0)|}\,\mathrm{d}\mathcal{H}^{d-1} =0\,.
\label{eq:7.5}
\end{equation}
Now, we can define the sequence $u_\eps$ as above and prove $(i)$, $(ii)$ and $(iii)$. Assertion $(i)'$ will follow from $(iii)$, \eqref{eq:7.5}, $(c)$ and H\"older's inequality, since $u_\eps=u$ in $B_\eps(x_0)\backslash B_{\sigma\eps}(x_0)$.

Finally, since by construction it holds that $|[u_\eps]|\leq |[u]|$ $\mathcal{H}^{d-1}$-a.e., property $(iii)'$ is a consequence of \eqref{eq:givenpropertiestris}.

\endproof

We are now in a position to prove Lemma \ref{lemma: sameminbulk}, which will follow as a consequence of Lemma~\ref{lem:bulk1}
 and Lemma~\ref{lem:bulk2}.

\begin{lemma}\label{lem:bulk1}
Let $p:\Omega\to(1,+\infty)$ be a Riemann integrable variable exponent satisfying {\rm\ref{assP1}}. Suppose that $\mathcal{F}$ satisfies {\rm\ref{assH1}} and  {\rm\ref{assH3}}--{\rm\ref{assH4}}  and let $u \in GSBV^{p(\cdot)}(\Omega;\R^m)$.  Then for $\mathcal{L}^{d}$-a.e.\ $x_0 \in \Omega$ we have
\begin{equation}
\displaystyle \lim_{\eps \to 0}\frac{\mathbf{m}_{\mathcal{F}}(u,B_\eps(x_0))}{\gamma_d\eps^{d}} \leq \mathop{\lim\sup}_{\eps\to0} \frac{\mathbf{m}_{\mathcal{F}}(\bar{u}_{x_0}^{\rm bulk},B_{\eps}(x_0))}{\gamma_d\eps^{d}}\,.
\label{eq:bulkineq1}
\end{equation} 
\end{lemma}
\proof
We will prove the assertion for those points $x_0\in \Omega$ for which the statement of Lemma~\ref{lem:lemma2fonseca} holds and  $\lim_{\eps \to 0} (\gamma_{d}\eps^{d})^{-1}\mu(B_\eps(x_0))=1$. This holds for $\mathcal{L}^d$-a.e. $x_0\in\Omega$.
Also, by Lemma~\ref{lemma: F=m}, we know that for $\mathcal{L}^d$-a.e. $x_0\in\Omega$
\begin{equation}
\lim_{\eps \to 0}\frac{\mathcal{F}(u,B_\eps(x_0))}{\gamma_d \eps^d} =  \lim_{\eps \to 0}\frac{\mathbf{m}_{\mathcal{F}}(u,B_\eps(x_0))}{\gamma_d \eps^d}<+\infty\,.
\label{eq:lemma4.1}
\end{equation}

Let $(u_\eps)_\eps$ be the sequence of Lemma~\ref{lem:lemma2fonseca} and we fix $\sigma\in(0,1)$ such that \eqref{eq:(16)}$(ii)$ holds. We write $\sigma=1-\theta$ for some $\theta\in(0,1)$.

Given $z_\eps\in GSBV^{p(\cdot)}(B_{(1-3\theta)\eps}(x_0);\R^m)$ such that $z_\eps=\bar{u}_{x_0}^{\rm bulk}$ in a neighborhood of $\partial B_{(1-3\theta)\eps}(x_0)$ and
\begin{equation}
\mathcal{F}(z_\eps,B_{(1-3\theta)\eps}(x_0)) \leq \mathbf{m}_{\mathcal{F}}(\bar{u}_{x_0}^{\rm bulk},B_{(1-3\theta)\eps}(x_0)) + \gamma_d\eps^{d+1}\,,
\label{eq:5.8}
\end{equation}
we extend it to $z_\eps\in GSBV^{p(\cdot)}(B_{\eps}(x_0);\R^m)$ by setting $z_\eps=\bar{u}_{x_0}^{\rm bulk}$ outside $B_{(1-3\theta)\eps}(x_0)$. Now, we apply Lemma~\ref{lemma: fundamental estimate} with $u$ and $v$ replaced by $z_\eps$ and $u_\eps$, respectively, and 
\begin{equation}
D'_{\eps,x_0}:=B_{(1-2\theta)\eps}(x_0)\,,\,\, D''_{\eps,x_0}:=B_{(1-\theta)\eps}(x_0)\,,\,\, E_{\eps,x_0}:=C_{\eps,\theta}(x_0)\,, 
\label{eq:choiceofsets}
\end{equation}
where, to enlighten the notation, we denote by $C_{\eps,\theta}(x_0)$ the annulus $B_\eps(x_0)\backslash \overline{B_{(1-4\theta)\eps}(x_0)}$. Note that $C_{\eps,\theta}(x_0)=(C_{1,\theta}(x_0))_{\eps,x_0}$ according to notation \eqref{eq: shift-not}, where $C_{1,\theta}(x_0):=B_1(x_0)\backslash \overline{B_{(1-4\theta)}(x_0)}$. Also, $\mathcal{L}^d(C_{1,\theta}(x_0))=\gamma_d(1-(1-4\theta)^d)\to0$ as $\theta\to0$. 

With fixed $\eta>0$, we then find ${w}_\eps\in GSBV^{p(\cdot)}(B_\eps(x_0);\R^m)$ such that ${w}_\eps=u_\eps$ on $B_\eps(x_0)\backslash B_{(1-\theta)\eps}(x_0)$ and
\begin{equation}
\begin{split}
\mathcal{F}({w}_\eps, B_\eps(x_0)) & \leq (1+\eta) \left(\mathcal{F}(z_\eps, B_{(1-\theta)\eps}(x_0)) + \mathcal{F}(u_\eps, C_{\eps,\theta}(x_0))\right) \\
& + M \int_{(B_{(1-\theta)\eps}(x_0) \setminus B_{(1-2\theta)\eps}(x_0))}\left(\frac{|z_\eps-u_\eps|}{\eps}\right)^{p(x)}\,\mathrm{d}x +\eta\mathcal{L}^d(B_\eps(x_0))\,.
\end{split}
\label{eq:5.10}
\end{equation}
Recalling the definition of $u_\eps$, we have $w_\eps=u_\eps=u$ in a neighborhood of $\partial B_\eps(x_0)$. Moreover, since $z_\eps=\bar{u}_{x_0}^{\rm bulk}$ outside $B_{(1-3\theta)\eps}(x_0)$, by virtue of \eqref{eq:(16)}$(ii)$ we conclude that
\begin{equation}
\begin{split}
&\lim_{\eps\to0} \eps^{-d} \int_{(B_{(1-\theta)\eps}(x_0) \setminus B_{(1-2\theta)\eps}(x_0))}\left(\frac{|z_\eps-u_\eps|}{\eps}\right)^{p(x)}\,\mathrm{d}x \\
& = \lim_{\eps \to 0}  \ \eps^{-d} \int_{B_{(1-\theta)\eps}(x_0)} \left(\frac{|u_\eps - \bar{u}_{x_0}^{\rm bulk}|}{\eps} \right)^{p(x)} \, \mathrm{d}x = 0\,.
\end{split}
\label{eq:5.11}
\end{equation}
From this and \eqref{eq:5.10} we infer that there exists  a non-negative sequence $(\varrho_\eps)_\eps$, vanishing as $\eps\to0$, such that  
\begin{align}\label{eq:5.12}
\mathcal{F}&(w_\eps, B_\eps(x_0)) \le   (1+\eta)\left(\mathcal{F}(z_\eps, B_{(1-\theta)\eps}(x_0)) + \mathcal{F}(u_\eps, C_{\eps,\theta}(x_0))\right) +  \eps^d\varrho_\eps +\gamma_d\eps^d\eta.
\end{align}
We set for brevity
\begin{equation}
|\nabla u(x_0)|^{\widetilde{p}}:=\max\{|\nabla u(x_0)|^{p^-}, |\nabla u(x_0)|^{p^+}\}\,.
\label{eq:gradientpmax}
\end{equation}
Then, by using that  $z_\eps = \bar{u}^{\rm bulk}_{x_0}$ on  $B_{\eps}(x_0) \setminus B_{(1-3\theta)\eps}(x_0) \subset C_{\eps,\theta}(x_0)$,  \ref{assH1},  \ref{assH4},   and \eqref{eq:5.8} we compute  
\begin{align}\label{eq: rep1}
\limsup_{\eps\to 0}\frac{\mathcal{F}(z_\eps, B_{(1-\theta)\eps}(x_0))}{\eps^{d}} &\le \limsup_{\eps\to 0} \frac{\mathcal{F}(z_\eps,B_{(1- 3 \theta)\eps}(x_0))}{\eps^{d}}+  \limsup_{\eps\to 0} \frac{\mathcal{F}(\bar{u}^{\rm bulk}_{x_0}, C_{\eps,\theta}(x_0))}{\eps^{d}}\notag\\
&  \le  \limsup_{\eps\to 0} \frac{\mathbf{m}_{\mathcal{F}}(\bar{u}^{\rm bulk}_{x_0},B_{(1-3\theta)\eps}(x_0))}{\eps^{d}} + \beta \,  \mathcal{L}^d(C_{1,\theta}(x_0)) (1+|\nabla u(x_0)|^{\widetilde{p}}) \notag \\
&\le  (1-3\theta)^{d}\limsup_{\eps\to 0} \frac{\mathbf{m}_{\mathcal{F}}(\bar{u}^{\rm bulk}_{x_0},B_{(1-3\theta)\eps}(x_0))}{(1-3\theta)^{d}\eps^{d}} \\
& +\beta \, \mathcal{L}^d(C_{1,\theta}(x_0)) (1+|\nabla u(x_0)|^{\widetilde{p}})\,. \notag
\end{align}
On the other hand, by  \ref{assH4}  
we also obtain
\begin{align*}
\mathcal{F}(u_\eps,  C_{\eps,\theta}(x_0)) &\le \beta \int_{C_{\eps,\theta}(x_0)} (1+  |\nabla u_\eps|^{p(x)})\,\mathrm{d}x + \beta \,\mathcal{H}^{d-1}(J_{u_\eps} \cap  C_{\eps,\theta}(x_0))) \\
&\le \beta \mathcal{L}^d(C_{\eps,\theta}(x_0)) 
+ \beta \int_{ C_{\eps,\theta}(x_0)} |\nabla u_\eps|^{p(x)}\,\mathrm{d}x +\beta\mathcal{H}^{d-1}(J_{u_\eps}\cap C_{\eps,\theta}(x_0))\,. 
\end{align*}  
Now, taking into account \eqref{eq:(16)}$(iii)$ we get
\begin{align}\label{eq: rep1XXX}
\limsup_{\eps \to 0} \frac{\mathcal{F}(u_\eps,  C_{\eps,\theta}(x_0))}{\eps^{d}}  \le \beta  \, \mathcal{L}^d( C_{1,\theta}(x_0))  + \limsup_{\eps \to 0} \frac{\beta}{\eps^d} \int_{ C_{\eps,\theta}(x_0)} |\nabla u_\eps|^{p(x)}\,\mathrm{d}x\,.
\end{align}
Since $|\nabla u_\eps|\leq |\nabla u|$ $\mathcal{L}^d$-a.e., we have, with \eqref{eq:givenpropertiesbis}$(a)$,
\begin{equation}\label{eq:stimaintermedia}
\begin{split}
&\limsup_{\eps \to 0} \frac{\beta}{\eps^d} \int_{ C_{\eps,\theta}(x_0)} |\nabla u_\eps|^{p(x)}\,\mathrm{d}x \\
&\leq  \limsup_{\eps \to 0} \frac{\beta}{\eps^d} \int_{ C_{\eps,\theta}(x_0)} |\nabla u|^{p(x)}\,\mathrm{d}x \\
& \leq \limsup_{\eps \to 0} 2^{p^+-1}\left(\frac{\beta}{\eps^d}\int_{B_\eps(x_0)} |\nabla u -\nabla u(x_0)|^{p(x)}\,\mathrm{d}x + \beta|\nabla u(x_0)|^{\widetilde{p}}\mathcal{L}^d(C_{1,\theta}(x_0))\right) \\
& \leq 2^{p^+-1} \beta|\nabla u(x_0)|^{\widetilde{p}} \mathcal{L}^d(C_{1,\theta}(x_0))\,.
\end{split}
\end{equation} 
Combining \eqref{eq: rep1XXX} with \eqref{eq:stimaintermedia} we finally get
\begin{equation}
\limsup_{\eps \to 0} \frac{\mathcal{F}(u_\eps, C_{\eps,\theta}(x_0))}{\eps^{d}}  \le \beta  \,\mathcal{L}^d(C_{1,\theta}(x_0)) (1+2^{p^+-1} |\nabla u(x_0)|^{\widetilde{p}})\,.
\label{eq:5.14}
\end{equation}
Recall that  $w_\eps = u$ in a neighborhood of $\partial B_\eps(x_0)$. This along with  \eqref{eq:5.12}, \eqref{eq: rep1}, \eqref{eq:5.14} and  $\varrho_\eps \to 0$ yields
\begin{align*}
 \displaystyle \lim_{\eps \to 0}\frac{\mathbf{m}_{\mathcal{F}}(u,B_\eps(x_0))}{\gamma_{d}\eps^{d}} &\le  \limsup_{\eps \to 0}\frac{\mathcal{F}(w_\eps, B_\eps(x_0))}{\gamma_{d}\eps^{d}} \notag \\
 & \le (1+\eta) \, (1-3\theta)^{d} \limsup_{\eps \to 0}  \frac{\mathbf{m}_{\mathcal{F}}(\bar{u}^{\rm bulk}_{x_0},B_\eps(x_0))}{\gamma_{d}\eps^{d}} \\ & \ \ \ + (1+\eta) \beta \gamma_d^{-1}\mathcal{L}^d(C_{1,\theta}(x_0)) (1+2^{p^+}|\nabla u(x_0)|^{\widetilde{p}}) + \eta,
\end{align*}
whence \eqref{eq:bulkineq1} follows up to passing to $\eta,\theta \to 0$. The proof is concluded.

\endproof

\begin{lemma}\label{lem:bulk2}
Under the assumptions of Lemma~\ref{lem:bulk1}, for $\mathcal{L}^{d}$-a.e.\ $x_0 \in \Omega$ we have
\begin{equation}
\displaystyle \lim_{\eps \to 0}\frac{\mathbf{m}_{\mathcal{F}}(u,B_\eps(x_0))}{\gamma_d\eps^{d}} \geq \mathop{\lim\sup}_{\eps\to0} \frac{\mathbf{m}_{\mathcal{F}}(\bar{u}_{x_0}^{\rm bulk},B_{\eps}(x_0))}{\gamma_d\eps^{d}}\,.
\label{eq:bulkineq2}
\end{equation} 
\end{lemma}

\proof
We can restrict the proof to those points $x_0\in\Omega$ considered in Lemma~\ref{lem:bulk1}. Let $\eta>0$, $\sigma=1-\theta$ fixed as in Lemma~\ref{lem:bulk1}, and let $(u_\eps)_\eps$ be the sequence provided by Lemma~\ref{lem:lemma2fonseca}. An argument based on Fubini's Theorem (see \eqref{eq:contocoarea} and \eqref{eq:43}) shows that for each $\eps>0$ we can find $s\in(1-4\theta, 1-3\theta)$ such that
\begin{equation}
\begin{split}
&\mathcal{H}^{d-1}(\partial B_{s\eps}(x_0)\cap (J_{u_\eps}\cup J_u))=0\,,\quad \mbox{ for all $\eps>0$,}  \\
&\lim_{\eps\to0} \eps^{-d} \mathcal{H}^{d-1}(\{u\neq u_\eps\}\cap \partial B_{s\eps}(x_0))
=0\,.
\end{split}
\label{eq:5.15}
\end{equation}
From now on, the argument of the proof closely follows that of Lemma~\ref{lem:bulk1}. We choose a sequence $z_\eps\in GSBV^{p(\cdot)}(B_{s\eps}(x_0);\R^m)$ such that $z_\eps=u$ in a neighborhood of $\partial B_{s\eps}(x_0)$ and
\begin{equation}
\mathcal{F}(z_\eps,B_{s\eps}(x_0)) \leq \mathbf{m}_{\mathcal{F}}(u,B_{s\eps}(x_0)) + \gamma_d\eps^{d+1}\,.
\label{eq:5.16}
\end{equation}
Setting $z_\eps=u_\eps$ outside $B_{s\eps}(x_0)$, we extend it to $z_\eps\in GSBV^{p(\cdot)}(B_{\eps}(x_0);\R^m)$. Now, we apply Lemma~\ref{lemma: fundamental estimate} with $u$ and $v$ replaced by $z_\eps$ and $\bar{u}_{x_0}^{\rm bulk}$, respectively, and the same choice for the sets $D'_{\eps,x_0}$, $D''_{\eps,x_0}$ and $E_{\eps,x_0}$ as in Lemma~\ref{lem:bulk1}, see \eqref{eq:choiceofsets}. 

By virtue of Lemma~\ref{lemma: fundamental estimate}, we then find ${w}_\eps\in GSBV^{p(\cdot)}(B_\eps(x_0);\R^m)$ such that ${w}_\eps=\bar{u}_{x_0}^{\rm bulk}$ on $B_\eps(x_0)\backslash B_{(1-\theta)\eps}(x_0)$ and
\begin{equation*}
\begin{split}
\mathcal{F}({w}_\eps, B_\eps(x_0)) & \leq (1+\eta) \left(\mathcal{F}(z_\eps, B_{(1-\theta)\eps}(x_0)) + \mathcal{F}(\bar{u}_{x_0}^{\rm bulk}, C_{\eps,\theta}(x_0))\right) \\
& + M \int_{(B_{(1-\theta)\eps}(x_0) \setminus B_{(1-2\theta)\eps}(x_0))}\left(\frac{|z_\eps-\bar{u}_{x_0}^{\rm bulk}|}{\eps}\right)^{p(x)}\,\mathrm{d}x +\eta\mathcal{L}^d(B_\eps(x_0))\,.
\end{split}
\end{equation*}
Since $z_\eps=u_\eps$ outside $B_{(1-3\theta)\eps}(x_0)$ from the choice of $s$, by arguing as in Lemma~\ref{lem:bulk1}, see in particular \eqref{eq:5.11} and \eqref{eq:5.12}, we find a non-negative sequence $(\varrho_\eps)_\eps$, vanishing as $\eps\to0$, such that  
\begin{align}\label{eq:5.18}
\mathcal{F}&(w_\eps, B_\eps(x_0)) \le   (1+\eta)\left(\mathcal{F}(z_\eps, B_{(1-\theta)\eps}(x_0)) + \mathcal{F}(\bar{u}_{x_0}^{\rm bulk}, C_{\eps,\theta}(x_0))\right) +  \eps^d\varrho_\eps +\gamma_d\eps^d\eta.
\end{align} 
We now proceed to the estimate of the terms in \eqref{eq:5.18}. Using that  $z_\eps = u_\eps$ on  $B_{\eps}(x_0) \setminus B_{s\eps}(x_0) \subset C_{\eps,\theta}(x_0)$,  \ref{assH1},  \ref{assH4},   and \eqref{eq:5.16} we obtain 
\begin{equation}
\begin{split}
\mathcal{F}(z_\eps, B_{(1-\theta)\eps}(x_0)) & \leq \mathbf{m}_{\mathcal{F}}(u,B_{s\eps}(x_0)) + \gamma_d\eps^{d+1} + \beta \mathcal{H}^{d-1}(\partial B_{s\eps}(x_0)\cap (J_{u_\eps}\cup J_u)) \\
& + \mathcal{F}(u_\eps, C_{\eps,\theta}(x_0))\,. 
\end{split}
\label{eq:5.19}
\end{equation}
Now, with \eqref{eq:5.14} and \eqref{eq:5.15} and the fact that $s\eps\leq (1-3\theta)\eps$ we get
\begin{align}\label{eq:5.20}
\limsup_{\eps\to 0}\frac{\mathcal{F}(z_\eps, B_{(1-\theta)\eps}(x_0))}{\eps^{d}} & \le s^d \limsup_{\eps\to 0} \frac{\mathbf{m}_{\mathcal{F}}(u,B_{s\eps}(x_0))}{(s\eps)^{d}} + \beta \,  \mathcal{L}^d(C_{1,\theta}) (1+{|\nabla u(x_0)|^{\widetilde{p}}}) \notag \\
&\le  (1-3\theta)^{d}\limsup_{\eps\to 0} \frac{\mathbf{m}_{\mathcal{F}}(u,B_{\eps}(x_0))}{\eps^{d}} +\beta \, \mathcal{L}^d(C_{1,\theta}(x_0)) (1+{|\nabla u(x_0)|^{\widetilde{p}}})\,,
\end{align}
where $|\nabla u(x_0)|^{\widetilde{p}}$ is defined as in \eqref{eq:gradientpmax}.

The analogous of the estimate for $\mathcal{F}(\bar{u}_{x_0}^{\rm bulk}, C_{\eps,\theta}(x_0))$ in \eqref{eq: rep1}, the estimates \eqref{eq:5.18}, \eqref{eq:5.19}, \eqref{eq:5.20} and $\varrho_\eps\to0$ give
\begin{equation*}
\begin{split}
\limsup_{\eps\to 0}\frac{\mathcal{F}(w_\eps, B_{\eps}(x_0))}{\eps^{d}} & \leq (1+\eta) (1-3\theta)^{d}\limsup_{\eps\to 0} \frac{\mathbf{m}_{\mathcal{F}}(u,B_{\eps}(x_0))}{\eps^{d}} \\
& +2(1+\eta)\beta \, \mathcal{L}^d(C_{1,\theta}(x_0)) (1+{|\nabla u(x_0)|^{\widetilde{p}}})+\gamma_d\eta\,.
\end{split}
\end{equation*}
Finally, letting $\eta$ and $\theta$ to 0, and recalling that $w_\eps=\bar{u}_{x_0}^{\rm bulk}$ in a neighborhood of $\partial B_\eps(x_0)$, we can write
\begin{equation*}
\begin{split}
\mathop{\lim\sup}_{\eps\to0} \frac{\mathbf{m}_{\mathcal{F}}(\bar{u}_{x_0}^{\rm bulk},B_{\eps}(x_0))}{\gamma_d\eps^{d}} & \leq \mathop{\lim\sup}_{\eps\to0} \frac{\mathcal{F}(w_\eps,B_{\eps}(x_0))}{\gamma_d\eps^{d}} \\
& \leq \mathop{\lim\sup}_{\eps\to0} \frac{\mathbf{m}_{\mathcal{F}}(u,B_{\eps}(x_0))}{\gamma_d\eps^{d}} = \lim_{\eps\to0}  \frac{\mathbf{m}_{\mathcal{F}}(u,B_{\eps}(x_0))}{\gamma_d\eps^{d}}\,,
\end{split}
\end{equation*}
and this concludes the proof of \eqref{eq:bulkineq2}.
\EEE
\endproof

\subsection{The surface density}\label{sec:surf}

The proof of Lemma~\ref{lemma: sameminsurf} requires the analysis of the blow-up at the jump points of function $u$. To this aim, we need a refinement of the results of \cite[Lemma~3]{BFLM} to the case of a variable exponent $p(\cdot)$. This requires a careful analysis of the asymptotic behavior of some constants, where the assumption of log-H\"older continuity of the variable exponent $p(\cdot)$, see \ref{assP2}, plays a crucial role. 

We state and prove the announced blow-up properties for $u\in GSBV^{p(\cdot)}$ around each jump point $x_0\in J_u$. 

\begin{lemma}\label{lem:Lemma3fon}
Assume that $p:\Omega\to(1,+\infty)$ be continuous and complying with {\rm\ref{assP2}}. 
Let $u\in GSBV^{p(\cdot)}(\Omega;\R^m)$. Then for $\mathcal{H}^{d-1}$-a.e. $x_0\in J_u$, for $\mathcal{L}^1$-a.e. $\sigma\in(0,1)$ and for every $\eps>0$ there exists a function $\bar{u}_\eps\in GSBV^{p(\cdot)}(B_\eps(x_0);\R^m)$ with $\nu=\nu_u(x_0)$ such that
\begin{equation}
\begin{aligned}
& (i) \,\, \mathcal{H}^{d-1}((J_{\bar{u}_\eps}\backslash J_u)\cap B_\eps(x_0))=0\,,\\
& (ii) \,\, \displaystyle\lim_{\eps\to0} \eps^{-(d-1)} \int_{B_\eps(x_0)}|\nabla \bar{u}_\eps|^{p(x)}\,\mathrm{d}x=0\,, \\
& (iii) \,\,\displaystyle\lim_{\eps\to0} \eps^{-(d-1)} \int_{B_{\sigma\eps}(x_0)}\left(\frac{|\bar{u}_\eps-\bar{u}^{\rm surf}_{x_0}|}{\eps}\right)^{p(x)}\,\mathrm{d}x=0\,,\\
& (iv) \,\, \bar{u}_\eps=u \mbox{\,\,on $B_\eps(x_0)\backslash \overline{B_{\sigma\eps}(x_0)}$}\,
\end{aligned}
\label{eq:24}
\end{equation}
and
\begin{equation}
\begin{split}
&\lim_{\eps\to0} \eps^{-d} \mathcal{L}^d(\{x\in B_\eps(x_0):\,\, \bar{u}_\eps\neq u\})=0\,. 
\end{split}
\label{eq:25}
\end{equation}
If, in addition, $u\in SBV^{p(\cdot)}(\Omega;\R^m)$, we also have
\begin{equation}
\displaystyle\lim_{\eps\to0} \eps^{-d}\int_{B_\eps(x_0)}|\bar{u}_\eps(x)-u(x)|\,\mathrm{d}x=0\,,
\label{eq:25bis}
\end{equation}
and
\begin{equation}
\displaystyle\lim_{\eps\to0} \eps^{-(d-1)}\int_{J_{\bar{u}_\eps}\cap E_{\eps,x_0}}|[\bar{u}_\eps]|\,\mathrm{d}\mathcal{H}^{d-1}=|[\bar{u}^{\rm surf}_{x_0}]|\mathcal{H}^{d-1}(\Pi_0\cap E) \quad \mbox{ for all Borel sets $E\subset B_1(x_0)$}\,,
\label{eq:(7.8)_2}
\end{equation}
where $\Pi_0$ is the hyperplane passing through $x_0$ with normal $\nu_u(x_0)$.
\end{lemma}

\proof

We first note that since $|\nabla u|^{p(\cdot)}\in L^1(\Omega)$, the points $x_0\in J_u$ can be fixed such that
\begin{equation}
\lim_{\eps\to0} \eps^{-(d-1)}\int_{B_\eps(x_0)}|\nabla u|^{p(x)}\,\mathrm{d}x=0
\label{eq:blowupevgar}
\end{equation}
(see, e.g., \cite[Section~2.4.3, Theorem~3]{EG}). 
Further, since $J_u$ is $(d-1)$-rectifiable, there exists a sequence of compact sets $K_j$ such that $J_u=\bigcup_{j=1}^\infty K_j\cup N$, for some $N$ such that $\mathcal{H}^{d-1}(N)=0$, and each $K_j$ is a subset of a $C^1$ hypersurface. Then, in a neighborhood $B_{\eps_0}(y)\subset \Omega$ of each point $y\in K_j$, up to a rotation, we may find a $C^1$ function $\Gamma_j:\R^{d-1}\to\R$ such that 
\begin{equation*}
K_j\cap B_{\eps_0}(y) \subset \{x=(x',x_d)\in B_{\eps_0}(y):\,\,x_d=\Gamma_j(x')\}\,.
\end{equation*}
We now define the function $w\in GSBV^{p(\cdot)}(B_{\eps_0}(y);\R^m)$ by setting 
\begin{equation*}
w(x):=
\begin{cases}
u(x',x_d) & \mbox{ if } x_d>\Gamma_j(x')\,,\\
u(x', -x_d+2\Gamma_j(x'))  & \mbox{ if } x_d<\Gamma_j(x')\,.
\end{cases}
\end{equation*}
Notice indeed that by construction we have $|\nabla w|\le C |\nabla u|$ a.e., hence $w\in GSBV^{p(\cdot)}(B_{\eps_0}(y);\R^m)$. Furthemore, $J_w\cap B_{\eps_0}(y)\subset B_{\eps_0}(y)\backslash K_j$. Now, following the argument of \cite[Lemma~3]{BFLM}, we can fix $x_0\in B_{\eps_0}(y)\cap K_j$ with the following properties: 
\begin{equation}
w(x_0)=u^+(x_0)\,,\quad \lim_{\eps\to 0^+}\frac1{\eps^{d-1}}\int_{B_\eps(x_0)\cap K_j} |w-u^+(x_0)|\,\mathrm{d}x=0\,,
\label{eq:32plus}
\end{equation} 
and for fixed $\eta>0$ (small enough) there exists (a smaller, if necessary) $\eps_0>0$ such that 
\begin{equation}
\int_{B_\eps(x_0)} |\nabla w|^{q}\,\mathrm{d}x + \mathcal{H}^{d-1}(J_w\cap B_\eps(x_0)) < \eta \eps^{d-1}
\label{eq:32}
\end{equation}
holds, for all $\eps<\eps_0$, for $q=p^-_\Omega$. Moreover, if we set for every $\eps>0$
\begin{equation*}
p^-_\eps:= \inf_{x\in B_\eps(x_0)} p(x)\,,\quad p^+_\eps:= \sup_{x\in B_\eps(x_0)} p(x)\,,
\end{equation*}
combining with \eqref{eq:blowupevgar} we have that \eqref{eq:32} is indeed satisfied for $q=p^-_\eps$. \UUU

\BBB Now, fix $q$ such that \eqref{eq:32} holds. Define $T_\eps w(x):= T_{B_\eps(x_0)}w(x)$ as in \eqref{eq:truncated} with $u=w$ and $B=B_\eps(x_0)$. From the Poincar\'e inequality \eqref{(11)}, \eqref{eq:32} and for any $q\leq r < q^*$ we have
\begin{equation*}
\begin{split}
&\int_{B_\eps(x_0)}{|T_\eps w-{\rm med}(w;B_\eps(x_0))|}^r\,\mathrm{d}x \\
&\leq \left(\int_{B_\eps(x_0)}{|T_\eps w-{\rm med}(w;B_\eps(x_0))|}^{q^*}\,\mathrm{d}x\right)^{\frac{r}{q^*}}[\mathcal{L}^d(B_\eps)]^{1-\frac{r}{q^*}} \\ 
& \leq \left({2\gamma_{\rm iso}q^*(d-1)}\right)^r\left(\int_{B_\eps(x_0)}{|\nabla w|}^{q}\,\mathrm{d}x\right)^{\frac{r}{q}}[\mathcal{L}^d(B_\eps)]^{1-\frac{r}{q^*}} \\
& \leq C(d,q,r)\eta^{\frac{r}{q}} \eps^{\frac{r(d-1)}{q}+d-\frac{r(d-q)}{q}} \\
& = C(d,q,r)\eta^{\frac{r}{q}} \eps^{\frac{r(q-1)}{q}+d}\,,
\end{split}
\end{equation*}
where $C(d,q,r):=\left({2\gamma_{\rm iso}q^*(d-1)}\gamma_d^{\frac{1}{r}-\frac{1}{q^*}}\right)^r$. Since arguing as for the proof of \cite[eq. (34)]{BFLM} we can prove that
\begin{equation*}
\begin{split}
&\int_{B_\eps(x_0)}{|{\rm med}(w,B_\eps(x_0)) - u^+(x_0)|}^r\,\mathrm{d}x \leq \eta^r \eps^{\frac{r(q-1)}{q}+d}
\end{split}
\end{equation*}
for $\eps$ small enough, collecting the previous estimates we finally obtain
\begin{equation}
\frac{1}{\eps^{d-1}}\int_{B_\eps(x_0)}\left({\frac{|T_\eps w-u^+(x_0)|}{\eps}}\right)^r\,\mathrm{d}x \leq 2\max\{C(d,q,r)\eta^{\frac{r}{q}},\eta^r\}\eps^{-\frac{(r-q)}{q}}\,.
\label{eq:crucial}
\end{equation}
If we define the function $z$ as
\begin{equation*}
z(x):=
\begin{cases}
u(x',x_d) & \mbox{ if } x_d<\Gamma_j(x')\,,\\
u(x', -x_d+2\Gamma_j(x'))  & \mbox{ if } x_d>\Gamma_j(x')\,,
\end{cases}
\end{equation*}
then $z$ complies with the \eqref{eq:32plus}-\eqref{eq:32}, up to replacing $w$ with $z$ and $u^+(x_0)$ with $u^-(x_0)$. Hence, an analogous estimate as in \eqref{eq:crucial} can be inferred for the sequence $T_\eps z$ defined as the truncation $T_{B_\eps(x_0)}z$ of the function $z$. We then set
\begin{equation*}
u_\eps(x):=
\begin{cases}
T_\eps w(x) & \mbox{ if } x_d>\Gamma_j(x')\,,\\
T_\eps z(x) & \mbox{ if } x_d<\Gamma_j(x')\,,
\end{cases}
\end{equation*}
and we have
\begin{equation}
\frac{1}{\eps^{d-1}}\int_{B_{\sigma\eps}(x_0)}\left({\frac{|{u}_\eps-\bar{u}^{\rm surf}_{x_0}|}{\eps}}\right)^{q}\,\mathrm{d}x \leq 2\max\{C(d,q,r)\eta^{\frac{r}{q}},\eta^{r}\}\eps^{-\frac{(r-q)}{q}}\,,\quad \forall r\in[q,q^*)\,,
\label{eq:crucial0}
\end{equation}
Arguing exactly as in \cite[Lemma~3]{BFLM}, we also have
\begin{equation}\label{eq:25-}
\lim_{\eps\to0} \eps^{-d} \mathcal{L}^d(\{x\in B_\eps(x_0):\,\, u_\eps\neq u\})=0\,.
\end{equation}
\EEE

Now, an analogous argument as for \eqref{eq:contocoarea} shows that for every sequence $\eps\to0$ one can find a subsequence (not relabeled) such that, for $\mathcal{L}^1$-a.e. $\sigma\in(0,1)$,
\begin{equation*}
\begin{split}
\mu(\partial B_{\sigma\eps}(x_0)) = \mathcal{H}^{d-1}(\partial B_{\sigma\eps}(x_0)\cap J_{u_\eps})=0\,,  \\
\lim_{\eps\to0} \eps^{-d} \mathcal{H}^{d-1}(\{u_\eps\neq u\}\cap \partial B_{\sigma\eps}(x_0)) = 0\,. 
\end{split}
\end{equation*}
We then define
\begin{equation*}
\bar{u}_\eps(x) =
\begin{cases}
u_\eps(x) & \mbox{ in } B_{\sigma\eps}(x_0)\,,\\
u(x) & \mbox{ in } B_\eps(x_0)\backslash \overline{B_{\sigma\eps}(x_0)}\,.
\end{cases}
\end{equation*} 
Now, property \eqref{eq:24} $(i)$, $(ii)$ and $(iv)$  follow from the definition and \eqref{eq:blowupevgar}, while \eqref{eq:25} is immediate from \eqref{eq:25-}. As for \eqref{eq:24} $(iii)$, with fixed $\eta>0$, the estimate \eqref{eq:crucial0} with $q=p^-_\eps$ and $r=p^+_\eps$ implies that
\begin{equation}
\frac{1}{\eps^{d-1}}\int_{B_{\sigma\eps}(x_0)}\left({\frac{|\bar{u}_\eps-\bar{u}^{\rm surf}_{x_0}|}{\eps}}\right)^{p^+_\eps}\,\mathrm{d}x \leq 2 \max\{C(d,p^-_\eps, p^+_\eps)\eta^{\frac{p^+_\eps}{p^-_\eps}},\eta^{p^+_\eps}\}\eps^{-\frac{(p^+_\eps-p^-_\eps)}{p^-_\eps}}
\label{eq:crucial2}
\end{equation}
for $\eps$ small enough. Observe that, by its definition, the constant $C(d,p^-_\eps, p^+_\eps)$ is a bounded function of $\eps$. Now, since
\begin{equation*}
\begin{split}
\frac{1}{\eps^{d-1}}\int_{B_{\sigma\eps}(x_0)}\left({\frac{|\bar{u}_\eps-\bar{u}^{\rm surf}_{x_0}|}{\eps}}\right)^{p(\cdot)}\,\mathrm{d}x & \leq \eps + \frac{1}{\eps^{d-1}}\int_{B_{\sigma\eps}(x_0)}\left({\frac{|\bar{u}_\eps-\bar{u}^{\rm surf}_{x_0}|}{\eps}}\right)^{p^+_\eps}\,\mathrm{d}x \\
& \leq \eps + 2\max\{C(d,p^-_\eps, p^+_\eps)\eta^{\frac{p^+_\eps}{p^-_\eps}},\eta^{p^+_\eps}\}\eps^{-\frac{(p^+_\eps-p^-_\eps)}{p^-_\eps}}\,,
\end{split}
\end{equation*}
and with $p_\eps^+\le p^+_\Omega$, assertion $(iii)$ in \eqref{eq:24} will follow sending $\eps \to 0$ first and then $\eta\to0$, once we note that 
\begin{equation*}
\mathop{\lim\sup}_{\eps\to0} \eps^{-\frac{(p^+_\eps-p^-_\eps)}{p^-_\eps}}\leq c_1 
\end{equation*}
for some constant $c_1$ by virtue of \ref{assP2}. 

Assertion \eqref{eq:25bis} for a function $u\in SBV(\Omega;\R^m)$ can be obtained exactly as in \cite[Lemma~3]{BFLM} as a consequence of H\"older's inequality, combining \eqref{eq:crucial0}, written for $r=q=p^-_\Omega$, and the property
\begin{equation*}
\lim_{\eps\to0} \eps^{-d}\int_{B_\eps(x_0)}|u(x)-\bar{u}^{\rm surf}_{x_0}(x)|\,\mathrm{d}x=0\,.
\end{equation*}
We omit further details.

Concerning \eqref{eq:(7.8)_2}, we begin by observing that, if $u\in SBV(\Omega; \mathbb{R}^m)$ for $\mathcal{H}^{d-1}$-a.e. $x_0\in J_u$ we have
\begin{equation}\label{eq:fauljohannes}
\lim_{\eps\to 0}\frac{1}{\eps^{d-1}}\int_{(J_u\cap B_\eps(x_0))\setminus K_j}(1+|[u]|)\,\mathrm{d}\mathcal{H}^{d-1}=0\,.
\end{equation}
If we now set
\[
\bar{u}^{\rm surf}_{\eps, x_0}(x):=
\begin{cases}
\tau' (w, B_\eps(x_0))\wedge u^+(x_0)\vee \tau{''} (w, B_\eps(x_0)) & \mbox{ if } x_d>\Gamma_j(x')\,,\\
\tau' (z, B_\eps(x_0))\wedge u^-(x_0)\vee \tau{''} (z, B_\eps(x_0)) & \mbox{ if } x_d<\Gamma_j(x')\,,
\end{cases}
\]
with \eqref{eq:32plus}, \eqref{eq:32},  \eqref{eq:fauljohannes}, and since truncations are $1$-Lipschitz, we get
\[
\lim_{\eps\to 0}\frac{1}{\eps^{d-1}}\int_{J_{u_\eps}}|u_\eps-\bar{u}^{\rm surf}_{\eps, x_0}|\,\mathrm{d}\mathcal{H}^{d-1}=0\,.
\]
Now, as shown in \cite[Remark 2, Formula (39)]{BFLM}, one has componentwise
\[
\limsup_{\eps\to 0}\tau'' (w_i, B_\eps(x_0))\geq u^+_i(x_0)\,,\quad\liminf_{\eps\to 0}\tau' (w_i, B_\eps(x_0))\leq u^-_i(x_0)
\]
and the same properties also hold for $z$. With this, one has, for all $E\subset B_1(x_0)$,
\begin{equation*}
\begin{split}
\lim_{\eps\to 0}\frac{1}{\eps^{d-1}}\int_{J_{u_\eps}\cap E_{\eps, x_0}}|[u_\eps]|\,\mathrm{d}\mathcal{H}^{d-1} & =|[u]|(x_0)\lim_{\eps\to 0}\frac{1}{\eps^{d-1}}\mathcal{H}^{d-1}(J_{u_\eps}\cap E_{\eps, x_0}) \\
&=|[u]|(x_0)\mathcal{H}^{d-1}(\Pi_0\cap E)\,,
\end{split}
\end{equation*}
since the last property is satisfied at $\mathcal{H}^{d-1}$-a.e. $x_0\in J_u$ by the definition of measure-theoretic normal to a rectifiable set. This is clearly equivalent to \eqref{eq:(7.8)_2}.
\EEE 
\endproof

We now prove Lemma~\ref{lemma: sameminsurf}. The two inequalities in \eqref{eq: samemin-surf} will be shown with Lemma~\ref{lem:jump1} and Lemma~\ref{lem:jump2} below. 

\begin{lemma}\label{lem:jump1}
Let $p:\Omega\to(1,+\infty)$ be a variable exponent satisfying {\rm\ref{assP1}}-{\rm\ref{assP2}}. \EEE Suppose that $\mathcal{F}$ satisfies {\rm\ref{assH1}} and  {\rm\ref{assH3}}--{\rm\ref{assH4}} and let $u \in GSBV^{p(\cdot)}(\Omega;\R^m)$.  Then for $\mathcal{H}^{d-1}$-a.e.\ $x_0 \in J_u$  we have
\begin{align}\label{eq: surf-proof1upper}
  \lim_{\eps \to 0}\frac{\mathbf{m}_{\mathcal{F}}(u,B_\eps(x_0))}{\gamma_{d-1}\eps^{d-1}} \leq  \limsup_{\eps \to 0}\frac{\mathbf{m}_{\mathcal{F}}(\bar{u}^{\rm surf}_{x_0},B_\eps(x_0))}{\gamma_{d-1}\eps^{d-1}}\,. 
  \end{align}
\end{lemma}

\proof
Let $\bar{u}_\eps$ be the sequence of Lemma~\ref{lem:Lemma3fon}, let $x_0$ be such that Lemma~\ref{lem:Lemma3fon} holds, and set $\nu:=\nu_u(x_0)$. By Lemma~\ref{lemma: F=m}, for $\mathcal{H}^{d-1}$-a.e. $x_0\in J_u\cap\Omega$ we have
\begin{equation}
\frac{\mathrm{d}\mathcal{F}(u,\cdot)}{\mathrm{d}\mathcal{H}^{d-1}\lfloor_{J_u}}(x_0)  = \lim_{\eps \to 0}\frac{\mathcal{F}(u,B_\eps(x_0))}{\mu(B_\eps(x_0))}= \lim_{\eps \to 0}\frac{\mathbf{m}_{\mathcal{F}}(u,B_\eps(x_0))}{\gamma_{d-1}\eps^{d-1}} < \infty\,.
\label{eq:52}
\end{equation}

Let $\eta>0$ and $\sigma\in(0,1)$ be fixed such that Lemma~\ref{lem:Lemma3fon} holds, and set $\sigma=1-\theta$ for some $\theta\in(0,1)$. We consider a sequence $\bar{z}_\eps \in GSBV^{p(\cdot)}(B_{(1-3\theta)\eps}(x_0);\R^m)$ with $\bar{z}_\eps = \bar{u}^{\rm surf}_{x_0}$ in a neighborhood of $\partial B_{(1-3\theta)\eps}(x_0)$ and
\begin{align}\label{eq:6.20}
\mathcal{F}\big(\bar{z}_\eps,B_{(1-3\theta)\eps}(x_0)\big) \le \mathbf{m}_{\mathcal{F}}\big(\bar{u}^{\rm surf}_{x_0},B_{(1-3\theta)\eps}(x_0)\big) + \gamma_{d-1} \eps^{d}.
\end{align}
We extend $\bar{z}_\eps$ to a function in $GSBV^{p(\cdot)}(B_\eps(x_0);\R^m)$ by setting  $\bar{z}_\eps = \bar{u}^{\rm surf}_{x_0}$ outside $B_{(1-3\theta)\eps}(x_0)$. 
Now, we apply Lemma~\ref{lemma: fundamental estimate} with $u$ and $v$ replaced by $\bar{z}_\eps$ and $\bar{u}_\eps$, respectively, and $D'_{\eps,x_0}:=B_{(1-2\theta)\eps}(x_0)$, $D''_{\eps,x_0}:=B_{(1-\theta)\eps}(x_0)$ and $E_{\eps,x_0}:=C_{\eps,\theta}(x_0)$, where $C_{\eps,\theta}(x_0)$ still denotes the annulus $B_\eps(x_0)\backslash \overline{B_{(1-4\theta)\eps}(x_0)}$ (see \eqref{eq:choiceofsets}). We then find $\bar{w}_\eps\in GSBV^{p(\cdot)}(B_\eps(x_0);\R^m)$ such that $\bar{w}_\eps=\bar{u}_\eps$ on $B_\eps(x_0)\backslash B_{(1-\theta)\eps}(x_0)$ and
\begin{equation}
\begin{split}
\mathcal{F}(\bar{w}_\eps, B_\eps(x_0)) & \leq (1+\eta) \left(\mathcal{F}(\bar{z}_\eps, B_{(1-\theta)\eps}(x_0)) + \mathcal{F}(\bar{u}_\eps, C_{\eps,\theta}(x_0))\right) \\
& + M \int_{(B_{(1-\theta)\eps}(x_0) \setminus B_{(1-2\theta)\eps}(x_0))}\left(\frac{|\bar{z}_\eps-\bar{u}_\eps|}{\eps}\right)^{p(x)}\,\mathrm{d}x +\eta\mathcal{L}^d(B_\eps(x_0))\,.
\end{split}
\label{eq:6.21}
\end{equation}
In particular, by \eqref{eq:24}$(iv)$ we have that $\bar{w}_\eps=\bar{z}_\eps=u$ in a neighborhood of $\partial B_\eps(x_0)$. By \eqref{eq:24}$(iii)$ and the fact that $\bar{z}_\eps=\bar{u}^{\rm surf}_{x_0}$ outside $B_{(1-3\theta)\eps}(x_0)$ we get
\begin{equation}
\begin{split}
&\lim_{\eps\to0} \eps^{-(d+1)} \int_{(B_{(1-\theta)\eps}(x_0) \setminus B_{(1-2\theta)\eps}(x_0))}\left(\frac{|\bar{z}_\eps-\bar{u}_\eps|}{\eps}\right)^{p(x)}\,\mathrm{d}x \\
& = \lim_{\eps \to 0}  \ \eps^{-(d+1)} \int_{B_{(1-\theta)\eps}(x_0)} \left(\frac{|\bar{u}_\eps - \bar{u}^{\rm surf}_{x_0}|}{\eps} \right)^{p(x)} \, \mathrm{d}x = 0\,.
\end{split}
\label{eq:6.21bis}
\end{equation}
Plugging in \eqref{eq:6.21} we find that, for a suitable non-negative vanishing sequence $\varrho_\eps$, it holds that
\begin{equation}
\begin{split}
\mathcal{F}(\bar{w}_\eps, B_\eps(x_0))\leq (1+\eta) \left(\mathcal{F}(\bar{z}_\eps, B_{(1-\theta)\eps}(x_0)) + \mathcal{F}(\bar{u}_\eps, C_{\eps,\theta}(x_0))\right) + \eps^{d-1}\varrho_\eps +\gamma_d\eps^d\eta\,.
\end{split}
\label{eq:6.22}
\end{equation}
In order to estimate the terms in \eqref{eq:6.22}, using that  $\bar{z}_\eps = \bar{u}^{\rm surf}_{x_0}$ on  $B_{\eps}(x_0) \setminus B_{(1-3\theta)\eps}(x_0) \subset C_{\eps,\theta}(x_0)$,  \ref{assH1},  \ref{assH4},   and \eqref{eq:6.20} we compute  
\begin{align}\label{eq:6.23}
\limsup_{\eps\to 0}\frac{\mathcal{F}(\bar{z}_\eps, B_{(1-\theta)\eps}(x_0))}{\gamma_{d-1}\eps^{d-1}} &\le \limsup_{\eps\to 0} \frac{\mathcal{F}(\bar{z}_\eps,B_{(1- 3 \theta)\eps}(x_0))}{\gamma_{d-1}\eps^{d-1}}+  \limsup_{\eps\to 0} \frac{\mathcal{F}(\bar{u}^{\rm surf}_{x_0}, C_{\eps,\theta}(x_0))}{\gamma_{d-1}\eps^{d-1}}\notag\\
&  \le  \limsup_{\eps\to 0} \frac{\mathbf{m}_{\mathcal{F}}(\bar{u}^{\rm surf}_{x_0},B_{(1-3\theta)\eps}(x_0))}{\gamma_{d-1}\eps^{d-1}} + \frac{\beta}{\gamma_{d-1}} \,  \mathcal{H}^{d-1}(C_{1,\theta}(x_0)\cap \Pi_0) \notag \\
&\le  (1-3\theta)^{d-1}\limsup_{\eps\to 0} \frac{\mathbf{m}_{\mathcal{F}}(\bar{u}^{\rm surf}_{x_0},B_{\eps}(x_0))}{\gamma_{d-1}\eps^{d-1}} \\
& +\beta \,(1-(1-4\theta)^{d-1})\,, \notag
\end{align}
where we denote by $\Pi_0$ the hyperplane passing through $x_0$ with normal $\nu_u(x_0)$.

To estimate the remaining term, observe that by rectifiability of $J_u$ and \eqref{eq:24} (i) it holds
\begin{equation}
\begin{split}
\mathop{\lim\sup}_{\eps\to0}\frac{\beta\mathcal{H}^{d-1}(J_{\bar{u}_\eps} \cap C_{\eps,\theta}(x_0))}{\gamma_{d-1}\eps^{d-1}}&\le \lim_{\eps\to0}\frac{\beta\mathcal{H}^{d-1}(J_u \cap C_{\eps,\theta}(x_0))}{\gamma_{d-1}\eps^{d-1}}\\
&=\frac{\beta}{\gamma_{d-1}} \,  \mathcal{H}^{d-1}(C_{1,\theta}(x_0)\cap \Pi_0)=\beta \,(1-(1-4\theta)^{d-1})\,.
\end{split}
\label{eq:stimapersalto}
\end{equation}
With this, using  \ref{assH4} and \eqref{eq:24}$(ii)$ we infer
\begin{equation}
\begin{split}
\mathop{\lim\sup}_{\eps\to0}\frac{\mathcal{F}(\bar{u}_\eps,  C_{\eps,\theta}(x_0))}{\gamma_{d-1}\eps^{d-1}} &\le \mathop{\lim\sup}_{\eps\to0}\frac{\beta\int_{C_{\eps,\theta}(x_0)} (1+  |\nabla \bar{u}_\eps|^{p(x)})\,\mathrm{d}x}{\gamma_{d-1}\eps^{d-1}} +  \mathop{\lim\sup}_{\eps\to0}\frac{\beta\mathcal{H}^{d-1}(J_{\bar{u}_\eps} \cap C_{\eps,\theta}(x_0))}{\gamma_{d-1}\eps^{d-1}} \\
& = \beta(1-(1-4\theta)^{d-1})\,.
\end{split}
\label{eq:6.24}
\end{equation}  
Finally, collecting the estimates in \eqref{eq:6.22}, \eqref{eq:6.23} and \eqref{eq:6.24}, recalling that $\varrho_\eps\to0$ and that $\bar{w}_\eps=u$ in a neighborhood of $\partial B_\eps(x_0)$, we obtain
\begin{align*}
 \displaystyle \lim_{\eps \to 0}\frac{\mathbf{m}_{\mathcal{F}}(u,B_\eps(x_0))}{\gamma_{d-1}\eps^{d-1}} & \le  \limsup_{\eps \to 0}\frac{\mathcal{F}(\bar{w}_\eps, B_\eps(x_0))}{\gamma_{d-1}\eps^{d-1}} \notag \\
 & \le (1+\eta) \, (1-3\theta)^{d-1} \limsup_{\eps \to 0}  \frac{\mathbf{m}_{\mathcal{F}}(\bar{u}^{\rm surf}_{x_0},B_\eps(x_0))}{\gamma_{d-1}\eps^{d-1}} \\ & \ \ \ + 2\beta(1+\eta)(1-(1-4\theta)^{d-1})\,,
\end{align*}
whence \eqref{eq: surf-proof1upper} follows up to passing to $\eta,\theta \to 0$. The proof is concluded.
\EEE

\endproof

\begin{lemma}\label{lem:jump2}
Under the assumptions of Lemma~\ref{lem:jump1},
for $\mathcal{H}^{d-1}$-a.e.\ $x_0 \in J_u$  we have
\begin{align}\label{eq: surf-proof1lower}
  \lim_{\eps \to 0}\frac{\mathbf{m}_{\mathcal{F}}(u,B_\eps(x_0))}{\gamma_{d-1}\eps^{d-1}} \geq  \limsup_{\eps \to 0}\frac{\mathbf{m}_{\mathcal{F}}(\bar{u}^{\rm surf}_{x_0},B_\eps(x_0))}{\gamma_{d-1}\eps^{d-1}}. 
  \end{align}
\end{lemma}
\EEE
\proof
Let $\bar{u}_\eps$ be the sequence of Lemma~\ref{lem:Lemma3fon}, and let $\theta\in(0,1)$, $\eta>0$ be fixed. 
From \eqref{eq:25} it follows that
\begin{equation*}
\begin{split}
0\leq \eps^{-(d-1)} \int_0^1 \mathcal{H}^{d-1}(\{u\neq \bar{u}_\eps\}\cap \partial B_{\sigma\eps}(x_0))\,\mathrm{d}\sigma & = \frac{2}{\eps^{d}} \mathcal{L}^d(\{u\neq \bar{u}_\eps\}\cap B_{\eps}(x_0))\to0
\end{split}
\end{equation*} 
as $\eps\to0$. Then for each $\eps>0$ we can find $\sigma\in(1-4\theta,1-3\theta)$ such that
\begin{equation}
\begin{split}
&\mu(\partial B_{\sigma\eps}(x_0)) = \mathcal{H}^{d-1}(\partial B_{\sigma\eps}(x_0)\cap (J_{\bar{u}_\eps}\cup J_u))=0\,,\quad \mbox{ for all $\eps>0$,}  \\
&\lim_{\eps\to0} \eps^{-d} \mathcal{H}^{d-1}(\{u\neq \bar{u}_\eps\}\cap \partial B_{\sigma\eps}(x_0))
=0\,.
\end{split}
\label{eq:6.25}
\end{equation}

We consider a sequence $z_\eps \in GSBV^{p(\cdot)}(B_{\sigma\eps}(x_0);\R^m)$ with $z_\eps = u$ in a neighborhood of $\partial B_{\sigma\eps}(x_0)$ and
\begin{align}\label{eq:repr1+bis}
\mathcal{F}\big(z_\eps,B_{\sigma\eps}(x_0)\big) \le \mathbf{m}_{\mathcal{F}}\big(u,B_{\sigma\eps}(x_0)\big) + \gamma_{d-1}\eps^{d}\,.
\end{align}
We extend $z_\eps$ to a function in $GSBV^{p(\cdot)}(B_\eps(x_0);\R^m)$ by setting  $z_\eps = \bar{u}_\eps$ outside $B_{\sigma\eps}(x_0)$. By applying Lemma~\ref{lemma: fundamental estimate} with $u$ and $v$ replaced by $z_\eps$ and $\bar{u}^{\rm surf}_{x_0}$, respectively, and the same choice for the sets $D'_{\eps,x_0}$, $D''_{\eps,x_0}$ and $E_{\eps,x_0}$ as in Lemma~\ref{lem:jump1}, 
we find $\bar{w}_\eps\in GSBV^{p(\cdot)}(B_\eps(x_0);\R^m)$ such that $\bar{w}_\eps=\bar{u}^{\rm surf}_{x_0}$ on $B_\eps(x_0)\backslash B_{(1-\theta)\eps}$ and
\begin{equation*}
\begin{split}
\mathcal{F}(\bar{w}_\eps, B_\eps(x_0)) & \leq (1+\eta) \left(\mathcal{F}(z_\eps, B_{(1-\theta)\eps}(x_0)) + \mathcal{F}(\bar{u}^{\rm surf}_{x_0}, C_{\eps,\theta}(x_0))\right) \\
& + M \int_{(B_{(1-\theta)\eps}(x_0) \setminus B_{(1-2\theta)\eps}(x_0))}\left(\frac{|z_\eps-\bar{u}^{\rm surf}_{x_0}|}{\eps}\right)^{p(x)}\,\mathrm{d}x +\eta\mathcal{L}^d(B_\eps(x_0))\,.
\end{split}
\end{equation*}
We notice that, as a consequence of the choice of $\sigma$, $z_\eps = \bar{u}_\eps$ outside $B_{(1-3\theta)}(x_0)$. Then, by virtue of $\eqref{eq:24}_3$, we can find a non-negative sequence $\varrho_\eps$, vanishing as $\eps\to0$, such that
\begin{equation}
\begin{split}
\mathcal{F}(\bar{w}_\eps, B_\eps(x_0)) & \leq (1+\eta) \left(\mathcal{F}(z_\eps, B_{(1-\theta)\eps}(x_0)) + \mathcal{F}(\bar{u}^{\rm surf}_{x_0}, C_{\eps,\theta}(x_0))\right) \\
& + \eps^{d-1}\varrho_\eps +\eta\gamma_d\eps^d\,.
\end{split}
\label{eq:6.28}
\end{equation}
We now estimate each term in the right hand side of \eqref{eq:6.28}. Taking into account {\rm \ref{assH1}}, {\rm \ref{assH4}}, \eqref{eq:repr1+bis}, the fact that $z_\eps = \bar{u}_\eps$ on $B_\eps(x_0)\backslash B_{\sigma\eps}(x_0)$ and the choice of $\sigma$, we get
\begin{equation}
\begin{split}
\mathcal{F}(z_\eps, B_{(1-\theta)\eps}(x_0))\leq & \mathbf{m}_{\mathcal{F}}\big(u,B_{\sigma_\eps}(x_0)\big) + \gamma_{d-1}\eps^{d} + \beta \mathcal{H}^{d-1}\left((\{\bar{u}_\eps\neq u\}\cup J_u\cup J_{\bar{u}_\eps})\cap \partial B_{\sigma_\eps}(x_0)\right) \\
& + \mathcal{F}(\bar{u}_\eps, C_{\eps,\theta}(x_0))\,. 
\end{split}
\label{eq:6.29}
\end{equation}
Now, with \eqref{eq:6.24}, \eqref{eq:6.25} and $\sigma\leq (1-3\theta)$ we then obtain
\begin{equation}
\begin{split}
\mathop{\lim\sup}_{\eps\to0} \frac{\mathcal{F}(z_\eps, B_{(1-\theta)\eps}(x_0))}{\gamma_{d-1}\eps^{d-1}} & \leq \mathop{\lim\sup}_{\eps\to0} \frac{ \mathbf{m}_{\mathcal{F}}\big(u,B_{\sigma\eps}(x_0)\big)}{\gamma_{d-1}\eps^{d-1}}+\beta(1-(1-4\theta)^{d-1}) \\
& \leq (1-3\theta)^{d-1}\mathop{\lim\sup}_{\eps\to0} \frac{ \mathbf{m}_{\mathcal{F}}\big(u,B_{\eps}(x_0)\big)}{\gamma_{d-1}\eps^{d-1}}+\beta(1-(1-4\theta)^{d-1})\,,
\end{split}
\label{eq:6.30}
\end{equation} 
and, as already proven in \eqref{eq:6.23},
\begin{equation}
\mathop{\lim\sup}_{\eps\to0}\frac{\mathcal{F}(\bar{u}^{\rm surf}_{x_0}, C_{\eps,\theta}(x_0))}{\gamma_{d-1}\eps^{d-1}} \leq \beta(1-(1-4\theta)^{d-1})\,.
\label{eq:6.31}
\end{equation} 
Collecting the estimates \eqref{eq:6.28}, \eqref{eq:6.30}, \eqref{eq:6.31} and using $\varrho_\eps\to0$ we infer
\begin{equation*}
\mathop{\lim\sup}_{\eps\to0} \frac{\mathcal{F}(\bar{w}_\eps, B_\eps(x_0))}{\gamma_{d-1}\eps^{d-1}} \leq (1+\eta) \left((1-3\theta)^{d-1}\mathop{\lim\sup}_{\eps\to0} \frac{ \mathbf{m}_{\mathcal{F}}\big(u,B_{\eps}(x_0)\big)}{\gamma_{d-1}\eps^{d-1}}+2\beta(1-(1-4\theta)^{d-1})\right)\,.
\end{equation*}
Finally, since $\bar{w}_\eps=\bar{u}^{\rm surf}_{x_0}$ in a neighborhood of $\partial B_\eps(x_0)$, and using the arbitrariness of $\eta$ and $\theta$, we derive
\begin{equation*}
\begin{split}
 \limsup_{\eps \to 0}\frac{\mathbf{m}_{\mathcal{F}}(\bar{u}^{\rm surf}_{x_0},B_\eps(x_0))}{\gamma_{d-1}\eps^{d-1}} & \leq \mathop{\lim\sup}_{\eps\to0} \frac{\mathcal{F}(\bar{w}_\eps, B_\eps(x_0))}{\gamma_{d-1}\eps^{d-1}} \\
& \leq \mathop{\lim\sup}_{\eps\to0} \frac{\mathbf{m}_{\mathcal{F}}({u},B_{\eps}(x_0))}{\gamma_{d-1}\eps^{d-1}} \\
&= \lim_{\eps\to0} \frac{\mathbf{m}_{\mathcal{F}}({u},B_{\eps}(x_0))}{\gamma_{d-1}\eps^{d-1}}\,.
\end{split}
\end{equation*} 
The proof of \eqref{eq: surf-proof1lower} is concluded.
\endproof
\EEE

\section{$\Gamma$\hbox{-}convergence} \label{sec:gammaconv}


In this section, we present a general $\Gamma$-convergence result for functionals $\mathcal{F}\colon GSBV^{p(\cdot)}(\Omega;\R^m) \times \mathcal{A}(\Omega) \to [0,+\infty)$ of the form 
\begin{align}\label{eq:energy}
\mathcal{F}(u,A) = \int_{A} f\big(x, \nabla u(x) \big) \, {\rm d}x +  \int_{J_u \cap A} g\big(x,  [u](x),  \nu_u(x)\big) \, {\rm d} \mathcal{H}^{d-1} (x) 
\end{align} 
for each $u \in GSBV^{p(\cdot)}(\Omega;\R^m)$ and each $A \in \mathcal{A}(\Omega)$, where $[u](x):=u^+(x)-u^-(x)$ (we refer the reader to \cite{DalMaso:93} for an exhaustive treatment of the topic). To formulate the result, we adopt the notation of Section~\ref{sec: main} and define the minimization problems $\mathbf{m}_{ \mathcal{F}}(u,A)$ and the functions $\ell_{x_0,u_0,\xi}$ and $u_{x_0,a,b,\nu}$ as in \eqref{eq: general minimizationsgbv}, \eqref{eq: elastic competitor} and \eqref{eq: jump competitor}, respectively.

Let $0<\alpha\leq\beta<+\infty$ and $1\leq c<+\infty$ be fixed constants. We assume that $f\colon \R^d{\times} \R^{m{\times}d}\to [0,+\infty)$ satisfies the following assumptions: 
\begin{enumerate}[font={\normalfont},label={($f$\arabic*)}]
\item(measurability) $f$ is Borel measurable on $\R^d{\times} \R^{m{\times}d}$; \label{ass-f1}
\item(lower and upper bound) for every $x \in \R^d$ and every $\xi \in \R^{m{\times}d}$,
$$\alpha |\xi|^{p(\cdot)} \leq f(x,\xi)\leq \beta(1+|\xi|^{p(\cdot)})\,,$$
\label{eq:bound1}
\end{enumerate}
and that $g\colon \R^d{\times}\R^m_0{\times} \Sph^{d-1} \to [0,+\infty)$ complies with the following assumptions:
\begin{enumerate}[font={\normalfont},label={($g$\arabic*)}]
\item(measurability) $g$ is Borel measurable on $\R^d{\times}\R^m_0{\times} \Sph^{d-1}$; \label{ass-g1}
\item (estimate for $c|\zeta_1|\le|\zeta_2|$) for every $x\in \R^d$ and every $\nu \in \Sph^{d-1}$
we have
\begin{equation*}
g(x,\zeta_1,\nu) \leq \,g(x,\zeta_2,\nu)
\end{equation*}
for every $\zeta_1$, $\zeta_2\in \R^m_0$ with $c|\zeta_1|\le|\zeta_2|$; \label{ass-g4}
\item(lower and upper bound) for every $x\in \R^d$, $\zeta\in \R^m_0$, and $\nu \in \Sph^{d-1}$
\begin{equation*}
\alpha \le g(x,\zeta,\nu) \le \beta\,;
\end{equation*}\label{eq:bound2}
\item(symmetry) for every $x\in \R^d$, $\zeta\in \R^m_0$, and $\nu \in \Sph^{d-1}$
$$g(x,\zeta,\nu) = g(x,-\zeta,-\nu).$$ \label{ass-g7}
\end{enumerate}
For future reference, we also introduce the property
\begin{enumerate}[font={\normalfont},label={($g5$)}]
\item (estimate for $|\zeta_1|\le|\zeta_2|$) for every $x\in \R^d$ and every $\nu \in \Sph^{d-1}$ we have
$$
g(x,\zeta_1,\nu) \leq c \,g(x,\zeta_2,\nu)
$$
for every $\zeta_1$, $\zeta_2 \in \R^m_0$ with $|\zeta_1|\le |\zeta_2|$. \label{ass-g3}
\end{enumerate}
Notice that assumption \ref{eq:bound2} implies \ref{ass-g3} with $c:=\frac{\beta}{\alpha}$.


The first main result is the following. 

\begin{theorem}[$\Gamma$-convergence]\label{th: gamma}
 Let $\Omega \subset \R^d$ be open.  Let $(f_j)_j$ and $(g_j)_j$ be sequences of functions satisfying {\rm\ref{ass-f1}}-{\rm\ref{eq:bound1}} and {\rm\ref{ass-g1}}-{\rm\ref{ass-g7}}, respectively.  Let $\mathcal{F}_j \colon GSBV^{p(\cdot)}(\Omega;\R^m) \times \mathcal{A}(\Omega) \to [0,+\infty)$ be the corresponding sequence of functionals given in \eqref{eq:energy}.  Then, there exists a functional $\mathcal{F}_\infty\colon  GSBV^{p(\cdot)}(\Omega;\R^m)\times \mathcal{A}(\Omega) \to [0,+\infty)$ and a subsequence (not relabeled) such that
$$\mathcal{F}_\infty(\cdot,A) =\Gamma\text{-}\lim_{j \to \infty} \mathcal{F}_j(\cdot,A) \ \ \ \ \text{with respect to convergence in measure on $A$} $$
for all $A \in  \mathcal{A}(\Omega) $. Moreover, for every $u\in GSBV^{p(\cdot)}(\Omega;\R^m)$ and  $A\in \mathcal{A}(\Omega)$ we have that
\begin{align}\label{eq: representation}
\mathcal{F}_\infty(u,A)= \int_A f_{\infty}\big(x,u(x),   \nabla u(x)   \big)\, {\rm d}x +\int_{J_u\cap A}g_{\infty}(x,u^+(x), u^-(x),\nu_u(x))\, {\rm d}\mathcal{H}^{d-1}(x)\,,
\end{align}
where $f_\infty=f_{\infty}(x_0,u_0,\xi)$ is given by \eqref{eq:fdef-gsbv} for all $x_0 \in \Omega$, $u_0 \in \R^m$, $\xi \in \mathbb{R}^{m\times d}$, and $g_\infty=g_{\infty}(x_0,a,b,\nu)$ is given by \eqref{eq:gdef-gsbv} for all $  x_0  \in \Omega$,  $a,b \in \R^m$, and $\nu \in \mathbb{S}^{d-1}$. 
\end{theorem}

We will prove the compactness of $\Gamma$-convergence via the localization technique for $\Gamma$-convergence (see \cite[Ch. 14--20]{DalMaso:93} for the general method), where the main ingredient is the fundamental estimate in $GSBV^{p(\cdot)}$, proven with Lemma~\ref{lemma: fundamental estimate}. The representation \eqref{eq: representation} in terms of the densities $f_\infty$ and $g_\infty$ then will follow by the integral representation result of Theorem~\ref{thm: int-representation-gsbv}. Indeed, since each $\mathcal{F}_j$ is invariant under translations of $u$, then also $\mathcal{F}_\infty$, as $\Gamma$-limit, satisfies the same property. Thus, from Theorem~\ref{thm: int-representation-gsbv}, in particular \eqref{eq:fdef-gsbv}--\eqref{eq:gdef-gsbv}, we infer that $f_\infty=f_{\infty}(x_0,\xi)$, $g_\infty=g_{\infty}(x_0,a-b,\nu)$ so that $\mathcal{F}_\infty$ has the form
\begin{align}\label{eq: representationbis}
\mathcal{F}_\infty(u,A)= \int_A f_{\infty}\big(x,   \nabla u(x)   \big)\, {\rm d}x +\int_{J_u\cap A}g_{\infty}(x,[u](x),\nu_u(x))\, {\rm d}\mathcal{H}^{d-1}(x)\,,
\end{align}
and the densities $f_\infty, g_\infty$ can be computed as
\begin{align}
f_\infty(x_0,\xi) & = \limsup_{\eps \to 0} \frac{\mathbf{m}_{\mathcal{F}}(\ell_{0,0,\xi},B_\eps(x_0))}{\gamma_d\eps^{d}}\,, \label{eq:fdef-gsbvbis} \\
g_\infty(x_0,\zeta,\nu) & = \limsup_{\eps \to 0} \frac{\mathbf{m}_{\mathcal{F}}(u_{x_0,\zeta,0,\nu},B_\eps(x_0))}{\gamma_{d-1}\eps^{d-1}} \,,  \label{eq:gdef-gsbvbis}
\end{align}
for all $x_0 \in \Omega$, $\xi \in \mathbb{R}^{m \times d}$, $\zeta\in\R^m$ and $\nu\in\mathbb{S}^{d-1}$.

For our purposes, it will be useful to consider functionals $\mathcal{I}: L^0(\Omega;\R^m)\times\mathcal{A}(\Omega)\to[0,+\infty]$ defined as
\begin{equation}
\mathcal{I}(u,A):=
\begin{cases}
\int_A f(x,\nabla u(x))\,\mathrm{d}x\,, & u\in GSBV^{p(\cdot)}(A;\R^m)\,, \\
+\infty & \mbox{ otherwise. }
\end{cases}
\label{eq:partialenergy}
\end{equation}

We recall a result concerning the existence of suitable truncations of a measurable function $u$ by which functionals $\mathcal{F}$ as above almost decrease (see \cite[Lemma~4.1]{CDMSZ}). For our purposes, the statement below is formulated in the $p(\cdot)$-setting, and since the adaptation of the original proof requires only minor changes, we omit the details.

\begin{lemma}\label{lem:cdmsz}
Let $\mathcal{F}$ be as in \eqref{eq:energy}, where we assume that $f$ satisfies {\rm\ref{ass-f1}}-{\rm\ref{eq:bound1}} and $g$ satisfies {\rm\ref{ass-g1}},{\rm\ref{ass-g4}}, {\rm\ref{ass-g7}} and {\rm\ref{ass-g3}}. Let $\mathcal{I}$ be as in \eqref{eq:partialenergy}. Let $\eta,\lambda>0$. Then there exists $\mu>\lambda$ depending on $\eta$, $\lambda$, $\alpha$, $\beta$, $c$ such that the following holds: for every open set $A\subset\Omega$ and for every $u\in L^0(\R^d, \R^m)$ such that ${u}|_A\in GSBV^{p(\cdot)}(A,\R^m)$, there exists $\hat{u}\in L^\infty(\R^d,\R^m)$ such that $\hat{u}|_A\in SBV^{p(\cdot)}(A,\R^m)$ and 
\begin{itemize}
\item[$(i)$] $|\hat{u}|\leq \mu$ on $\R^d$;
\item[$(ii)$] $\hat{u}=u$ $\mathcal{L}^d$-a.e. in $\{|u|\leq \lambda\}$;
\item[$(iii)$] $\mathcal{F}(\hat{u},A)  \leq (1+\eta) \mathcal{F}(u,A) + \beta \mathcal{L}^d (A\cap\{|u|\geq \lambda\})\,.$
\end{itemize}
Moreover, there exists $\hat{v}$ with the same properties of $\hat{u}$ such that $(iii)$ holds for the functional $\mathcal{I}$ with $\hat{v}$ in place of $\hat{u}$.
\end{lemma}


Let $(\mathcal{F}_j)_j$ be a sequence of functionals of the  form \eqref{eq:energy}. We start by proving some  properties of the $\Gamma$-liminf and $\Gamma$-limsup with respect to the topology of the convergence in measure. To this end, we define 
\begin{align}\label{eq: liminf-limsup}
\mathcal{F}_\infty'(u,A)&:=\Gamma-\liminf_{n \to \infty} \mathcal{F}_j(u,A)   = \inf \big\{ \liminf_{j \to \infty} \mathcal{F}_j(u_j,A):   \  u_j \to  u \hbox{ in measure on }  A  \big\}, \notag \\
\mathcal{F}_\infty''(u,A) &:= \Gamma-\limsup_{n \to \infty} \mathcal{F}_j(u,A)  = \inf \big\{  \limsup_{j \to \infty} \mathcal{F}_j(u_j,A):   \ u_j \to  u \hbox{ in measure on }  A \big\}
\end{align} 
for all $u \in GSBV^{p(\cdot)}(\Omega;\R^m)$ and $A \in \mathcal{A}(\Omega)$.

\begin{lemma}[Properties of $\Gamma$-liminf and  $\Gamma$-limsup]\label{eq: liminflimsup-prop}
 Let $\Omega \subset \R^d$  be an open set, and \\ $\mathcal{F}_j\colon  GSBV^{p(\cdot)}(\Omega;\R^m)\times \mathcal{A}(\Omega) \to [0,\infty)$ be a sequence of functionals as in {\rm(\ref{eq:energy})}, where we assume that $f_j$ and $g_j$ comply with {\rm\ref{ass-f1}}-{\rm\ref{eq:bound1}} and {\rm\ref{ass-g1}}-{\rm\ref{ass-g7}}, respectively, for all $j\in\mathbb{N}$. Define $\mathcal{F}_\infty'$ and $\mathcal{F}_\infty''$ as in \eqref{eq: liminf-limsup}, and write, for brevity, $$ \mathcal{G}(u,A):= \int_A |\nabla u|^{p(\cdot)}\,\mathrm{d}x+ \mathcal{H}^{d-1}(J_u \cap A)\,.$$ Then we have
\begin{align}\label{eq: infsup0}
{\rm (i)} & \ \ 
\mathcal{F}_\infty'(u,A) \le \mathcal{F}_\infty'(u,B), \ \ \ \ \ \  \mathcal{F}_\infty''(u,A) \le \mathcal{F}_\infty''(u,B) \ \ \ \text{ whenever } A \subset B, \notag \\
{\rm (ii)} & \ \ 
 \alpha \mathcal{G}(u,A)  \le \mathcal{F}_\infty'(u,A)    \le  \mathcal{F}_\infty''(u,A) \le \beta\mathcal{G}(u,A) + \beta\mathcal{L}^d(A), \notag \\
{\rm (iii)} & \ \ 
\mathcal{F}_\infty'(u,A)  = \sup\nolimits_{B \subset \subset A} \mathcal{F}_\infty'(u,B),  \ \ \ \ \mathcal{F}_\infty''(u,A)  = \sup\nolimits_{B \subset \subset A} \mathcal{F}_\infty''(u,B)  \ \ \  \text{ whenever } A \in \mathcal{A}(\Omega),  \notag \\
{\rm (iv)} & \ \ 
\mathcal{F}_\infty'(u,A\cup B) \le  \mathcal{F}_\infty'(u,A) + \mathcal{F}_\infty'(u,B), \notag \\
& \ \   \mathcal{F}_\infty''(u,A\cup B) \le  \mathcal{F}_\infty''(u,A) + \mathcal{F}_\infty''(u,B)  \ \ \  \text{ whenever } A,B \in \mathcal{A}(\Omega),
\end{align}
where  $\alpha, \beta$ have been introduced in {\rm\ref{eq:bound1}} and {\rm\ref{eq:bound2}}.
\end{lemma}

\begin{proof}
The monotonicity property (i) follows from the fact that $\mathcal{F}_j(u,\cdot)$ are measures. The upper bound in (ii) can be inferred choosing the constant sequence $u_j=u$ in \eqref{eq: liminf-limsup} and taking into account the upper bounds in {\rm\ref{eq:bound1}} and {\rm\ref{eq:bound2}}. For what concerns the lower bound in (ii), we consider an (almost) optimal sequence $(v_j)_j$ in \eqref{eq: liminf-limsup}. Then, with the lower bounds in {\rm\ref{eq:bound1}} and {\rm\ref{eq:bound2}} we get
\begin{equation*}
 \sup_{j \in \N} \alpha \mathcal{G}(v_j, A)<+\infty\,.
\end{equation*}
Now, since $v_j\to u$ in measure on $A$, by arguing as in the proof of Lemma~\ref{lemma: F=m*} and exploiting the lower semicontinuity inequalities
\begin{equation*}
\begin{split}
\int_A |\nabla u(x)|^{p(x)}\,\mathrm{d}x &\leq \displaystyle\mathop{\lim\inf}_{j\to+\infty} \int_A|\nabla v_j(x)|^{p(x)}\,\mathrm{d}x <+\infty\,, \\
\mathcal{H}^{d-1}(J_u\cap A) & \leq \displaystyle\mathop{\lim\inf}_{j\to+\infty} \mathcal{H}^{d-1}(J_{v_j}\cap A)\,,
\end{split}
\end{equation*}
we easily obtain the lower bound.

In order to prove (iii) and (iv), we preliminary show that for every $U, V$ and $W$ open subsets of $\Omega$, with $V\subset\subset W \subset\subset U$, we have
\begin{equation}
\mathcal{F}_\infty'(u,U) \leq \mathcal{F}_\infty'(u,W) +  \mathcal{F}_\infty'(u,U\backslash \overline{V})\,,\quad \mathcal{F}_\infty''(u,U) \leq \mathcal{F}_\infty''(u,W) +  \mathcal{F}_\infty''(u,U\backslash \overline{V})\,.
\label{eq:4.35FPS}
\end{equation}
We confine ourselves to the proof of the first assertion in \eqref{eq:4.35FPS}, the other one being similar. Let $(u_j)_j$ and $(v_j)_j$ be sequences in $GSBV^{p(\cdot)}(\Omega;\R^m)$ converging in measure to $u$ on $W$ and $U\backslash\overline{V}$, respectively, such that
\begin{equation}
\mathcal{F}_\infty'(u,W) = \mathop{\lim\inf}_{j\to+\infty} \mathcal{F}_j(u_j,W)\,, \quad \mathcal{F}_\infty'(u,U\backslash \overline{V}) = \mathop{\lim\inf}_{j\to+\infty} \mathcal{F}_j(v_j,U\backslash \overline{V})\,. 
\label{eq:4.36FPS}
\end{equation}
We may assume, up to passing to a not relabeled subsequence, that each liminf above is a limit. We fix $\eta\in(0,1)$ and $\lambda>0$ such that
\begin{equation}
\beta \mathcal{L}^d (U\cap\{|u|\geq \lambda\})<\eta\,.
\label{eq:boundonu}
\end{equation}
By virtue of Lemma~\ref{lem:cdmsz}, there exists $\mu>\lambda$ such that, for every $k\geq1$ we can find $\hat{u}_{j_k}\in SBV^{p(\cdot)}(W;\R^m)\cap L^\infty(W;\R^m)$, with $|\hat{u}_{j_k}|\leq\mu$, $\hat{v}_{j_k}\in SBV^{p(\cdot)}(U\backslash \overline{V};\R^m)\cap L^\infty(U\backslash \overline{V};\R^m)$, with $|\hat{v}_{j_k}|\leq\mu$, such that $\hat{u}_{j_k}=u_j$ $\mathcal{L}^d$-a.e. in $W\cap\{|u_j|\leq \lambda\}$, $\hat{v}_{j_k}=v_j$ $\mathcal{L}^d$-a.e. in $(U\backslash \overline{V})\cap\{|v_j|\leq \lambda\}$ and 
\begin{equation}
\begin{split}
\mathcal{F}_j(\hat{u}_{j_k},W) & \leq \left(1+\eta\right) \mathcal{F}_j(u_j,W) + \beta \mathcal{L}^d (W\cap\{|u_j|\geq \lambda\})\,, \\
\mathcal{F}_j(\hat{v}_{j_k},U\backslash \overline{V}) & \leq \left(1+\eta\right) \mathcal{F}_j(v_j,U\backslash \overline{V}) + \beta \mathcal{L}^d ((U\backslash \overline{V})\cap\{|v_j|\geq \lambda\})\,.
\end{split}
\label{eq:stimette}
\end{equation}
We apply Lemma~\ref{lemma: fundamental estimate} with $\eta$ above, 
$D'':=W$, $E:=U\backslash \overline{V}$, $u=\hat{u}_{j_k}$, $v=\hat{v}_{j_k}$, for some $D'$ with $V\subset\subset D'\subset\subset W$. Note that $W\backslash D'\subset U\backslash \overline{V}$. We then find a function $\hat{w}_{j_k}\in SBV^{p(\cdot)}(U;\R^m)\cap L^\infty(U;\R^m)$ such that
\begin{equation}
\begin{split}
\mathcal{F}_j (\hat{w}_{j_k}, U)  \le  &  \left(1+\eta\right)\big(\mathcal{F}_j(\hat{u}_{j_k},W)  + \mathcal{F}_j(\hat{v}_{j_k},U\backslash \overline{V}) \big) \\
& + M \int_{W\backslash D'}\left(\frac{|\hat{u}_{j_k}-\hat{v}_{j_k}|}{\delta}\right)^{p(x)}\,\mathrm{d}x + \eta 
\mathcal{L}^d(D' \cup E)\,.
\end{split}
\label{eq:stimafond}
\end{equation}
Note that, by the dominated convergence in measure, $\hat{u}_{j_k}-\hat{v}_{j_k}\to 0$ in $L^{p(\cdot)}(W\backslash D';\R^m)$ as $k\to+\infty$. Moreover, recalling \eqref{eq: assertionfund}(ii), we have that $\hat{w}_{j_k}\to u$ in measure on $U$ as $k\to+\infty$. By a diagonal argument this implies, in particular, that
\begin{equation}
\mathcal{F}_\infty'(u,U) \leq \mathop{\lim\inf}_{k\to+\infty} \mathcal{F}_{j_k} (\hat{w}_{j_k}, U)\,.
\label{eq:semicontinuo}
\end{equation}
Note also that, from \eqref{eq:boundonu} and the convergence in measure of both $u_{j_k}$ and $v_{j_k}$ to $u$, we have 
\begin{equation*}
\mathcal{L}^d (W\cap\{|u_{j_k}|\geq \lambda\})<\eta \quad \mbox{ and } \quad \mathcal{L}^d ((U\backslash \overline{V})\cap\{|v_{j_k}|\geq \lambda\})<\eta
\end{equation*}
for $k$ large enough. Then, combining \eqref{eq:stimette} with \eqref{eq:stimafond}, \eqref{eq:4.36FPS} and passing to the limit as $k\to+\infty$, and then letting $\eta\to0^+$, assertion \eqref{eq:4.35FPS} follows.

We now prove the inner regularity of $\mathcal{F}_\infty'$, the first property in (iii). Combining \eqref{eq: infsup0}(ii) and \eqref{eq:4.35FPS} we find
\begin{equation*}
 \mathcal{F}_\infty'(u,U) \le  \mathcal{F}_\infty'(u,W)+ \beta\mathcal{G}(u,U\backslash\overline{V}) + \beta\mathcal{L}^d(U\backslash\overline{V})\,.
\end{equation*}
Now, we can choose $V\subset\subset U$ and $U$ in such a way that $\mathcal{L}^d(U\backslash\overline{V})$ and $\mathcal{G}(u,U\backslash\overline{V})$ be arbitrarily small, and recalling that $\mathcal{F}_\infty'(u,\cdot)$ is an increasing set function by \eqref{eq: infsup0}(i), we obtain \eqref{eq: infsup0}(iii) for $\mathcal{F}_\infty'$. The proof of the analogous property for $\mathcal{F}_\infty''$ is similar.

We conclude by showing property (iv) for $\mathcal{F}_\infty'$. First, we note that it is not restrictive to assume that $A \cap B \neq \emptyset$, otherwise the inequalities in (iv) are straightforward. It is well known (see, e.g., \cite[Proof of Lemma 5.2]{AmbrosioBraides}) that given $\eta >0$, one can choose in $\Omega$ open sets $U \subset \subset U' \subset \subset A$ and $V \subset \subset V' \subset \subset B$ such that  $U' \cap V' = \emptyset$, and  $\mathcal{G}(u,(A\cup B) \setminus (\overline{U \cup V} ))  + \mathcal{L}^d((A\cup B) \setminus (\overline{U \cup V} ))  \le \eta$. Then using, \eqref{eq: infsup0}(i),(ii) and \eqref{eq:4.35FPS}  we get 
\begin{align*}
\mathcal{F}_\infty'(u,A \cup B) & \le  \mathcal{F}_\infty'(u,U' \cup V') + \mathcal{F}_\infty'(u, (A\cup B) \setminus \overline{U \cup V} ) \le  \mathcal{F}_\infty'(u,U') +  \mathcal{F}_\infty'(u,V') + \beta  \eta  \\ & \le \mathcal{F}_\infty'(u,A) +  \mathcal{F}_\infty'(u,B) + \beta  \eta,
\end{align*}
where we also used $\mathcal{F}_\infty'(u,U' \cup V') \le \mathcal{F}_\infty'(u,U') +  \mathcal{F}_\infty'(u,V')$ which holds due to $U' \cap V' = \emptyset$. Since $\eta$ was arbitrary, the statement follows. 
\end{proof}

We can now prove Theorem \ref{th: gamma}.

\begin{proof}[Proof of Theorem \ref{th: gamma}]  
First, we prove the existence of the $\Gamma$-limit by applying an abstract compactness result for $\bar{\Gamma}$-convergence, see \cite[Theorem 16.9]{DalMaso:93}. This implies the existence of an increasing sequence of integers $(j_k)_k$ such that $\mathcal{F}_\infty'$ and $\mathcal{F}_\infty''$ defined in \eqref{eq: liminf-limsup} with respect to  $(j_k)_k$ satisfy
$$(\mathcal{F}_\infty')_-(u,A) = (\mathcal{F}_\infty'')_-(u,A)  $$
for all $u \in GSBV^{p(\cdot)}(\Omega;\R^m)$ and $A \in \mathcal{A}(\Omega)$, where $(\mathcal{F}_\infty')_-$ and $(\mathcal{F}_\infty'')_-$ denote the inner regular envelopes of $\mathcal{F}_\infty'$ and $\mathcal{F}_\infty''$, respectively.   By   \eqref{eq: infsup0}(iii) we know that $\mathcal{F}_\infty'$ and $\mathcal{F}_\infty''$ are inner regular, and thus they both coincide with their respective inner regular envelopes.   This shows that  the $\Gamma$-limit, denoted by $\mathcal{F}_\infty:= \mathcal{F}_\infty' = \mathcal{F}_\infty''$, exists for all $u \in GSBV^{p(\cdot)}(\Omega;\R^m)$ and all $A \in \mathcal{A}(\Omega)$. 

 We now check that $\mathcal{F}_\infty$ satisfies assumptions \ref{assH1}--\ref{assH4} of the integral representation result, Theorem~\ref{thm: int-representation-gsbv}. First, the definition in \eqref{eq: liminf-limsup} and the locality of each $\mathcal{F}_j$ show that $\mathcal{F}_\infty(\cdot, A)$ is local according to \ref{assH3} for any $A \in \mathcal{A}(\Omega)$. 
Moreover, $\mathcal{F}_\infty(\cdot,A)$ complies with \ref{assH2} for any $A \in \mathcal{A}(\Omega)$ in view of  \cite[Remark 16.3]{DalMaso:93}. Now, since $\mathcal{F}_\infty$ is increasing, superadditive, inner regular (see \cite[Proposition 16.12 and Remark 16.3]{DalMaso:93}) and subadditive by \eqref{eq: infsup0}(iv), the De  Giorgi-Letta criterion (see \cite[Theorem 14.23]{DalMaso:93}) ensures that $\mathcal{F}_\infty(u,\cdot)$ can be extended to a Borel measure. Thus, also \ref{assH1} is satisfied. Eventually, by \eqref{eq: infsup0}(ii) we get \ref{assH4}. 
Therefore, we can conclude that $\mathcal{F}_\infty$ admits a representation of the form \eqref{eq: representation}. 
\end{proof}

\section{Identification of the $\Gamma$-limit}\label{sec:identification}

In this section we identify the structure of the $\Gamma$-limit provided by Theorem~\ref{th: gamma}, by showing a \emph{separation of scales} effect; i.e., that there is no interaction between the bulk and surface densities, as $f_\infty$ is only determined by $(f_j)_j$ and $g_\infty$ is only determined by $(g_j)_j$.

We assume that $f\colon \R^d{\times} \R^{m{\times}d}\to [0,+\infty)$ satisfies \ref{ass-f1}, \ref{eq:bound1} and the following: for every $x \in \R^d$ and every $\xi \in \R^{m{\times}d}$,
\begin{enumerate}[font={\normalfont},label={($f3$)}]
\item (continuity in $\xi$) for every $x \in \R^d$ we have
\begin{equation*}
|f(x,\xi_1)-f(x,\xi_2)| \leq \omega_1(|\xi_1-\xi_2|)\big(1+f(x,\xi_1)+f(x,\xi_2)\big)
\end{equation*}
for every $\xi_1$, $\xi_2 \in \R^{m{\times}d}$;\label{eq:ass-f3}
\end{enumerate}
 and that $g\colon \R^d{\times}\R^m_0{\times} \Sph^{d-1} \to [0,+\infty)$ satisfies {\rm\ref{ass-g1}}, {\rm\ref{ass-g4}},  {\rm\ref{eq:bound2}}, {\rm\ref{ass-g7}} and complies with
\begin{enumerate}[font={\normalfont},label={($g6$)}]
\item (continuity in $\zeta$) for every $x\in \R^d$ and every $\nu \in \Sph^{d-1}$ we have
\begin{equation*}
|g(x,\zeta_2,\nu)-g(x,\zeta_1,\nu)|\leq \omega_2(|\zeta_1-\zeta_2|)\big(g(x,\zeta_1,\nu)+g(x,\zeta_2,\nu)\big)
\end{equation*}
for every $\zeta_1$, $\zeta_2\in \R^m_0$, where $\omega_2\colon  [0,+\infty) \to [0,+\infty)$ is a nondecreasing continuous function such that $\omega_2(0)=0$. \label{ass-g2}
\end{enumerate}

%
%

\EEE

\subsection{Identification of the bulk density}

We start with the identification of the bulk density. To do this, we restrict functionals $\mathcal{F}$ as in \eqref{eq:energy} to Sobolev functions $W^{1,p(\cdot)}(\Omega;\R^m)$. Indeed, since every Sobolev function has a $\mathcal{H}^{d-1}$-negligible jump set we have
\begin{equation}
\mathcal{F}(u,A)= \int_\Omega f(x,\nabla u)\,\mathrm{d}x\,,\quad \mbox{ for all } u\in W^{1,p(\cdot)}(\Omega;\R^m)\,.
\label{eq:restricSob}
\end{equation}
We set, for every $\xi\in\R^{m\times d}$,
\begin{equation}
\bar{\ell}_\xi:=\ell_{0,0,\xi}\,,
\label{eq:affinecomp}
\end{equation}
where $\ell_{x_0,u_0,\xi}$ is defined as in \eqref{eq: elastic competitor}. In analogy to \eqref{eq: general minimizationsgbv}, for every $u\in W^{1,p(\cdot)}(\Omega;\R^m)$ and $A\in\mathcal{A}(\Omega)$ we define
\begin{equation}\label{eq: general minimizationsSobolev}
\mathbf{m}_{\mathcal{F}}^{1,p(\cdot)}(u,A) = \inf_{v \in W^{1,p(\cdot)}(\Omega;\R^m)} \  \lbrace \mathcal{F}(v,A): \ v = u \ \text{ in a neighborhood of } \partial A \rbrace\,.
\end{equation}
We consider the functionals ${F}_j: L^1(\Omega;\R^m)\times\mathcal{A}(\Omega)\to[0,+\infty]$ defined as
\begin{equation*}
{F}_j(u,A):=
\begin{cases}
\int_A f_j(x,\nabla u(x))\,\mathrm{d}x\,, & u\in W^{1,p(\cdot)}(\Omega;\R^m)\,, \\
+\infty & \mbox{ otherwise. }
\end{cases}
\end{equation*}
where $f_j$ satisfies {\rm\ref{ass-f1}},{\rm\ref{eq:bound1}} and {\rm\ref{eq:ass-f3}} for every $j\in\N$.
We then have the following $\Gamma$-convergence result.

\begin{proposition}
The functionals ${F}_j(\cdot,A)$ $\Gamma$-converge (up to a not relabeled subsequence) as $j\to+\infty$ in the strong topology of $L^1(\Omega;\R^m)$ to the functional ${F}(\cdot,A)$ for every $A\in\mathcal{A}(\Omega)$, where
\begin{equation}
{F}(u,A)=\int_A f_{\rm sob}(x,\nabla u(x))\,\mathrm{d}x
\label{eq:densitySob}
\end{equation}
and 
\begin{equation}
f_{\rm sob}(x,\xi):= \mathop{\lim\sup}_{\eps\to0^+} \frac{\mathbf{m}_{{F}}^{1,p(\cdot)}(\bar{\ell}_\xi, B_\eps(x))}{\gamma_d \eps^d}\,,\quad \mbox{ for all $x\in\Omega$ and $\xi\in\R^{m\times d}$\,.}
\label{eq:minprobSob}
\end{equation}
Moreover, $f_{\rm sob}$ is a Carath\'eodory function satisfying {\rm\ref{eq:bound1}} and it holds that
\begin{equation}
f_{\rm sob}(x,\xi) = \mathop{\lim\sup}_{\eps\to0^+} \mathop{\lim\inf}_{j\to+\infty} \frac{\mathbf{m}_{{F}_j}^{1,p(\cdot)}(\bar{\ell}_\xi, B_\eps(x))}{\gamma_d \eps^d} = \mathop{\lim\sup}_{\eps\to0^+} \mathop{\lim\sup}_{j\to+\infty} \frac{\mathbf{m}_{{F}_j}^{1,p(\cdot)}(\bar{\ell}_\xi, B_\eps(x))}{\gamma_d \eps^d}\,.
\label{eq:minprobSob2}
\end{equation}
\end{proposition}

\begin{proof}
The proof of the $\Gamma$-convergence result and the integral representation \eqref{eq:densitySob} can be obtained as in \cite[Theorem~4.1 and 4.2]{CM}. The characterization \eqref{eq:minprobSob} follows by adapting the global method of Section~\ref{sec: global method} to the variable exponent Sobolev setting, while \eqref{eq:minprobSob2} is a standard consequence of the $\Gamma$-convergence. We omit the details.
\end{proof}

We can now proceed with the announced identification of the bulk density.

\begin{theorem}\label{thm: fsob}
Under the assumptions of Theorem~\ref{th: gamma} and assumption {\rm\ref{eq:ass-f3}} on the sequence $(f_j)$, 
let $f_{\rm sob}$ and $f_\infty$ be defined as in \eqref{eq:minprobSob} and \eqref{eq:fdef-gsbvbis}, respectively. Then, for all $u\in GSBV^{p(\cdot)}(\Omega;\R^m)$ we have that 
\begin{align}\label{eq: f_infty=f}
 f_\infty(x,\nabla u(x))  = f_{\rm sob}(x,\nabla u(x))\quad \textit{\emph{for $\mathcal{L}^d$-a.e.\ $x\in\Omega$}}.
\end{align}
\end{theorem}

\begin{proof}
We show the two inequalities in \eqref{eq: f_infty=f}. We first prove
\begin{equation}
f_\infty(x,\nabla u(x)) \leq f_{\rm sob}(x,\nabla u(x)) \quad  \mbox{ for $\mathcal{L}^d$-a.e.\ $x\in\Omega$\,.}   
\label{eq:inequalitybulk1}
\end{equation}
First,   in view of  \eqref{eq: general minimizationsgbv}  and \eqref{eq: general minimizationsSobolev}, we get $\mathbf{m}_{\mathcal{F}}(\bar{\ell}_{ \xi },B_{\eps}(x)) \le \mathbf{m}^{1,p(\cdot)}_{\mathcal{F}}(\bar{\ell}_{ \xi },B_{\eps}(x))$  for all $\xi \in \R^{m\times d}$,  where we recall the notation $\bar{\ell}_{\xi}$ introduced in \eqref{eq:affinecomp}.  Then \eqref{eq:fdef-gsbv} implies 
\begin{align}\label{es: step1-1}
f_\infty(x,\xi)&=\limsup_{\eps\to 0^+}\frac{\mathbf{m}_{\mathcal{F}}(\bar{\ell}_{\xi}, B_{\eps}(x))}{\gamma_d\eps^d}  \le \limsup_{\eps\to 0^+}\frac{\mathbf{m}^{1,p(\cdot)}_{\mathcal{F}}(\bar{\ell}_{\xi},B_{\eps}(x))}{\gamma_d\eps^d}\,,
\end{align}
while by \eqref{eq:minprobSob}  and \eqref{eq:restricSob}  we find 
\begin{align}\label{es: step1-2}
f_{\rm sob}(x,\xi) =   \limsup_{\eps \to 0^+}   \frac{\mathbf{m}^{1,p(\cdot)}_{\mathcal{F}}(\bar{\ell}_{\xi},B_{\eps}(x))}{\gamma_d\eps^d}. 
\end{align} 
Thus, since both $f_\infty$ and $f_{\rm sob}$ are continuous with respect to $\xi$ by \ref{eq:ass-f3}, combining \eqref{es: step1-1}--\eqref{es: step1-2} we obtain \eqref{eq:inequalitybulk1}.

We now prove the reverse inequality
\begin{equation}
f_{\rm sob}(x,\nabla u(x)) \leq f_\infty(x,\nabla u(x)) \quad  \mbox{ for $\mathcal{L}^d$-a.e.\ $x\in\Omega$\,.}   
\label{eq:inequalitybulk2}
\end{equation}

First, from the Radon-Nikod\'ym Theorem we have that
\begin{equation}
f_\infty(x,\nabla u(x))= \lim_{\eps\to 0^+} \frac{\mathcal{F}(u,B_\eps(x))}{\gamma_d\eps^d}<+\infty
\label{eq:6.3fps}
\end{equation}
holds for $\mathcal{L}^d$-a.e.\ $x\in\Omega$\,. Let $(u_j)$ be a sequence of measurable functions such that $u_j\in GSBV^{p(\cdot)}(\Omega;\R^m)$ 
\begin{equation*}
u_j\to u \mbox{ in measure on $\Omega$ } \quad \mbox{ and } \quad  \lim_{j\to+\infty} \mathcal{F}_j(u_j,\Omega)=\mathcal{F}(u,\Omega)\,. 
\end{equation*}
Since $u\in GSBV^{p(\cdot)}(\Omega;\R^d)$, by virtue of Lemma~\ref{lemma: approx-grad} the approximate gradient $\nabla u(x)$ exists for $\mathcal{L}^d$-a.e. $x\in\Omega$. Then, since \eqref{eq:inequalitybulk2} needs to hold for $\mathcal{L}^d$-a.e. $x\in\Omega$, we may assume that \eqref{eq:6.3fps} holds at $x$ and that $\nabla u(x)$ exists. Since $\mathcal{F}(u,\cdot)$ is a Radon measure, there exists a subsequence $(\eps_k)\subset(0,+\infty)$ with $\eps_k\searrow 0$ as $k\to+\infty$ such that $\mathcal{F}(u,\partial B_{\eps_k}(x))=0$ for every $k\in\N$ and such that \eqref{eq:minprobSob2} holds along $(\eps_k)$, namely
\begin{equation}
f_{\rm sob}(x,\nabla u(x)) = \mathop{\lim}_{k\to+\infty} \mathop{\lim\sup}_{j\to+\infty} \frac{\mathbf{m}_{\mathcal{F}_j}^{1,p(\cdot)}(\bar{\ell}_{\nabla u(x)}, B_{\eps_k}(x))}{\gamma_d \eps_k^d}\,.
\label{eq:6.6fps}
\end{equation}
Moreover, with fixed $\eta\in(0,1)$, since $(u_j)$ is a recovery sequence and $\mathcal{F}(u,\cdot)$ is a Radon measure, for every $k\in\N$ we can find $j_k\in\N$ (depending also on $\eta$) such that, for every $j\geq j_k$ it holds that
\begin{equation}
\frac{\mathcal{F}_j(u_j, B_{\eps_k}(x))}{\gamma_d \eps_k^d} \leq \frac{\mathcal{F}(u, B_{\eps_k}(x))}{\gamma_d \eps_k^d} + \eta\,. 
\label{eq:6.7fps}
\end{equation}

Now, we have to modify the sequence $(u_j)$ to construct a competitor for the minimization problem $\mathbf{m}_{\mathcal{F}_j}^{1,p(\cdot)}(\bar{\ell}_{\nabla u(x)}, B_{\eps}(x))$ which defines $f_{\rm sob}$.

We introduce the functions
\begin{equation*}
u_j^{\eps_k}(y):= \frac{u_j(x+\eps_ky)-u_j(x)}{\eps_k} \quad \mbox{ and } \quad u^{\eps_k}(y):= \frac{u(x+\eps_ky)-u(x)}{\eps_k} \quad \mbox{ for $y\in B_1$.}
\end{equation*}
Then, since $u_j\to u$ in measure on $\Omega$, we have that $u_j^{\eps_k}\to u^{\eps_k}$ in measure on $B_1$ as $j\to+\infty$. In addition, by a diagonal argument and up to passing to a larger $j_k\in\N$, we also have
\begin{equation}
\hat{u}_k:=u_{j_k}^{\eps_k} \to \bar{\ell}_{\nabla u(x)} \quad \mbox{ in measure on $B_1$ as $k\to+\infty$.}
\label{eq:6.8fps}
\end{equation}
By virtue of \eqref{eq:6.6fps}, we may choose $(j_k)_k$ such that also
\begin{equation}
f_{\rm sob}(x,\nabla u(x)) = \mathop{\lim}_{k\to+\infty} \frac{\mathbf{m}_{\mathcal{F}_{j_k}}^{1,p(\cdot)}(\bar{\ell}_{\nabla u(x)}, B_{\eps_k}(x))}{\gamma_d \eps_k^d}
\label{eq:6.11fps}
\end{equation}
holds. Finally, taking into account \eqref{eq:energy}, \eqref{eq:6.3fps}, \eqref{eq:6.7fps} and with a change of variables we find
\begin{equation}
\mathop{\lim\sup}_{k\to+\infty} \int_{B_1} f_{j_k}(x+\eps_ky,\nabla \hat{u}_k(y))\,\mathrm{d}y \leq \lim_{k\to+\infty} \frac{\mathcal{F}(u,B_{\eps_k}(x))}{\gamma_d\eps_k^d}+\eta = f_\infty(x,\nabla u(x))+\eta \,.
\label{eq:6.13fps}
\end{equation}

Let $\mathcal{I}_k$ be defined as in \eqref{eq:partialenergy}, with $f_{j_k}(x+\eps_ky,\cdot)$ in place of $f(x,\cdot)$, and set
\begin{equation*}
p_k(y):=p(x+\eps_k y)\,,\quad y\in B_1\,.
\end{equation*}
We define, accordingly, 
\begin{equation*}
p_k^+:=\sup_{y\in B_1} p_k(y)\,,\quad p_k^-:=\inf_{y\in B_1} p_k(y)\,.
\end{equation*}
Let $\lambda>|\nabla u(x)|$. Then, by virtue of Lemma~\ref{lem:cdmsz} there exists $\mu>\lambda$ such that, for every $k$, we can find a function $\hat{v}_k\in SBV^{p_k(\cdot)}(B_1;\R^d)\cap L^\infty(B_1;\R^d)$ such that $\hat{v}_k= \hat{u}_k$ $\mathcal{L}^d$-a.e. in $B_1\cap\{|\hat{u}_k|\leq \lambda\}$, $|\hat{v}_k|\leq\mu$ and
\begin{equation}
\mathcal{I}_k(\hat{v}_k, B_1) \leq \left(1+\eta\right) \mathcal{I}_k(\hat{u}_k, B_1) + \beta\mathcal{L}^d (B_1\cap\{|\hat{u}_k|\geq \lambda\})\,.
\label{eq:(7.8cdsz)}
\end{equation}
Moreover, with \eqref{eq:6.8fps} and the fact that $|\bar{\ell}_{\nabla u(x)}|\leq|\nabla u(x)|<\lambda$ in $B_1$, we get 
\begin{equation}
\hat{v}_k\to \bar{\ell}_{\nabla u(x)} \quad \mbox{ in measure on $B_1$ as $k\to+\infty$} 
\label{eq:6.15fps0}
\end{equation}
and $\mathcal{L}^d (B_1\cap\{|\hat{u}_k|\geq \lambda\})\leq \eps_k$ for $k$ large enough. \EEE Taking into account {\rm\ref{eq:bound1}}, {\rm\ref{eq:bound2}}, \eqref{eq:(7.8cdsz)} and a change of variables we get
\begin{equation*}
\begin{split}
&\alpha\int_{B_1} |\nabla \hat{v}_k|^{p_k(y)}\,\mathrm{d}y \leq \frac{1+\eta}{\eps_k^d}\int_{B_{\eps_k}(x)}f_{j_k}(y,\nabla u_{j_k}(y))\,\mathrm{d}y + \beta \eps_k\,, \\
& \frac{\alpha}{\eps_k}\mathcal{H}^{d-1}(J_{\hat{v}_k}\cap B_1) \leq  \frac{\alpha}{\eps_k^d}\mathcal{H}^{d-1}(J_{{u}_{j_k}}\cap B_{\eps_k}(x)) \leq \frac{1}{\eps_k^d}\int_{J_{{u}_{j_k}}\cap B_{\eps_k}(x)} g_{j_k}(y, [u_{j_k}], \nu_{u_{j_k}})\,\mathrm{d}\mathcal{H}^{d-1}\,,
\end{split}
\end{equation*}
for $k$ large enough. Then, with \eqref{eq:6.3fps} and \eqref{eq:6.7fps}, 
we can find a constant $M>0$ independent of $k$ and $\eta$ such that
\begin{equation}
\int_{B_1} |\nabla \hat{v}_k|^{p_k(\cdot)}\,\mathrm{d}y \leq M \quad \mbox{ and } \quad \mathcal{H}^{d-1}(J_{\hat{v}_k}\cap B_1)\leq M \eps_k\,, 
\label{eq:7.13}
\end{equation}
for $k$ large enough, and
\begin{equation}
|D^s \hat{v}_k|(B_1)\leq 2\mu M\eps_k\,.
\label{eq:7.14}
\end{equation}

Now, we regularize the sequence $(\hat{v}_k)$ in order to obtain a sequence $\hat{w}_k\in W^{1,p_k(\cdot)}(B_1;\R^m)$ such that
\begin{equation}
\int_{B_1} f_{j_k}(x+\eps_ky,\nabla \hat{w}_k(y))\,\mathrm{d}y \leq \int_{B_1} f_{j_k}(x+\eps_ky,\nabla \hat{v}_k(y))\,\mathrm{d}y + \eta\,. 
\label{eq:7.25cdmsz}
\end{equation}

For this, we may adapt to the variable exponent setting the argument for the proof of \cite[Theorem~5.2(b), Step~1]{CDMSZ}, devised for a constant exponent $q$. We just provide the main steps of this adaptation. 

For fixed $t>0$, we first define the sets
\begin{equation*}
\begin{split}
R^t_k & :=\left\{y\in B_1:\,\, \frac{|D^s \hat{v}_k|(\overline{B_r(y)})}{\mathcal{L}^d(B_r(y))}\leq t \quad \mbox{ for every $r>0$ with }\overline{B_r(y)}\subset B_1 \right \}\,, \\
S^t_k & := J_{\hat{v}_k} \cup \left\{ y\in B_1:\,\, |\nabla \hat{v}_k(y)|\geq \frac{t}{2} \right\}\,.
\end{split}
\end{equation*}
We claim that
\begin{equation*}
\begin{split}
\mathcal{L}^d(B_1\backslash R^t_k) & \leq \frac{2\cdot 5^d}{t} \left( |D^s \hat{v}_k|(B_1) + \int_{S_k^t} |\nabla \hat{v}_k(y)|\,\mathrm{d}y\right) \\
& \leq \frac{2\cdot 5^d}{t} |D^s \hat{v}_k|(B_1) + 2^{p^+_k+1}\cdot 5^d \int_{S_k^t} \left(\frac{|\nabla \hat{v}_k(y)|}{t} \right)^{p_k(y)}\,\mathrm{d}y\,,
\end{split}
\end{equation*}
Indeed, the first inequality follows from the Vitali Covering Lemma, arguing exactly as in \cite[Theorem~5.2(b), Step~1]{CDMSZ}. The second inequality follows from the first one, using that $\frac{2|\nabla \hat{v}_k(y)|}{t}\geq 1$ on $S_k^t$. Now, taking into account \eqref{eq:7.13}, we get
\begin{equation}
\begin{split}
\mathcal{L}^d(B_1\backslash R^t_k) &  \leq \frac{2\cdot 5^d}{t} |D^s \hat{v}_k|(B_1) + \frac{2^{p^+_k+1}\cdot 5^d M}{\min\{t^{p^+_k}, t^{p^-_k}\}}\,.
\end{split}
\label{eq:7.21}
\end{equation}

Choosing
\begin{equation*}
t_k:= (2\mu M \eps_k)^{-\frac{1}{p^-_k-1}}
\end{equation*}
we have $t_k\geq1$ for $k$ large enough and, taking into account \eqref{eq:7.14}, from \eqref{eq:7.21} we obtain
\begin{equation*}
\begin{split}
t_k^{p^-_k}\mathcal{L}^d(B_1\backslash R^{t_k}_k) &  \leq {2\cdot 5^d}+ {2^{p^+_k+1}\cdot 5^d M} \,.
\end{split}
\end{equation*}
By virtue of Lemma~\ref{lem:dieni3.2}, and since $p^-_k-1\ge p^--1>0$, it holds now
\begin{equation*}
t_k^{p^+_k-p^-_k} \leq {\gamma_d^{-(p^-_k-p^+_k)}}[\mathcal{L}^d(B_{\eps_k})]^{\frac{p^-_k-p^+_k}{d(p^-_k-1)}} \leq C
\end{equation*}
for $k$ large. We then conclude that
\begin{equation}
\begin{split}
t_k^{p^+_k}\mathcal{L}^d(B_1\backslash R^{t_k}_k) \leq C({2\cdot 5^d}+ {2^{p^++1}\cdot 5^d M})=:\widetilde{M} \,,
\end{split}
\label{eq:7.22}
\end{equation}
whence, in particular, since $\eps_k<1$ for $k$ large enough, we get
\begin{equation}
\begin{split}
\mathcal{L}^d(B_1\backslash R^{t_k}_k) \leq \frac{\widetilde{M}}{t_k^{p^+_k}}= \widetilde{M}(2\mu M \eps_k)^{\frac{{p^+_k}}{p^-_k-1}}\leq \overline{M}_{p}\eps_k^{\frac{{p^-}}{p^+-1}} \,.
\end{split}
\label{eq:7.22bis}
\end{equation}

Now, by a Lusin's type approximation argument (see, e.g., \cite{EG}), one can construct a sequence of Lipschitz functions $(\hat{z}_k)$ on $B_1$, with Lip$(\hat{z}_k) \le c_d t_k$ for some constant $c_d$ depending only on the dimension, such that $\hat{z}_k=\hat{v}_k$ $\mathcal{L}^d$-a.e. in $R_k^{t_k}$. Setting $\bar{p}:=p(x)$, we claim that $(\hat{z}_k)$ are bounded in $W^{1,\bar{p}}(B_1;\R^d)$. Indeed, with \eqref{eq:7.22} and \eqref{eq:7.13} we first have
\begin{equation}
\begin{split}
\int_{B_1} |\nabla \hat{z}_k|^{p_k(y)}\,\mathrm{d}y & \leq 2^{p^+_k-1}\left( \int_{R_k^{t_k}} |\nabla \hat{v}_k|^{p_k(y)}\,\mathrm{d}y + \max\{c_d^{p^+_k},c_d^{p^-_1}\} t_k^{p^+_k}\mathcal{L}^d(B_1\backslash R^{t_k}_k)\right) \\
 & \leq 2^{p^+-1} (M+ \max\{c_d^{p^+},c_d^{p^-}\}\widetilde{M})=:M_{d,p}\,.
\end{split}
\label{eq:stimapk}
\end{equation}
Note that $\bar{p}=p_k(0)$ for every $k\in\N$. Then, since $(\bar{p}-p_k(y))^+\leq p_k^+-p_k^-$ and $t_k\geq1$ for $k$ large enough, with \eqref{eq:7.22bis} for every $y\in B_1$ we get
\begin{equation}
|\nabla \hat{z}_k|^{(\bar{p}-p_k(y))^+} \leq c_d^{p_k^+-p_k^-}t_k^{p^+_k-p^-_k}\leq C\,.
\label{eq:7.22tris}
\end{equation}
Finally, with \eqref{eq:stimapk} and \eqref{eq:7.22tris}, by a simple inequality we obtain
\begin{equation}
\begin{split}
\int_{B_1} |\nabla \hat{z}_k|^{\bar{p}}\,\mathrm{d}y &\leq \mathcal{L}^d(B_1) + \int_{B_1} |\nabla \hat{z}_k|^{(\bar{p}-p_k(y))^+}|\nabla \hat{z}_k|^{p_k(y)}\,\mathrm{d}y \\
 & \leq \mathcal{L}^d(B_1) + CM_{d,p}\,. 
\end{split}
\label{eq:7.22quater}
\end{equation}
Then, by applying \cite[Lemma~1.2]{FMP} to $(\hat{z}_k)$, we find a sequence of Lipschitz functions $(\hat{w}_k)$ which satisfy $\hat{w}_k\in W^{1,\bar{p}}(B_1;\R^m)$, $|\nabla \hat{w}_k|^{\bar{p}}$ equi-integrable uniformly with respect to $k$, and $\mathcal{L}^d(\{\hat{z}_k\neq \hat{w}_k\})\to0$ as $k\to+\infty$. Since $|\hat{z}_k|\leq \mu$ in $B_1$, we may assume also that $|\hat{w}_k|\leq \mu$ $\mathcal{L}^d$-a.e. in $B_1$. An inspection to the proof of \cite[Lemma~1.2]{FMP} shows that  $(\hat{w}_k)$ can be chosen in such a way that
\begin{equation}
\mathrm{Lip}(\hat{w}_k)\leq c_d\,\mathrm{Lip}(\hat{z}_k)
\label{eq:lipconst}
\end{equation}
holds.

 We claim that $(|\nabla \hat{w}_k|^{p_k(\cdot)})$ is equi-integrable on $B_1$ uniformly with respect to $k$. Indeed, arguing as for \eqref{eq:7.22tris} we first get, for every $y\in B_1$,
\begin{equation}
|\nabla \hat{w}_k|^{(p_k(y)-\bar{p})^+} \leq \widetilde{c_d}^{p_k^+-p_k^-}t_k^{p^+_k-p^-_k}\leq C\,.
\label{eq:7.23}
\end{equation}
Then, for every fixed $E\subseteq B_1$, arguing as for \eqref{eq:7.22quater} and taking into account \eqref{eq:7.23} we obtain
\begin{equation}
\begin{split}
\int_{E} |\nabla \hat{w}_k|^{p_k(y)}\,\mathrm{d}y &\leq \mathcal{L}^d(E) + \int_{E} |\nabla \hat{w}_k|^{(p_k(y)-\bar{p})^+}|\nabla \hat{w}_k|^{\bar{p}}\,\mathrm{d}y \\
 & \leq \mathcal{L}^d(E) + C\int_{E} |\nabla \hat{w}_k|^{\bar{p}}\,\mathrm{d}y\,. 
\end{split}
\label{eq:7.23bis}
\end{equation}
This and the equi-integrability of $|\nabla \hat{w}_k|^{\bar{p}}$ imply the claim.

Moreover, from \eqref{eq:7.22bis}, and since by \eqref{eq:6.15fps0} the equibounded sequence $(\hat{w}_k-\bar{\ell}_{\nabla u(x)})$ tends to 0 in measure on $B_1$, we have
\begin{equation}
\mathcal{L}^d(\{\hat{w}_k\neq \hat{v}_k\})\leq \overline{M}_{p}\eps_k^{\frac{{p^-}}{p^+-1}}\,, \quad \mbox{ and } \quad \int_{B_1} |\hat{w}_k-\bar{\ell}_{\nabla u(x)} |^{p_k(y)}\,\mathrm{d}y\to0
\label{eq:6.15fps}
\end{equation}
as $k\to+\infty$. 

In order to prove \eqref{eq:7.25cdmsz}, we notice that
\begin{equation*}
\int_{B_1} f_{j_k}(x+\eps_ky,\nabla \hat{w}_k(y))\,\mathrm{d}y \leq \int_{B_1} f_{j_k}(x+\eps_ky,\nabla \hat{v}_k(y))\,\mathrm{d}y + \int_{\{\hat{w}_k\neq \hat{v}_k\}}f_{j_k}(x+\eps_ky,\nabla \hat{w}_k(y))\,\mathrm{d}y\,.
\end{equation*}
Now, taking into account the equi-integrability of $(|\nabla \hat{w}_k|^{p_k(\cdot)})$, the upper bound {\rm\ref{eq:bound1}} and \eqref{eq:6.15fps}, for $\eps_k$ small enough we get
\begin{equation*}
\int_{\{\hat{w}_k\neq \hat{v}_k\}}f_{j_k}(x+\eps_ky,\nabla \hat{w}_k(y))\,\mathrm{d}y<\eta\,, 
\end{equation*}
whence \eqref{eq:7.25cdmsz} follows.

Finally, we have to modify the sequence $(\hat{w}_k)$ in such a way that it attains the boundary datum $\bar{\ell}_{\nabla u(x)}$ in a neighborhood of $\partial B_1$. 
We know that the functionals $\mathcal{I}_k(u,A)$ above for $u\in W^{1,p_k(\cdot)}(A;\R^m)$ and $A\in\mathcal{A}(\Omega)$ satisfy uniformly the Fundamental Estimate proved in Lemma~\ref{lemma: fundamental estimate}. Namely, corresponding to the fixed $\eta$ above, 
there exist a constant $C_\eta>0$ and a sequence $(\hat{\rm w}_k)$ in $W^{1,p_k(\cdot)}(B_1;\R^m)$ with $\hat{\rm w}_k=\bar{\ell}_{\nabla u(x)}$ in a neighborhood of $\partial B_1$ for all $k\in\N$ such that
\begin{equation}
\mathcal{I}_k(\hat{\rm w}_k,B_1) \leq (1+\eta)\left(\mathcal{I}_k(\hat{w}_k,B_1) + \mathcal{I}_k(\bar{\ell}_{\nabla u(x)},B_1\backslash \overline{B_{1-\eta}})\right) + C_\eta \int_{B_1} |\hat{w}_k - \bar{\ell}_{\nabla u(x)}|^{p_k(y)}\,\mathrm{d}y + \gamma_d\eta\,.
\label{eq:fundestbulk}
\end{equation}
\EEE
Now, taking into account {\rm\ref{eq:bound1}}, \eqref{eq:6.15fps} and the fact that $\mathcal{L}^d(B_1\backslash \overline{B_{1-\eta}})\leq d\eta$, we get
\begin{equation}
\mathop{\lim\sup}_{k\to+\infty}\mathcal{I}_k(\hat{\rm w}_k,B_1) \leq (1+\eta) \mathop{\lim\sup}_{k\to+\infty}\mathcal{I}_k(\hat{w}_k,B_1) + d\eta(1+\eta)\beta (1+|\nabla u(x)|^{\bar{p}})+\gamma_d\eta\,. 
\label{eq:6.17fps}
\end{equation}
Then, with \eqref{eq:6.13fps}, \eqref{eq:7.25cdmsz} and recalling the definition of $\mathcal{I}_k$, we obtain
\begin{equation}
\mathop{\lim\sup}_{k\to+\infty}\mathcal{I}_k(\hat{\rm w}_k,B_1) \leq (1+\eta)( f_\infty(x,\nabla u(x)) +\eta) + d\eta(1+\eta)\beta (1+|\nabla u(x)|^{\bar{p}})+\gamma_d\eta\,. 
\label{eq:6.18fps}
\end{equation}
Setting
\begin{equation*}
\widetilde{\rm w}_k(y):= \eps_k \hat{\rm w}_k((y-x)/\eps_k) + \bar{\ell}_{\nabla u(x)}x \quad \mbox{ for } y\in B_{\eps_k}(x)\,,  
\end{equation*}
we have $\widetilde{\rm w}_k\in W^{1,p(\cdot)}(B_{\eps_k}(x);\R^m)$ and
\begin{equation}
\mathcal{I}_k(\hat{\rm w}_k,B_1) = \frac{1}{\eps_k^d} \int_{B_{\eps_k}(x)} f_{j_k}(y, \nabla \widetilde{\rm w}_k(y))\,\mathrm{d}y\,.
\label{eq:6.19fps}
\end{equation}
Moreover, since $\hat{\rm w}_k=\bar{\ell}_{\nabla u(x)}$ in a neighborhood of $\partial B_1$, it follows that $\widetilde{\rm w}_k=\bar{\ell}_{\nabla u(x)}$ in a neighborhood of $\partial B_{\eps_k}(x)$. Then, with \eqref{eq: general minimizationsSobolev}, \eqref{eq:restricSob} and \eqref{eq:6.19fps} we obtain
\begin{equation*}
 \frac{\mathbf{m}_{\mathcal{F}_{j_k}}^{1,p(\cdot)}(\bar{\ell}_{\nabla u(x)}, B_{\eps_k}(x))}{\gamma_d \eps_k^d} \leq \mathcal{I}_k(\hat{\rm w}_k,B_1)\,,
\end{equation*}
whence passing to the limsup as $k\to+\infty$, recalling \eqref{eq:6.18fps}, and then letting $\eta\to0^+$, 
we get
\begin{equation*}
\mathop{\lim\sup}_{k\to+\infty}  \frac{\mathbf{m}_{\mathcal{F}_{j_k}}^{1,p(\cdot)}(\bar{\ell}_{\nabla u(x)}, B_{\eps_k}(x))}{\gamma_d \eps_k^d} \leq f_\infty(x,\nabla u(x))\,.
\end{equation*}
The assertion \eqref{eq:inequalitybulk2} then follows from \eqref{eq:6.11fps}.
\end{proof}

\subsection{Identification of the surface density}\label{sec:identsurfdens}

We conclude our analysis with the identification of the surface density. We will prove that it coincides with the asymptotic surface density of functionals $\mathcal{F}_j$ when restricted to the space $SBV_{\mathrm{pc}}(A,\R^m)$ of those functions $u\in SBV(A,\R^m)$ such that $\nabla u=0$ $\mathcal{L}^d$-a.e. in $A$ and $\mathcal{H}^{d-1}(J_u)<+\infty$.

In order to do that, we consider the sequence of surface energies
\begin{equation}
G_j(u,A) := 
\begin{cases}
\ds \int_{J_u\cap A}g_j(x,[u],\nu_u)d \mathcal{H}^{d-1} &\text{if} \; u|_A\in GSBV^{p(\cdot)}(A,\R^m),\\
+\infty \quad & \mbox{otherwise in}\; L^0(\R^d,\R^m), 
\end{cases}
\label{eq:surfenergies}
\end{equation}
and, correspondingly, we define the sequence of minimum problems
\begin{equation}
{\bf m}^{{PC}}_{G_j} (u_{x,\zeta,\nu},A) := \inf\left\{G_j(u,A):  u\in L^0(\R^d,\R^m),\ u|_A\in SBV_{\mathrm{pc}}(A,\R^m),\ u=u_{x,\zeta,\nu} \textrm{ near }\partial A \right\}\,,
\label{eq:surfaceminprobl}
\end{equation}
where $u_{x,\zeta,\nu}$ coincides with $u_{x,\zeta,0,\nu}$ defined in \eqref{eq: jump competitor}.

Since, to the best of our knowledge, a $\Gamma$-convergence result for functionals $G_j$ whose densities $g_j$ explicitly depend on the jump $[u]$ is still missing in literature, with Theorem~\ref{thm:identsurf} below we will show directly that
\begin{equation*}
 g_{\infty}(x,[u](x),\nu_u) = \limsup_{\e \to 0^+} \lim_{j \to + \infty} \frac{
{\bf m}_{G_j}^{PC}(u_{x,[u](x),\nu_u}, B_\e(x))}{\gamma_{d-1}\e^{d-1}}\,,\quad \mbox{ for $\mathcal{H}^{d-1}$-a.e. $x\in J_u$.}
\end{equation*}
We also remark that, in the proof below, Theorem \ref{thm:fspoincare} allows for a quick construction in Step 2.3 of an optimal sequence of piecewise constant functions (cfr. the more involved arguments in \cite[Theorem 5.2, (c)-(d)]{CDMSZ}, whose compliance with the present setting was not investigated).

\begin{theorem}\label{thm:identsurf}
Let $\Omega \subset \R^d$ be open and $p:\Omega\to(1,+\infty)$ be a continuous variable exponent. Let $(f_j)_j$ and $(g_j)_j$ be sequences functions satisfying {\rm\ref{ass-f1}}-{\rm\ref{eq:ass-f3}} and {\rm\ref{ass-g1}}, {\rm\ref{ass-g4}},  {\rm\ref{eq:bound2}}, {\rm\ref{ass-g7}} and {\rm\ref{ass-g2}}, respectively.  Let $g_\infty$ be defined by  \eqref{eq:gdef-gsbvbis}.  Then, for all $u\in GSBV^{p(\cdot)}(\Omega,\R^m)$ we have that 
\begin{equation}\label{eq: g_infty=g}
g_\infty(x,[u](x),\nu_u(x))
=  g_{\rm pc}(x,[u](x),\nu_u(x))\,,\quad \mbox{for $\mathcal{H}^{d-1}$-a.e. $x\in J_u$,}
\end{equation}
where
\begin{equation}
 g_{\rm pc}(x,\zeta,\nu):= \limsup_{\e \to 0^+} \lim_{j \to + \infty} \frac{
{\bf m}_{G_j}^{PC}(u_{x,\zeta,\nu}, B_\e(x))}{\gamma_{d-1}\e^{d-1}}\,.
\label{eq:gpc}
\end{equation}
\end{theorem}

\begin{proof}
For every $x \in \R^d$, $\zeta \in \R_0^m$, and $\nu \in \mathbb{S}^{d-1}$
we define 
\begin{align}
g' (x, \zeta ,  \nu) &: = \limsup_{\e \to 0^+} \liminf_{j \to + \infty} \frac{
{\bf m}_{G_j}^{PC}(u_{x,\zeta,\nu}, B_\e(x))}{\gamma_{d-1}\e^{d-1}} ,  \label{eq:g'}\\
g'' (x, \zeta ,  \nu) &: = \limsup_{\e \to 0^+} \limsup_{j \to + \infty} \frac{ 
{\bf m}_{G_j}^{PC}(u_{x,\zeta,\nu}, B_\e(x))}{\gamma_{d-1}\e^{d-1}}. \label{eq:g''}
\end{align}
\emph{Step 1.} We start with the proof of the inequality 
\begin{equation}
g_\infty(x,\zeta,\nu) \le g'(x,\zeta ,\nu)\,.
\label{eq:gineq1}
\end{equation}

For this, we fix a triple $(x,\zeta,\nu)\in \R^d\times \R^m_0 \times \Sph^{d-1}$ and $0<\eta<1$. By the definition of ${\bf m}^{PC}_{G_j}$ (see \eqref{eq:surfaceminprobl}), for every $j$ there exists $u_j \in L^0(\R^d,\R^m)$, with $u_j|_{B_\e( x)}\in SBV_{\mathrm{pc}}(B_\e( x), \R^m)$, such that $u_j=u_{x,\zeta ,\nu }$ in a neighborhood of $\partial B_\e( x)$ and
\begin{equation}\label{eq:8.1}
G_j(u_j,B_\e( x)) \leq {\bf m}^{PC}_{G_j}(u_{x,\zeta ,\nu },B_\e( x))+ \eta\, \e^{d-1}.
\end{equation}
Now, given $\lambda> |\zeta|$, 
by virtue of Lemma~\ref{lem:cdmsz}, 
 for every $j$ there exists $\hat{u}_j$ such that
\begin{equation*}
\mathcal{F}_j(\hat{u}_j,B_\e( x))\le (1+\eta) \mathcal{F}_j(u_j,B_\e( x))+\beta\mathcal{L}^d(B_\e( x)\cap\{|u_j|\ge \lambda\})\,.
\end{equation*}
Moreover, $\hat{u}_j=u_{x,\zeta ,\nu }$ in a neighborhood of $\partial  B_\e( x)$, $|\hat{u}_j|\le \mu$ in $\R^d$ and, from the chain rule, $\nabla \hat{u}_j=0$ $\mathcal{L}^d$-a.e.\ in $B_\e( x)$. Consequently, the functions $v_j$ defined for every $j\in\N$ as
\begin{equation} \label{eq:8.2}
v_j:= \begin{cases}
\hat{u}_j & \text{in }\; B_\e( x)
\cr
u_{x,\zeta ,\nu } & \text{in }\; \R^d\setminus B_\e( x)
\end{cases}
\end{equation}
satisfy $v_j|_A\in SBV_{\mathrm{pc}}(A, \R^m)$ for every $A\in\mathcal{A}(\Omega)$ and, from the definition, also the uniform bound
\begin{equation}\label{eq:8.3}
|v_j|\le \mu\quad\hbox{in }\R^d\,.
\end{equation}
Now, arguing as for the proof of \cite[eq.~(8.4)]{CDMSZ}, with \ref{eq:bound2}, \ref{ass-g7} and \ref{ass-g3} (which holds with $c=\frac{\beta}{\alpha}$) we find that for every $j$
\begin{equation}\label{eq:8.4}
\mathcal{H}^{d-1}(J_{v_j}\cap B_\e(x))  \leq M_d \e^{d-1},
\end{equation}
where $M_d:= \frac{\beta}{\alpha^2}(\beta \gamma_{d-1} + 1)$.

Since $v_j \in SBV_{\mathrm{pc}}(B_\e(x), \R^m)$ 
and \eqref{eq:8.3}-\eqref{eq:8.4} hold, we can apply the compactness result \cite[Theorem 4.8]{Ambrosio-Fusco-Pallara:2000} to deduce the existence of a function $v\in SBV_{\mathrm{pc}}(B_\e(x),\R^m)\cap L^\infty(B_\e(x),\R^m)$ and a subsequence (not relabelled) converging in measure to $v$ on $B_\eps(x)$. 
We extend $v$ to $\R^d$ by setting $v=u_{x,\zeta ,\nu }$ in $\R^d\setminus B_\e(x)$ and observe that $v|_A\in SBV_{\mathrm{pc}}(A, \R^m)$ for every $A\in\mathcal{A}(\Omega)$. 
Moreover, by the definitions of $v_j$ and $v$ and by \eqref{eq:8.3}, the convergence in measure on $B_\e(x)$ 
implies that $|v|\le \mu\ \;\mathcal{L}^d\text{-a.e.\ in }\; \R^d$.

In particular, for $A=B_{(1+\eta)\e}(x)$ we have $v|_{B_{(1+\eta)\e}(x)}\in SBV_{\mathrm{pc}}(B_{(1+\eta)\e}(x),\R^m)$ and $v=u_{x,\zeta ,\nu }$ in $B_{(1+\eta)\e}(x)\setminus B_\e(x)$, which combined with the $\Gamma$-convergence of $\mathcal{F}_j(\cdot, B_{(1+\eta)\e}(x))$ to $\mathcal{F}(\cdot,B_{(1+\eta)\e}(x))$ with respect to the convergence in measure gives 
\begin{equation}\label{eq:8.37}
m_{\mathcal{F}}(u_{x,\zeta ,\nu },B_{(1+\eta)\e}(x ))\le \mathcal{F}(v,B_{(1+\eta)\e}(x)) \leq \liminf_{j\to+\infty} {} \mathcal{F}_j(v_j,B_{(1+\eta)\e}(x))\,.
\end{equation}

Taking into account the upper bounds in \ref{eq:bound1}, \ref{eq:bound2}, and \eqref{eq:8.1}-\eqref{eq:8.2} 
we obtain
\begin{align*}
\mathcal{F}_j(v_j&,B_{(1+\eta)\e}(x )) \le \mathcal{F}_j(v_j,B_\e(x )) +\mathcal{F}_j(u_{x,\zeta ,\nu },B_{(1+\eta)\e}(x )\setminus \overline B_\e(x )) 
\\
 &\le (1+\eta) \mathcal{F}_j(u_j,B_\e( x)) + \beta\gamma_d(1+2^d)\e^d + G_j(u_{x,\zeta ,\nu },B_{(1+\eta)\e}(x )\setminus \overline B_\e(x )) 
\\
&\le (1+\eta)G_j(u_j,B_\e(x ))  + \beta\gamma_d(3+2^d)\e^d + \beta\gamma_{d-1}((1+\eta)^{d-1}-1)\e^{d-1}
\\
&\le  (1+\eta)\, {\bf m}^{PC}_{G_j}(u_{x,\zeta ,\nu },B_\e(x))  + \beta\gamma_d(3+2^d)\e^d +C_d\eta \e^{d-1}
\end{align*}
where $C_d:=2+\beta\gamma_{d-1}(2^{d-1}-1)$. This inequality, together with \eqref{eq:8.37}, gives 
$$
{\bf m}_{\mathcal{F}}(u_{x,\zeta ,\nu },B_{(1+\eta)\e}(x))\le  (1+\eta)\liminf_{k\to+\infty} {\bf m}^{PC}_{G_j}(u_{x,\zeta ,\nu },B_\e(x)) + \beta\gamma_d(3+2^d)\e^d +C_d\eta
\e^{d-1}\,.
$$
Now, dividing both the sides by $\gamma_{d-1}\e^{d-1}$, taking the limsup as $\e\to 0^+$, and recalling \eqref{eq:g'} and \eqref{eq:gdef-gsbvbis}, we obtain 
$$
(1+\eta)^{d-1} g_\infty(x,\zeta,\nu)
\le   (1+\eta)g'(x,\zeta ,\nu)+\frac{\eta}{\gamma_{d-1}} C_d\,,
$$
whence by taking the limit as $\eta\to 0^+$ we get \eqref{eq:gineq1}. \\
\noindent
\emph{Step 2.} We now prove
\begin{equation}
g''(x,[u](x),\nu_u(x))\leq g_\infty(x,[u](x),\nu_u(x))
\label{eq:gineq2}
\end{equation}
for $\mathcal{H}^{d-1}$-a.e. $x\in J_u\cap A$.

We will prove \eqref{eq:gineq2} for functions $u$ which belong to $SBV^{p(\cdot)}(A,\R^m)\cap L^\infty(A,\R^m)$, while the general case of (unbounded) functions in $GSBV^{p(\cdot)}(A,\R^m)$ can be obtained from the previous case by constructing a sequence of truncations of function $u$ as in the Step~5 of \cite[Proof of Theorem~5.2(d)]{CDMSZ}.
\smallskip

Let $A\in\mathcal{A}(\Omega)$, $u\in SBV^{p(\cdot)}(A,\R^m)\cap L^\infty(A,\R^m)$. Let $\eta\in(0,1)$. We fix $x\in J_u$ such that, by setting $\zeta:=[u](x)$ and $\nu:=\nu_u(x)$, we have
\begin{eqnarray}
\label{eq:8.10}
&\ds \zeta\neq 0,
\\
\label{eq:8.11}
&\ds 
\lim_{\e \to 0^+} \frac{1}{(\eta\e)^d}\int_{B_{\eta\e}(x)}|u(y)-u_{x,\zeta,\nu}(y)|^{p(y)}\,\mathrm{d}y = 0 , 
\\
\label{eq:8.12}
&\ds g_\infty(x,\zeta,\nu) = \lim_{\e \to 0^+} \frac{\mathcal{F}(u, B_{\eta\e}(x))}{\gamma_{d-1}(\eta\e)^{d-1}}\,.
\end{eqnarray}
Note that \eqref{eq:8.10} and \eqref{eq:8.11} are satisfied for $\hs^{d-1}$-a.e.\ $x\in J_u$ for $p(\cdot)\equiv1$ (see, e.g., \cite[Definition 3.67 and Theorem~3.78]{Ambrosio-Fusco-Pallara:2000}). This, combined with the boundedness of both $u$ and $u_{x,\zeta,\nu}$, implies the \eqref{eq:8.11} for any variable exponent such that $p^-\geq1$ and $p^+<+\infty$. Also \eqref{eq:8.12} holds for $\hs^{d-1}$-a.e.\ $x\in J_u$, thanks to a generalized version of the Besicovitch Differentiation Theorem (see \cite{Mor} and \cite[Sections~1.2.1-1.2.2]{FonLeo}).

We extend $u$ to $\R^d$ by setting $u=0$ on $\R^d\setminus A$. By the $\Gamma$-convergence of $\mathcal{F}_j (\cdot,A)$ to $\mathcal{F} (\cdot,A)$ there exists a sequence $(u_j)$ converging to $u$ in $L^0(\R^d,\R^m)$ such that
$$
 \lim_{k\to +\infty}\mathcal{F}_j(u_j,A)=\mathcal{F}(u,A)\,.
$$

Since $\mathcal{F}(u,\cdot)$ is a finite Radon measure, we have that
$\mathcal{F}(u,\partial B_{\eta\e}(x))=0$
for all $\e>0$ such that $B_{\eta\e}(x)\subset A$, except for a countable set. As a consequence $(u_j)$ is a recovery sequence for $\mathcal{F}(u,\cdot)$ also in $B_{\eta\e}(x)$\ie
\begin{equation}\label{e:bordo}
\lim_{k\to +\infty}\mathcal{F}_j(u_j,B_{\eta\e}(x))=\mathcal{F}(u,B_{\eta\e}(x)),
\end{equation}
for all $\e>0$ except for a countable set. Let $\e$ be such that \eqref{e:bordo} holds.

We now fix 
$\lambda> \max\{ \|u\|_{L^\infty(\R^d,\R^m)}, |\zeta| \}$ and $\mu$ as in Lemma~\ref{lem:cdmsz}.
Then for every $j$ there exists $v_j$ such that
\begin{equation} \nonumber
\mathcal{F}_j(v_j,B_{\eta\e}(x))\le (1+\eta) \mathcal{F}_j(u_j,B_{\eta\e}(x))+\beta\mathcal{L}^d(B_{\eta\e}(x)\cap\{|u_j|\ge\lambda\})\,,
\end{equation}
and $ |v_j|\le \mu$ in $\R^d$. We deduce that $v_j\to u$ in $L^{p(\cdot)}_{\mathrm{loc}}(\R^d,\R^m)$ as well as
\begin{equation*} 
 \limsup_{j\to +\infty} \mathcal{F}_j(v_j,B_{\eta\e}(x))\le (1+\eta) \mathcal{F} (u,B_{\eta\e}(x)).
\end{equation*}
Hence there exists $j_0(\e)>0$ such that whenever $j \ge  j_0(\e)$
\begin{align}\label{estimate:1}
 \mathcal{F}_j(v_j, B_{\eta\e}(x))
\le (1+\eta) \mathcal{F} (u,B_{\eta\e}(x))+ (\eta\e)^d\,.
\end{align}
We now modify each $v_j$ in order to obtain a function $z_j$ which is an admissible competitor in the $j$-th minimization problem defining $g''(x,\zeta,\nu)$.   

\medskip

\noindent\textit{Step 2.1.} 
We first define the blow-up function $v_j^\e$ at $x$ as 
\begin{equation*}
v_j^\e(y):= v_j(x+\e y)\quad\hbox{for }y \in B_\eta\,,
\end{equation*}
and the blow-up variable exponent at $x$ as 
\begin{equation*}
p_\eps(y):= p(x+\e y)\quad\hbox{for }y \in B_\eta\,.
\end{equation*}
Now, we modify  $v_j^\e$ so that it agrees with the boundary datum
$u_{0,\zeta,\nu}$ in a neighbourhood of $\partial B_\eta$. To this end, we apply the Fundamental Estimate (Lemma~\ref{lemma: fundamental estimate}) to the functionals $\mathcal{F}_{j,\e}\colon \big(SBV^{p_\eps(\cdot)}(B_\eta, \R^m)\cap L^\infty(B_\eta, \R^m)\big){\times} \mathcal{A}(B_\eta)\to[0,+\infty)$ defined as \begin{equation}\label{eq:8.15}
\mathcal{F}_{j,\e}(v,A):= \int_A f_j (x+\e y,\nabla v(y) ) \mathrm{d}y 
+ \int_{J_v\cap A} g_j (x+\e y,[v](y),\nu_v(y) ) \mathrm{d}\mathcal{H}^{d-1}(y)\,,
\end{equation}
where $\mathcal{A}(B_\eta)$ denotes the class of open subsets in $B_\eta$.

Let $K_\eta \subset B_\eta$ be a compact set such that
\begin{equation}\label{eq:8.16}
 \beta\left(\mathcal{L}^{d}(B_\eta\setminus K_\eta) + \mathcal{H}^{d-1}(\Pi^\nu_0\cap(B_\eta\setminus K_\eta))\right) < \eta^d\,.
\end{equation}
Then, the argument of the proof of Lemma~\ref{lemma: fundamental estimate} allows us to deduce the existence of a constant $M_\eta>0$ and a finite family of cut-off functions 
$\varphi_1,\dots, \varphi_N\in C_c^\infty(B_\eta)$ such that $0\leq \varphi_i\leq 1$ in $B_\eta$, $\varphi_i=1$ in a neighbourhood of $K_\eta$, and 
\begin{align}\nonumber
\mathcal{F}_{j,\e}(\hat{v}^\e_j,B_\eta) \leq {}&(1+\eta)\big(\mathcal{F}_{j,\e}(v^\e_j, B_\eta) + \mathcal{F}_{j,\e}(u_{0,\zeta,\nu},B_\eta\setminus K_\eta)\big)
\\
&+ M_\eta \int_{B_\eta}|v^\e_j(y) - u_{0,\zeta,\nu}(y)|^{p_\eps(y)}\,\mathrm{d}y + \gamma_d\eta^{d+1},\label{eq:8.17}
\end{align}   
where  $\hat{v}^\e_j:= \varphi_{i_j}v^\e_j + (1-\varphi_{i_j})u_{0,\zeta,\nu}$ for a suitable $i_ j\in \{1,\dots,N\}$. It is clear from the definition that
\begin{equation}\label{eq:8.18}
|\hat{v}^\e_j| \le \mu\quad\hbox{in }\,B_\eta
\end{equation}
and  $\hat{v}^\e_j=u_{0,\zeta,\nu}$ in a neighborhood of $\partial B_\eta$. 
By the upper bounds in \ref{eq:bound1} and \ref{eq:bound2}, and by \eqref{eq:8.16}, we deduce that 
\begin{align*}
\mathcal{F}_{j,\e}(u_{0,\zeta,\nu},B_\eta\setminus K_\eta) <\eta^d\,.
\end{align*}
Since $v_j \to u$ in $L^{p(\cdot)}(B_{\eta\e}(x),\R^m)$, it follows that 
\begin{equation} \label{eq:8.19}
v^\e_j (\cdot) = v_j(x+\e\,\cdot) \to u(x+\e\,\cdot) \quad \textrm{in  } \,\,L^{p_\eps(\cdot)}(B_\eta,\R^m) \  \textrm{ as } \,\, j \to +\infty.
\end{equation}
Therefore, from \eqref{eq:8.17} and \eqref{eq:8.19} we have 
\begin{align}\nonumber
\limsup_{j\to +\infty}\mathcal{F}_{j,\e}(\hat{v}^\e_j,B_\eta) \leq{}& (1+\eta)\Big(\limsup_{j\to +\infty}\mathcal{F}_{j,\e}(v^\e_j, B_\eta) + \eta^d\Big)
\\
& + M_\eta  \int_{B_\eta}|u(x+\e\,y) - u_{0,\zeta,\nu}(y)|^{p_\eps(y)}\,\mathrm{d}y + \gamma_d\eta^{d+1}\,.
\label{eq:8.20}
\end{align}

\medskip

\noindent\textit{Step 2.2.} 
We now show that $\nabla \hat{v}_j^\e$ is small in $L^{p^-_\eps}$-norm for $j$ large and $\e$ small. By the definition of $\hat{v}^\e_j$ we have
\begin{align}
\|\nabla \hat{v}_j^\e\|_{L^{p^-_\eps}(B_\eta, \mathbb{R}^{m\times d})} &\leq C_\eta \| v_j^\e - u_{0,\zeta,\nu}\|_{L^{p^-_\eps}(B_\eta, \R^m)} +  \|\nabla v_j^\e\|_{L^{p^-_\eps}(B_\eta, \mathbb{R}^{m\times d})},\label{eq:8.21}
\end{align}
where the constant $C_\eta>0$ is an upper bound for $\|\nabla \varphi_{i_j}\|_{L^\infty(B_\eta, \R^m)}$. 

We now estimate separately the two terms in the right-hand side of \eqref{eq:8.21}. Concerning the first term, by \eqref{eq:8.19} we can find $j_1(\e) \geq  j_0(\e)$ such that, for $j \geq j_1(\e)$ and from \eqref{eq:8.11}, we have 
\begin{align}\nonumber
\| v^\e_j &- u_{0,\zeta,\nu}\|_{L^{p^-_\eps}(B_\eta, \R^m)}
\\
\label{eq:8.22}
&\leq  \|v^\e_j(\cdot) - u(x+\e\,\cdot)\|_{L^{p^-_\eps}(B_\eta,\R^m)} + \|u(x+\e\,\cdot) - u_{0,\zeta,\nu}(\cdot)\|_{L^{p^-_\eps}(B_\eta, \R^m)}
\leq \omega_I (\e),
\end{align}
where $\omega_I(\e)$ is independent of $j$ and $\omega_I(\e) \to 0$ as $\e \to 0^+$.
 
As for the second term in \eqref{eq:8.21}, by the definition of $v_j^\e$, the lower bound in \ref{eq:bound1}, and the positivity of $g_j$, for $\e$ small enough we have that 
\begin{align}\label{eq:8.23}
\int_{B_\eta}|\nabla v_j^\e|^{p_\eps(y)} \mathrm{d}y & \leq 
\e^{p^-_\eps-d} \int_{B_{\eta\e}(x)}|\nabla v_j |^{p(y)} \,\mathrm{d}y \nonumber \\
& \leq \frac{\e^{p^-_\e-d}}{\alpha} \int_{B_{\eta\e}(x)} f_j (y,\nabla v_j )\,\mathrm{d}y  \\
& \leq \frac{\e^{p^-_\e-1}}{\alpha}\cdot\frac{\mathcal{F}_j(v_j, B_{\eta\e}(x))}{\e^{d-1}} \nonumber \,.
\end{align}

Now, by \eqref{eq:8.12} there exists $\e_0>0$ such that for every $0<\e<\e_0$ satisfying \eqref{e:bordo} we can find $j_2(\e)\ge  j_1(\e)$ such that, taking into account also \eqref{eq:8.23}, we have
\begin{equation}\label{eq:8.24}
\int_{B_\eta}|\nabla v_j^\e|^{p_\eps(y)} \mathrm{d}y  \leq \frac{\gamma_{d-1}\eta^{d-1}\e^{p^-_\e-1}}{\alpha}(g_\infty(x,\zeta,\nu)+1)
\end{equation}
for  every $j \geq j_{2}(\e)$. Finally, collecting \eqref{eq:8.21}, \eqref{eq:8.22}, and \eqref{eq:8.24} we conclude that 
\begin{equation}\label{eq:8.25}
\|\nabla \hat{v}_j^\e\|_{L^{p^-_\e}(B_\eta, \mathbb{R}^{m\times d})} \leq \omega_{II}(\e)
\end{equation}
for every $0<\e<\e_0$ satisfying \eqref{e:bordo} and every $j \geq j_{ 2 }(\e)$, 
where $\omega_{II}(\e)$ is independent of $j$ and $\omega_{II}(\e) \to 0$ as $\e \to 0^+$.

\medskip

\noindent\textit{Step 2.3.} As a next step, we need to modify $\hat{v}_{j}^\e$ to make it piecewise constant.

Let $\zeta_1,\dots,\zeta_d$ be the coordinates of $\zeta$.
By \eqref{eq:8.10} for every $0<\e<\e_0$ satisfying \eqref{e:bordo} there exists an integer $N_{\e}>0$, with $\frac{1}{N_\e}< \mu$ and $\frac{1}{N_\e}<  |\zeta_i|$ for every $i$ with $\zeta_i\neq 0$, such that,
\begin{equation}\label{eq:8.26}
N_{\e} \to +\infty \quad\text{and}\quad \omega_{II}(\e) \, N_\e \to 0^+ \quad \textrm{as } \, \e \to {0^+}.
\end{equation}
Note that, by \eqref{eq:8.18}, we have $|\hat v_j^\e|<2\mu- \frac{1}{N_\e}$ in $B_\eta$.

Since by \eqref{eq:8.25} the functions $\hat{v}_j^\e$ are equibounded in $L^1(B_\eta;\R^m)$ for every fixed $\e$, by virtue of Theorem~\ref{thm:fspoincare} applied with $\theta:=N_\e \|\nabla \hat{v}_j^\e\|_{L^{1}(B_\eta, \mathbb{R}^{m\times d})}$ we can find a partition $(P_l^{\e,j})_{l=1}^\infty$ of $B_\eta$ made of sets of finite perimeter and a piecewise constant function $w_j^\e:=\sum_{l=1}^\infty b_l \chi_{P_l^{\e,j}}$ such that the following properties hold: for every $0<\e<\e_0$ satisfying \eqref{e:bordo} and for every $j \geq j_{ 2 }(\e)$
\begin{align}
&w_j^\e = u_{0,\zeta,\nu} \text{ in a neighborhood of  }\partial B_\eta,\label{eq:8.28} \\
&\|w_j^\e - \hat v_j^\e\|_{L^\infty(B_\eta,\R^m)} \leq \frac{1}{N_\e} < \mu, \label{eq:8.29} \\
&\|w_j^\e \|_{L^\infty(B_\eta,\R^m)} \leq  2\mu, \label{eq:8.30} \\
&\mathcal{H}^{d-1}((J_{w_j^\e}\setminus J_{\hat v_j^\e})\cap B_\eta) \leq \omega_{III} (\e), \label{eq:8.31}
\vphantom{ \frac{2\sqrt m}{N_\e}}
\end{align}
where $\omega_{III}(\e):=c(d,p)\omega_{II}(\e)N_\e$ is independent of $j$ and  $\omega_{III}(\e)\to 0^+$ as $\e\to 0^+$. Note that \eqref{eq:8.29} and \eqref{eq:8.31} follow from Theorem~\ref{thm:fspoincare}$(ii)$ and $(i)$, respectively.

\medskip

\noindent \textit{Step 2.4.} 
Recalling the definition of $\mathcal{F}_{j,\e}(\hat v_j^\e,B_\eta)$ (see \eqref{eq:8.15}) and taking into account \eqref{eq:8.20}, we have 
\begin{align}\label{eq:8.32}
\limsup_{j\to +\infty} & \int_{J_{\hat v_j^\e}\cap B_\eta} 
g_j (x+\e y,[\hat v_j^\e] (y),\nu_{\hat v_j^\e} (y) )
\,\mathrm{d}\mathcal{H}^{d-1} (y)\nonumber\\
&\leq (1+\eta)\Big(\limsup_{j\to +\infty}\mathcal{F}_{j,\e}(v^\e_j, B_\eta) + \eta^d\Big)
+ M_\eta  \int_{B_\eta}|u(x+\e\,y) - u_{0,\zeta,\nu}(y)|^{p_\eps(y)}\,\mathrm{d}y + \gamma_d\eta^{d+1}\,.
\end{align}
Moreover, with the upper bound in \ref{eq:bound1} and \eqref{eq:8.24}, the volume integral in the right hand side of \eqref{eq:8.32} can be estimated as 
\begin{equation*} 
\int_{B_\eta}f_j (x+\e y, \nabla v_j^\e (y) ) \,\mathrm{d}y \leq \beta\int_{B_\eta} (1+ |\nabla v_j^\e|^{p_\e(\cdot)})\,\mathrm{d}y \leq \beta\Big(\gamma_d \eta^d + \frac{\gamma_{d-1}\eta^{d-1}\e^{p^-_\e-1}}{\alpha}(g_\infty(x,\zeta,\nu)+1)\Big)
\end{equation*}
for every $0<\e<\e_0$ satisfying \eqref{e:bordo}  and every $j \geq j_{ 2 }(\e)$.

By \eqref{eq:8.15} again, this inequality and \eqref{eq:8.32} yield in particular that 
\begin{align}\nonumber
\limsup_{j\to +\infty} & \int_{J_{\hat v_j^\e}\cap B_\eta} 
g_j (x+\e y,[\hat v_j^\e] (y),\nu_{\hat v_j^\e} (y))\,\mathrm{d}\mathcal{H}^{d-1} (y)\\\nonumber
&\leq  (1+\eta) \limsup_{j \to + \infty}\int_{J_{ v_j^\e}\cap B_\eta}
g_j (x+\e y,[v_j^\e] (y),\nu_{v_j^\e} (y))\,\mathrm{d}\mathcal{H}^{d-1} (y) \\\label{eq:8.33}
& + 2 \beta\Big(\gamma_d \eta^d+ \frac{\gamma_{d-1}\eta^{d-1}\e^{p^-_\e-1}}{\alpha}(g_\infty(x,\zeta,\nu)+1)\Big) \\
& + M_\eta  \int_{B_\eta}|u(x+\e\,y) - u_{0,\zeta,\nu}(y)|^{p_\eps(y)}\,\mathrm{d}y  + c(d)\eta^d \,, \nonumber
\end{align}
where $c(d):= 2+\gamma_d$.

Now, rewriting in terms of $v_j$ the surface integral in the right hand side and combining
with \eqref{estimate:1} and \eqref{eq:8.33} we obtain
\begin{align}
&\limsup_{j \to +\infty} \int_{J_{\hat v_j^\e}\cap B_\eta} 
g_j (x+\e y,[\hat v_j^\e] (y),\nu_{\hat v_j^\e} (y)) d\mathcal{H}^{d-1} (y)
\nonumber \\
&\leq (1+\eta)^2 \frac{1}{\e^{d-1}}\,\mathcal{F}(u, B_{\eta\e}(x))
+ 2 \eta^{d}\e   + 2\beta\Big(\gamma_d \eta^d+ \frac{\gamma_{d-1}\eta^{d-1}\e^{p^-_\e-1}}{\alpha}(g_\infty(x,\zeta,\nu)+1)\Big)
\label{eq:8.34} \\
&+ M_\eta  \int_{B_\eta}|u(x+\e\,y) - u_{0,\zeta,\nu}(y)|^{p_\eps(y)}\,\mathrm{d}y +  c(d)\eta^d\,. \nonumber
\end{align}
We now estimate the left-hand side in \eqref{eq:8.34}. Exploiting the assumptions  \ref{eq:bound2}, \ref{ass-g7}, \ref{ass-g2}, and the properties of $\hat v_j^\e$ and $w_j^\e$ we claim that 
\begin{align}\label{eq:8.36}
&\int_{J_{w_j^\e}\cap B_\eta}g_j (x+\e y,[w_j^\e](y),\nu_{w_j^\e}(y))\,\mathrm{d}\mathcal{H}^{d-1}(y)\nonumber\\
&\leq \int_{J_{\hat v_j^\e}\cap B_\eta}
g_j (x+\e y,[\hat v_j^\e](y),\nu_{\hat v_j^\e}(y) )\,\mathrm{d}\mathcal{H}^{d-1}(y)+\omega_{IV}(\e)+\omega_{V}(\e)\,, \nonumber\\
\end{align}
where  $\omega_{IV}(\e)$ and $\omega_{V}(\e)$ are independent of $j$ and tend to  $0^+$ as $\e \to 0^+$. There, the key estimate is 
\begin{align*}
& |g_j (x+\eps y,[\hat v_j^\eps](y),\nu_{\hat v_j^\eps}(y)) - g_j (x+\eps y,[w_j^\eps](y),\nu_{w_j^\eps}(y))| \\
&\leq \omega_2(|[\hat v_j^\eps] (y) -[w_j^\eps] (y)|)
\big(g_j (x+\eps y,[\hat v_j^\eps] (y),\nu_{\hat v_j^\eps} (y)) 
+ g_j (x+\eps y,[w_j^\eps] (y),\nu_{w_j^\eps} (y))\big) \\
&\leq 4 \beta^2 \omega_2(2\|\hat v_j^\eps - w_j^\eps\|_{L^{\infty}(B_\eta\!,\R^m)})\,, 
\end{align*}
for $\hs^{n-1}$-a.e. $y \in J_{\hat v_j^\eps}\cap J_{w_j^\eps}$. The claim follows then from \eqref{eq:8.29}, \eqref{eq:8.31} and the bounds on $g_j$ .

Now, \eqref{eq:8.36} together with \eqref{eq:8.34} gives 

\begin{align}
\limsup_{j\to +\infty} & \int_{J_{w_j^\e}\cap B_\eta} 
g_j (x+\e y,[w_j^\e] (y),\nu_{w_j^\e} (y))\,\mathrm{d}\mathcal{H}^{d-1} (y)\nonumber\\
&\leq  (1+\eta)^2 \frac{1}{\e^{d-1}}\,\mathcal{F}(u, B_{\eta\e}(x))
 + 2 \eta^d \e  + \omega_{IV}(\e)+\omega_{V}(\e) \label{eq:8.36bis}\\
& + 2\beta\Big(\gamma_d \eta^d+ \frac{\gamma_{d-1}\eta^{d-1}\e^{p^-_\e-1}}{\alpha}(g_\infty(x,\zeta,\nu)+1)\Big) \nonumber \\
&+ M_\eta  \int_{B_\eta}|u(x+\e\,y) - u_{0,\zeta,\nu}(y)|^{p_\eps(y)}\,\mathrm{d}y + c(d)\eta^d\,. \nonumber
\end{align}
Defining $z_j^\e(y):= w_j^\e((y-x)/\e)$ for every $y \in B_{\eta\e} (x)$, we clearly have that $z_j^\e \in SBV_{\mathrm{pc}} (B_{\eta\e} (x),\R^m)$ and $z_j^\e = u_{x, \zeta, \nu}$ in a neighborhood of $\partial B_{\eta\e} (x)$.
Then, rewriting \eqref{eq:8.36bis} in terms of the functions $z_j^\e$ we find

\begin{align*} 
\limsup_{j \to +\infty} \frac{1}{(\eta\e)^{d-1}} &{\bf m}^{PC}_{G_j}(u_{x,\zeta,\nu}, B_{\eta\e}(x)) 
\le \limsup_{j \to +\infty} \frac{1}{(\eta\e)^{d-1}} {\bf m}^{PC}_{G_j}(u_{x,\zeta,\nu}, B_{\eta\e}(x)) \nonumber\\
& \leq \limsup_{j\to +\infty} \frac{1}{(\eta\e)^{d-1}}\int_{J_{z_j^\e}\cap B_{\eta\e}(x)} 
g_j (y,[z_j^\e] (y),\nu_{z_j^\e} (y))\,\mathrm{d}\mathcal{H}^{d-1} (y)\nonumber\\
&\leq  (1+\eta)^2 \frac{1}{(\eta\e)^{d-1}}\,\mathcal{F}(u, B_{\eta\e}(x))
 + 2 \eta\e  + \frac{\omega_{IV}(\e)}{\eta^{d-1}} + \frac{\omega_{V}(\e)}{\eta^{d-1}} \nonumber\\
& +2\beta\Big(\gamma_d \eta+ \frac{\gamma_{d-1}\e^{p^-_\e-1}}{\alpha}(g_\infty(x,\zeta,\nu)+1)\Big) \\
& + \frac{M_\eta}{\eta^{d-1}}  \int_{B_1}|u(x+\e\,y) - u_{0,\zeta,\nu}(y)|^{p_\eps(y)}\,\mathrm{d}y + c(d)\eta\,. \nonumber
\end{align*}
Finally, dividing by $\gamma_{d-1}$, taking the limsup as $\e \to 0^+$ and using \eqref{eq:g''}, 
\eqref{eq:8.11}, and \eqref{eq:8.12}, 
 we obtain
$$
g'' (x,\zeta,\nu) \leq (1+\eta)^2  g_\infty (x,\zeta,\nu)  + C \eta,
$$
with $C:= (2\beta \gamma_d + c(d))/\gamma_{d-1}$.
Recalling the definition of $\zeta$ and $\nu$, we obtain that
\begin{equation*} 
g'' (x,[u](x),\nu_{u}(x)) \leq (1+\eta)^2 g_\infty (x,[u](x),\nu_u(x))  
+ C \eta 
\end{equation*}
holds true for $\hs^{n-1}$-a.e.\ $x\in J_u \cap A$. Taking the limit as $\eta\to 0^+$ 
we get 
\begin{equation*} 
g'' (x,[u](x),\nu_{u}(x)) \leq g_\infty(x,[u](x),\nu_u(x)) 
\end{equation*}
for $\hs^{n-1}$-a.e.\ $x\in J_u \cap A$,
thus proving \eqref{eq:gineq2} for $u\in SBV^{p(\cdot)}(A,\R^m)\cap L^{\infty}(A,\R^m)$.

Finally, since by definition $g'\leq g''$, combining \eqref{eq:gineq1} and \eqref{eq:gineq2} we get \eqref{eq: g_infty=g}-\eqref{eq:gpc}. This concludes the proof.

\end{proof}

\EEE

\appendix \section{A $\Gamma$-convergence result with weaker growth conditions from above} \label{sec:appendix}

In this section we will prove a $\Gamma$-convergence result for energies whose surface densities satisfy a weaker assumption than {\rm\ref{eq:bound2}} of Section~\ref{sec:gammaconv}. To do this, we will also take advantage an integral representation result on $SBV^{p(\cdot)}$ (see Theorem~\ref{thm: int-representation-sbv}) via a perturbation argument.

Let $(f_j)_{j\in\N}$ and $(g_j)_{j\in\N}$ be sequences of functions satisfying {\rm\ref{ass-f1}}-{\rm\ref{eq:bound1}} and {\rm\ref{ass-g1}}, {\rm\ref{ass-g4}}, {\rm\ref{ass-g7}}, respectively. In place of {\rm\ref{eq:bound2}}, we require each $g_j$ to comply with the additional property
\begin{enumerate}[font={\normalfont},label={($g3^\prime$)}]
\item(lower and upper bound) for every $x\in \R^d$, $\zeta\in \R^d_0$, and $\nu \in \Sph^{d-1}$
\begin{equation*}
\alpha \le g(x,\zeta,\nu) \le \beta (1+|\zeta|)\,,
\end{equation*}\label{eq:bound2bis}
\end{enumerate}
together with {\rm\ref{ass-g3}}.

Correspondingly, we define the functionals $\mathcal{E}_j: L^0(\Omega;\R^m)\times\mathcal{A}(\Omega)\to[0,+\infty]$ as
\begin{equation}
\mathcal{E}_j(u,A):=
\begin{cases}
\displaystyle \int_A f_j\big(x,   \nabla u(x)   \big)\, {\rm d}x +\int_{J_u\cap A}g_j(x,[u](x),\nu_u(x))\, {\rm d}\mathcal{H}^{d-1}(x)\,, & \mbox{ if $u\lfloor_{A}\in GSBV^{p(\cdot)}(A;\R^m)$,}  \\
+\infty\,, & \mbox{ otherwise.}
\end{cases}
\label{eq:unperturben}
\end{equation}

\subsection{Integral representation: the $SBV^{p(\cdot)}$ case}

In this section we discuss the minor modifications needed in order to obtain an integral representation result for functionals $\mathcal{F}\colon SBV^{p(\cdot)}(\Omega;\R^m) \times \mathcal{B}(\Omega) \to  [0,+\infty)$, satisfying assumptions \ref{assH1}-\ref{assH3} and the following
\begin{enumerate}[font={\normalfont},label={(${\rm H}_4'$)}]
\item  there exist $0 < \alpha  < \beta $ such that for any $u \in SBV^{p(\cdot)}(\Omega;\R^m)$ and $B \in \mathcal{B}(\Omega)$  we have
\begin{equation}
\begin{split}
\alpha \bigg(\int_{ B  } |\nabla u|^{p(x)}  \dx    +   \int_{J_u \cap B}(1+|[u]|)\,\mathrm{d}\mathcal{H}^{d-1}\bigg) & \le \mathcal{F}(u,B) \\
& \le \beta \bigg(\int_{ B } (1 + |\nabla u|^{p(x)})    \dx  +  \int_{J_u \cap B}(1+|[u]|)\,\mathrm{d}\mathcal{H}^{d-1}\bigg)\,.
\end{split}
\end{equation} \label{assH4'}
\end{enumerate}

For every $u \in SBV^{p(\cdot)}(\Omega;\R^m)$ and $A \in \mathcal{A}(\Omega)$ we define
\begin{equation}\label{eq: general minimizationsbv}
\mathbf{m}_{\mathcal{F}}(u,A) = \inf_{v \in SBV^{p(\cdot)}(\Omega;\R^m)} \  \lbrace \mathcal{F}(v,A): \ v = u \ \text{ in a neighborhood of } \partial A \rbrace\,.
\end{equation}

The main result of this section is the following integral representation theorem.

\begin{theorem}[Integral representation in $SBV^{p(\cdot)}$]\label{thm: int-representation-sbv}
Let $\Omega \subset \R^d$ be open, bounded with Lipschitz boundary, let $m \in \N$. Let $p:\Omega\to(1,+\infty)$ be a variable exponent complying with {\rm\ref{assP1}}-{\rm\ref{assP2}}, and suppose that  $\mathcal{F}\colon SBV^{p(\cdot)}(\Omega;\R^m)  \times \mathcal{B}(\Omega) \to [0,+\infty)$ satisfies {\rm\ref{assH1}}--{\rm\ref{assH3}} and {\rm\ref{assH4'}}. Then 
$$\mathcal{F}(u,B) = \int_B f\big(x,u(x),\nabla u(x)\big)  \, {\rm d}x +    \int_{J_u\cap  B} g\big(x,u^+(x),u^-(x),\nu_u(x)\big)\,  {\rm d}  \mathcal{H}^{d-1}(x)$$
for all $u \in  SBV^{p(\cdot)}(\Omega;\R^m)$  and   $B \in \mathcal{B}(\Omega)$, where $f$ is given  by
\begin{align}\label{eq:fdef-sbv}
f(x_0,u_0,\xi) = \limsup_{\eps \to 0} \frac{\mathbf{m}_{\mathcal{F}}(\ell_{x_0,u_0,\xi},B_\eps(x_0))}{\gamma_d\eps^{d}}
\end{align}
for all $x_0 \in \Omega$, $u_0 \in \R^m$, $\xi \in \mathbb{R}^{m \times d}$ and $\ell_{x_0,u_0,\xi}$ as in \eqref{eq: elastic competitor}, $g$ is given by 
\begin{align}\label{eq:gdef-sbv}
g(x_0,a,b,\nu) = \limsup_{\eps \to 0} \frac{\mathbf{m}_{\mathcal{F}}(u_{x_0,a,b,\nu},B_\eps(x_0))}{\gamma_{d-1}\eps^{d-1}}
\end{align}
for all $ x_0  \in \Omega$,  $a,b \in \R^m$, $\nu \in \mathbb{S}^{d-1}$ and $u_{x_0,a,b,\nu}$ as in \eqref{eq: jump competitor}, and $\mathbf{m}_{\mathcal{F}}$ is defined in \eqref{eq: general minimizationsbv}.
\end{theorem} 

The proof of Theorem~\ref{thm: int-representation-sbv} can be obtained by adapting the argument of Theorem~\ref{thm: int-representation-gsbv}, which concerns with $GSBV^{p(\cdot)}$ functions. For this, Lemma~\ref{lemma: F=m}, Lemma~\ref{lemma: sameminbulk} and Lemma~\ref{lemma: sameminsurf} are replaced by the corresponding $SBV^{p(\cdot)}$ versions, Lemma~\ref{lemma: F=msbv}, Lemma~\ref{lemma: sameminbulksbv} and Lemma~\ref{lemma: sameminsurfsbv} below, respectively. We will briefly list the main changes in the proofs due to the different assumption \ref{assH4'}.

\begin{lemma}\label{lemma: F=msbv}
Let $p:\Omega\to(1,+\infty)$ be a variable exponent satisfying {\rm\ref{assP1}}-{\rm\ref{assP2}}. Suppose that $\mathcal{F}$ satisfies {\rm\ref{assH1}}--{\rm\ref{assH3}} and {\rm\ref{assH4'}}. Let $u \in SBV^{p(\cdot)}(\Omega;\R^m)$ and  $\mu$ be defined as
\begin{equation}
\mu:= \mathcal{L}^d\lfloor_{\Omega}\,\, +\,\, (1+|u^+-u^-|)\mathcal{H}^{d-1}\lfloor_{J_u \cap \Omega}\,. 
\label{eq: measuremubis}
\end{equation}
Then for $\mu$-a.e.\ $x_0 \in \Omega$ we have
 $$\lim_{\eps \to 0}\frac{\mathcal{F}(u,B_\eps(x_0))}{\mu(B_\eps(x_0))} =  \lim_{\eps \to 0}\frac{\mathbf{m}_{\mathcal{F}}(u,B_\eps(x_0))}{\mu(B_\eps(x_0))}.$$
\end{lemma}

\begin{proof}
The only needed modification concerns the proof of Lemma~\ref{lemma: F=m*}. Indeed, under assumption \ref{assH4'}, for $u\in SBV^{p(\cdot)}(\Omega;\R^m)$ the same construction provides a sequence $v^{\delta,n}$ in $SBV^{p(\cdot)}(\Omega;\R^m)$ such that 
\begin{equation}
 \sup_{n \in \N} \left(\int_\Omega |\nabla v^{\delta,n}(x)|^{p(x)}\,\mathrm{d}x  + \int_{J_{ v^{\delta,n} }}|[v^{\delta,n}]|\,\mathrm{d}\mathcal{H}^{d-1}\right)<+\infty\,.
\label{eq: equibddvdeltabis}
\end{equation}
Then, an analogous compactness argument, based on \cite[Theorem~2.1]{Amb} yields $v^\delta \in SBV^{p^-}(\Omega;\R^m)$, which can be improved to $v^\delta \in SBV^{p{(\cdot)}}(\Omega;\R^m)$ by using Ioffe's theorem and the weak convegence of the gradients, exactly as in Lemma~\ref{lemma: F=m*}. Finally, assumption \ref{assH4'} does not change \eqref{eq: to show-flaviana2-NNN}.
\end{proof}

Note that the Fundamental estimate \eqref{eq: assertionfund}, proven with Lemma~\ref{lemma: fundamental estimate}, still holds if we replace \ref{assH4} by \ref{assH4'}.

\begin{lemma}\label{lemma: sameminbulksbv}
Let $p:\Omega\to(1,+\infty)$ be a Riemann-integrable variable exponent satisfying {\rm\ref{assP1}}. Suppose that $\mathcal{F}$ satisfies {\rm\ref{assH1}} and  {\rm\ref{assH3}},{\rm\ref{assH4'}}  and let $u \in SBV^{p(\cdot)}(\Omega;\R^m)$.  Then for $\mathcal{L}^{d}$-a.e.\ $x_0 \in \Omega$  we have
\begin{align}\label{eq: samemin-bulksbv}
  \lim_{\eps \to 0}\frac{\mathbf{m}_{\mathcal{F}}(u,B_\eps(x_0))}{\gamma_{d}\eps^{d}} =  \limsup_{\eps \to 0}\frac{\mathbf{m}_{\mathcal{F}}(\bar{u}^{\rm bulk}_{x_0},B_\eps(x_0))}{\gamma_{d}\eps^{d}}. 
 \end{align}
\end{lemma}

\begin{proof}
The proof of ``$\leq$'' inequality in \eqref{eq: samemin-bulksbv} can be obtained with the same construction of Lemma~\ref{lem:bulk1} applied to the sequence $(u_\eps)$ complying with Lemma~\ref{lem:lemma2fonseca}$(i)$-$(iii)$ and $(i)'$, $(iii)'$.

Applying the Fundamental estimate with the same choice of sets as in \eqref{eq:choiceofsets} and by assumption \ref{assH4'}, we get
\begin{equation*}
\mathcal{F}(u_\eps,  C_{\eps,\theta}(x_0)) \le \beta \int_{C_{\eps,\theta}(x_0)} (1+  |\nabla u_\eps|^{p(x)})\,\mathrm{d}x + \beta \int_{J_{u_\eps} \cap  C_{\eps,\theta}(x_0)} (1+|[u_\eps]|) \,\mathrm{d}\mathcal{H}^{d-1}\,,
\end{equation*}
whence by Lemma~\ref{lem:lemma2fonseca}$(iii)$, $(iii)'$, \eqref{eq:stimaintermedia} we obtain the analogous of \eqref{eq:5.14}, and this concludes the proof of the first inequality in \eqref{eq: samemin-bulksbv}.

The reverse inequality in \eqref{eq: samemin-bulksbv} can be proved following the argument of Lemma~\ref{lem:bulk2}. For this, we first notice that since $u_\eps$ satisfies Lemma~\ref{lem:lemma2fonseca}$(i)'$, in addition to \eqref{eq:5.15} we may require that
\begin{equation}
\lim_{\eps\to0} \eps^{-d} \int_{\partial B_{s\eps}(x_0)} |u^+-u^-_\eps|\,\mathrm{d}\mathcal{H}^{d-1}=0\,,
\label{eq:7.7}
\end{equation}
where $u^-_\eps$ and $u^+$ denote the inner and outer traces at $\partial B_{s\eps}(x_0)$ of $u_\eps$ and $u$, respectively. Then, estimates \eqref{eq:5.14} and \eqref{eq:5.19} (with the additional term $\beta \int_{\partial B_{s\eps}(x_0)} |u^+-u^-_\eps|\,\mathrm{d}\mathcal{H}^{d-1}$ in the left hand side) can be established. Finally, combining \eqref{eq:5.14}, \eqref{eq:5.15}, \eqref{eq:7.7} and the fact that $s\eps\leq (1-3\theta)\eps$, we obtain also \eqref{eq:5.20}. This will suffice to conclude the argument of Lemma~\ref{lem:bulk2} and then the proof of the inequality ``$\geq$'' in \eqref{eq: samemin-bulksbv}.
\end{proof}

\begin{lemma}\label{lemma: sameminsurfsbv}
Let $p:\Omega\to(1,+\infty)$ be a variable exponent satisfying {\rm\ref{assP1}}-{\rm\ref{assP2}}. Suppose that $\mathcal{F}$ satisfies {\rm\ref{assH1}} and  {\rm\ref{assH3}},{\rm\ref{assH4'}}  and let $u \in SBV^{p(\cdot)}(\Omega;\R^m)$.  Then for $\mathcal{H}^{d-1}$-a.e.\ $x_0 \in J_u$  we have
\begin{align}\label{eq: samemin-surfsbv}
  \lim_{\eps \to 0}\frac{\mathbf{m}_{\mathcal{F}}(u,B_\eps(x_0))}{\gamma_{d-1}\eps^{d-1}} =  \limsup_{\eps \to 0}\frac{\mathbf{m}_{\mathcal{F}}(\bar{u}^{\rm surf}_{x_0},B_\eps(x_0))}{\gamma_{d-1}\eps^{d-1}}. 
  \end{align}
\end{lemma}

\begin{proof}
The construction of Lemma~\ref{lem:jump1} can be performed using the sequence $(\bar{u}_\eps)$ which complies with Lemma~\ref{lem:Lemma3fon}$(i)$-$(iv)$, \eqref{eq:25bis} and \eqref{eq:(7.8)_2}, thus obtaining the analogous of estimates \eqref{eq:6.22} and \eqref{eq:6.23} where the constant $\beta$ is replaced by $\beta(1+|[\bar{u}^{\rm surf}_{x_0}]|)$. Now, taking into account \ref{assH4'}, \eqref{eq:24}$(ii)$, \eqref{eq:stimapersalto} and \eqref{eq:(7.8)_2}, we obtain the analogous of \eqref{eq:6.24}; i.e.,
\begin{equation*}
\mathop{\lim\sup}_{\eps\to0}\frac{\mathcal{F}(\bar{u}_\eps,  C_{\eps,\theta}(x_0))}{\gamma_{d-1}\eps^{d-1}} \le \beta(1+|[\bar{u}^{\rm surf}_{x_0}]|)(1-(1-4\theta)^{d-1})\,.
\end{equation*}  
With this, we can easily infer the upper inequality in \eqref{eq: samemin-surfsbv}.

As for the reverse inequality, given $(\bar{u}_\eps)$ as above, by \eqref{eq:25bis} we may require, in addition to \eqref{eq:6.25}, also the property
\begin{equation}
\lim_{\eps\to0} \eps^{-(d-1)} \int_{\partial B_{s\eps}(x_0)} |u^+-\bar{u}^-_\eps|\,\mathrm{d}\mathcal{H}^{d-1}=0\,,
\label{eq:7.15}
\end{equation}
where $u^+$ and $\bar{u}_\eps^-$ have the same meaning as in Lemma~\ref{lemma: sameminbulksbv}. Then we repeat the argument of Lemma~\ref{lem:jump2}, where \eqref{eq:6.29} is now replaced by
\begin{equation*}
\begin{split}
\mathcal{F}(z_\eps, B_{(1-\theta)\eps}(x_0))\leq & \mathbf{m}_{\mathcal{F}}\big(u,B_{\sigma_\eps}(x_0)\big) + \gamma_{d-1}\eps^{d} + \beta \mathcal{H}^{d-1}\left((\{\bar{u}_\eps\neq u\}\cup J_u\cup J_{\bar{u}_\eps})\cap \partial B_{\sigma_\eps}(x_0)\right) \\
& + \beta  \int_{\partial B_{s\eps}(x_0)} |u^+-\bar{u}^-_\eps|\,\mathrm{d}\mathcal{H}^{d-1} +\mathcal{F}(\bar{u}_\eps, C_{\eps,\theta}(x_0))\,. 
\end{split}
\end{equation*}
Now, as a consequence of \eqref{eq:6.25}, \eqref{eq:7.14}, \eqref{eq:7.15} and the fact that $\sigma\leq (1-3\theta)$ we then obtain
\begin{equation*}
\begin{split}
\mathop{\lim\sup}_{\eps\to0} \frac{\mathcal{F}(z_\eps, B_{(1-\theta)\eps}(x_0))}{\gamma_{d-1}\eps^{d-1}} \leq & (1-3\theta)^{d-1}\mathop{\lim\sup}_{\eps\to0} \frac{ \mathbf{m}_{\mathcal{F}}\big(u,B_{\eps}(x_0)\big)}{\gamma_{d-1}\eps^{d-1}}\\
& + \beta(1+|[\bar{u}^{\rm surf}_{x_0}]|)(1-(1-4\theta)^{d-1})\,,
\end{split}
\end{equation*} 
which corresponds to \eqref{eq:6.29}. The estimate \eqref{eq:6.31} now reads
\begin{equation*}
\mathop{\lim\sup}_{\eps\to0}\frac{\mathcal{F}(\bar{u}^{\rm surf}_{x_0}, C_{\eps,\theta}(x_0))}{\gamma_{d-1}\eps^{d-1}} \leq \beta(1+|[\bar{u}^{\rm surf}_{x_0}]|)(1-(1-4\theta)^{d-1})\,,
\end{equation*} 
whence the conclusion follows exactly in the same way as in Lemma~\ref{lem:jump2}. We omit further details.
\end{proof}

\subsection{$\Gamma$-convergence}

Let $\sigma>0$. We define the family of perturbed functionals $\mathcal{E}_j^\sigma: L^0(\Omega;\R^m)\times\mathcal{A}(\Omega)\to[0,+\infty]$, $j\in\N$, as
\begin{equation}
\mathcal{E}_j^\sigma(u,A):=
\begin{cases}
\displaystyle \int_A f_j\big(x,   \nabla u(x)   \big)\, {\rm d}x +\int_{J_u\cap A}g_j^\sigma(x,[u](x),\nu_u(x))\, {\rm d}\mathcal{H}^{d-1}(x)\,, & \mbox{ if $u\lfloor_{A}\in SBV^{p(\cdot)}(A;\R^m)$,}  \\
+\infty\,, & \mbox{ otherwise, }
\end{cases}
\label{eq:perturben}
\end{equation}
where 
\begin{equation}
g_j^\sigma(x,\zeta,\nu):=g_j(x,\zeta,\nu)+\sigma|\zeta|\,.
\label{eq:gjsigma}
\end{equation}

First, we prove a $\Gamma$-convergence result for the perturbed functionals $\mathcal{E}_j^\sigma$.

\begin{theorem}[$\Gamma$-convergence of perturbed functionals]\label{th: gammasbv}
 Let $\Omega \subset \R^d$ be open.  Let $(f_j)_j$ and $(g_j)_j$ be sequences of functions satisfying {\rm\ref{ass-f1}}-{\rm\ref{eq:bound1}} and {\rm\ref{ass-g1}}, {\rm\ref{ass-g4}}, {\rm\ref{eq:bound2bis}}, {\rm\ref{ass-g7}}, {\rm\ref{ass-g3}}, respectively.  Let $\sigma>0$ and $\mathcal{E}_j^\sigma \colon SBV^{p(\cdot)}(\Omega;\R^m) \times \mathcal{A}(\Omega) \to [0,+\infty)$ be the sequence of functionals given in \eqref{eq:perturben}.  Then, there exists a functional $\mathcal{E}^\sigma\colon  SBV^{p(\cdot)}(\Omega;\R^m)\times \mathcal{A}(\Omega) \to [0,+\infty)$ and a subsequence (not relabeled) such that
$$\mathcal{E}^\sigma(\cdot,A) =\Gamma\text{-}\lim_{j \to \infty} \mathcal{E}_j^\sigma(\cdot,A) \ \ \ \ \text{with respect to convergence in measure on $A$} $$
for all $A \in  \mathcal{A}(\Omega) $. 
Let $f_\infty^\sigma$ and $g_\infty^\sigma$ be defined as
\begin{align}\label{eq:fsigmadef-sbv}
f_\infty^\sigma(x_0,u_0,\xi) = \limsup_{\eps \to 0} \frac{\mathbf{m}_{\mathcal{E}^\sigma}(\ell_{x_0,u_0,\xi},B_\eps(x_0))}{\gamma_d\eps^{d}}
\end{align}
for all $x_0 \in \Omega$, $u_0 \in \R^m$, $\xi \in \mathbb{R}^{m \times d}$,  and 
\begin{align}\label{eq:gsigmadef-sbv}
g_\infty^\sigma(x_0,\zeta,\nu) = \limsup_{\eps \to 0} \frac{\mathbf{m}_{\mathcal{E}^\sigma}(u_{x_0,\zeta,0,\nu},B_\eps(x_0))}{\gamma_{d-1}\eps^{d-1}}
\end{align}
for all $ x_0  \in \Omega$,  $\zeta \in \R^m$, and $\nu \in \mathbb{S}^{d-1}$, where $\mathbf{m}_{\mathcal{E}^\sigma}$ is as in \eqref{eq: general minimizationsbv} with $\mathcal{F}=\mathcal{E}^\sigma$.

Then, for every $u\in SBV^{p(\cdot)}(\Omega;\R^m)$ and  $A\in \mathcal{A}(\Omega)$ we have that
\begin{align}\label{eq: representationsbv}
\mathcal{E}^\sigma(u,A)= \int_A f_{\infty}^\sigma\big(x,u(x),   \nabla u(x)   \big)\, {\rm d}x +\int_{J_u\cap A}g_{\infty}^\sigma(x,[u](x),\nu_u(x))\, {\rm d}\mathcal{H}^{d-1}(x)\,.
\end{align}
\end{theorem}

\begin{proof}
The proof of Theorem~\ref{th: gammasbv} can be obtained along the lines of the argument of Theorem~\ref{th: gamma}. We then briefly sketch the proof, referring the reader to Theorem~\ref{th: gamma} for further details.

 We start by observing that some  properties of the $\Gamma$-liminf and $\Gamma$-limsup with respect to the topology of the convergence in measure, established in Lemma~\ref{eq: liminflimsup-prop} for functionals $\mathcal{F}_j$, still hold true for $\mathcal{E}_j^\sigma$. To this end, we define 
\begin{align}\label{eq: liminf-limsupsbv}
(\mathcal{E}^\sigma)'(u,A)&:=\Gamma-\liminf_{n \to \infty} \mathcal{E}_j^\sigma(u,A)   = \inf \big\{ \liminf_{j \to \infty} \mathcal{E}_j^\sigma(u_j,A):   \  u_j \to  u \hbox{ in measure on }  A  \big\}, \notag \\
(\mathcal{E}^\sigma)''(u,A) &:= \Gamma-\limsup_{n \to \infty} \mathcal{E}_j^\sigma(u,A)  = \inf \big\{  \limsup_{j \to \infty} \mathcal{E}^\sigma_j(u_j,A):   \ u_j \to  u \hbox{ in measure on }  A \big\}
\end{align} 
for all $u \in SBV^{p(\cdot)}(\Omega;\R^m)$ and $A \in \mathcal{A}(\Omega)$.

Then the analogous of assertions $(i)$, $(iii)$ and $(iv)$ of Lemma~\ref{eq: liminflimsup-prop} still hold true for $(\mathcal{E}^\sigma)'$ and $(\mathcal{E}^\sigma)''$, since the arguments are based on Lemma~\ref{lem:cdmsz} and Lemma~\ref{lemma: fundamental estimate}. 
Setting
$$ \mathcal{H}(u,A):= \int_A |\nabla u|^{p(\cdot)}\,\mathrm{d}x+ \int_{J_u \cap A}(1+|[u]|)\,\mathrm{d}\mathcal{H}^{d-1}\,,$$
we only have to check that 

\begin{equation}
\begin{split}
& \min\{\alpha,\sigma\} \mathcal{H}(u,A) \le (\mathcal{E}^\sigma)'(u,A) \leq (\mathcal{E}^\sigma)''(u,A) \le (\beta+\sigma) \mathcal{H}(u,A) + \beta \mathcal{L}^d(A)\,.
\end{split}
\label{eq:(ii)'}
\end{equation}
The upper bound for $(\mathcal{E}^\sigma)''$ in \eqref{eq:(ii)'} can be inferred choosing the constant sequence $u_j=u$ in \eqref{eq: liminf-limsupsbv} and taking into account {\rm\ref{eq:bound1}}, the definition of $g_j^\sigma$ (equation \eqref{eq:gjsigma}) together with {\rm\ref{eq:bound2bis}}.
For what concerns the lower bound in \eqref{eq:(ii)'}, we consider an (almost) optimal sequence $(v_j)_j$ in \eqref{eq: liminf-limsupsbv}. Then, with {\rm\ref{eq:bound1}}, \eqref{eq:gjsigma} and {\rm\ref{eq:bound2bis}} we get
\begin{equation*}
 \sup_{j \in \N} \min\{\alpha,\sigma\} \mathcal{H}(v_j,A)<+\infty\,.
\end{equation*}
Now, since $v_j\to u$ in measure on $A$, we may appeal to the closure property of $SBV$ (see, e.g., \cite[Theorem~4.7]{Ambrosio-Fusco-Pallara:2000}). Then, by arguing as in the proof of Lemma~\ref{lemma: F=m*} and exploiting the lower semicontinuity inequalities
\begin{equation*}
\begin{split}
\int_A |\nabla u(x)|^{p(x)}\,\mathrm{d}x &\leq \displaystyle\mathop{\lim\inf}_{j\to+\infty} \int_A|\nabla v_j(x)|^{p(x)}\,\mathrm{d}x <+\infty\,, \\
\int_{J_{u} \cap A}\theta(|[u]|)\,\mathrm{d}\mathcal{H}^{d-1} & \leq \displaystyle\mathop{\lim\inf}_{j\to+\infty}\int_{J_{v_j} \cap A}\theta(|[v_j]|)\,\mathrm{d}\mathcal{H}^{d-1}\,,
\end{split}
\end{equation*}
for any concave function $\theta:(0,+\infty)\to(0,+\infty)$, we easily obtain the lower bound.

The existence of the $\Gamma$-limit is still a consequence of the abstract result \cite[Theorem 16.9]{DalMaso:93}, in view of the inner regularity of both  $(\mathcal{E}^\sigma)'$ and  $(\mathcal{E}^\sigma)''$. Since \eqref{eq:(ii)'} implies \ref{assH4'}, the functional $\mathcal{E}^\sigma=(\mathcal{E}^\sigma)'=(\mathcal{E}^\sigma)''$ satisfies all the assumptions of Theorem~\ref{thm: int-representation-sbv}. This concludes the proof.
\end{proof}

Now, we are in position to deduce the $\Gamma$-convergence result for the family of functionals $\mathcal{E}_j$, defined in \eqref{eq:unperturben}. The argument of the proof is analogous to that of \cite[Theorem~5.1]{CDMSZ}, but with some simplifications due to the fact that, by virtue of Theorem~\ref{th: gammasbv}, we do not need to use the $\Gamma$-convergence of the restrictions to $L^{p(\cdot)}$ of our functionals.

\begin{theorem}\label{thm:perturbedgamma}
Let $\Sigma$ be a countable subset of $(0,+\infty)$, with $0\in\overline{\Sigma}$. Assume that for every $\sigma\in \Sigma$ there exists a functional  $\mathcal{E}^\sigma: L^0(\Omega;\R^m)\times\mathcal{A}(\Omega)\to[0,+\infty]$ such that for every $A\in\mathcal{A}(\Omega)$ the sequence  $\mathcal{E}_j^\sigma(\cdot,A)$ defined in \eqref{eq:perturben} $\Gamma$-converges to $\mathcal{E}^\sigma(\cdot,A)$ in $L^0(\R^d;\R^m)$. Let $f_\infty^\sigma$ and $g_\infty^\sigma$ be the functions defined in \eqref{eq:fdef-sbv} and \eqref{eq:gdef-sbv}, respectively. Let $f_\infty^0:\R^d\times\R^{m\times d}\to[0,+\infty]$ and $g_\infty^0:\R^d\times\R^{m}_0\times\mathbb{S}^{d-1}\to[0,+\infty]$ be the functions defined as
\begin{align}
f_\infty^0(x,\xi) & := \inf_{\sigma\in\Sigma} f_\infty^\sigma(x,\xi)= \lim_{\substack{\sigma\to0^+\\ \sigma\in\Sigma}} f_\infty^\sigma(x,\xi)\,,      \label{eq:5.1cdmsz} \\
g_\infty^0(x,\zeta,\nu) & := \inf_{\sigma\in\Sigma} g_\infty^\sigma(x,\zeta,\nu)= \lim_{\substack{\sigma\to0^+\\ \sigma\in\Sigma}}  g_\infty^\sigma(x,\zeta,\nu)\,. \label{eq:5.2cdmsz}
\end{align}
Then, the functionals $\mathcal{E}_j(\cdot,A)$ defined in \eqref{eq:unperturben} $\Gamma$-converge in $L^0(\R^d;\R^m)$ to the functional $\mathcal{E}^0(\cdot,A)$ given by
\begin{equation}
\mathcal{E}^0(u,A)= \int_A f_{\infty}^0\big(x,u(x),   \nabla u(x)   \big)\, {\rm d}x +\int_{J_u\cap A}g_{\infty}^0(x,[u](x),\nu_u(x))\, {\rm d}\mathcal{H}^{d-1}(x)\,,
\label{eq:unperturbedlimit}
\end{equation}
for every $A\in\mathcal{A}(\Omega)$ and $u\in GSBV^{p(\cdot)}(A;\R^m)$.
\end{theorem}
\begin{proof}
It follows from \eqref{eq:fdef-sbv} and \eqref{eq:gdef-sbv}  that  $f^{\sigma_1}_\infty\le f^{\sigma_2}_\infty$ and $g^{\sigma_1}_\infty\le g^{\sigma_2}_\infty$ for $0<\sigma_1<\sigma_2$.
Then, by the Monotone Convergence Theorem we have
\begin{equation}\label{eq:5.3dm}
\mathcal{E}^{0}(u,A)= \lim_{\substack{\sigma\to0^+\\\sigma\in \Sigma}}\mathcal{E}^{\sigma}(u,A)
\end{equation}
for every $A\in\A(\Omega)$ and every $u\in L^0(\R^d,\R^m)$ with $u|_A\in SBV^{p(\cdot)}(A,\R^m)$.

Let $\mathcal{E}'$, $\mathcal{E}''\colon L^0(\R^d,\R^m){\times} \A(\Omega) \to [0,+\infty]$ be defined by
\begin{align*}
\mathcal{E}'(\cdot,A) :=\Gamma\hbox{-}\liminf_{j\to +\infty} \mathcal{E}_{j}(\cdot,A)\quad&\text{and}\quad \mathcal{E}''(\cdot,A):=\Gamma\hbox{-}\limsup_{j\to +\infty} \mathcal{E}_{j}(\cdot,A),
\end{align*}
where we use the topology of $L^0(\R^d,\R^m)$. We subdivide the rest of the proof into steps. \\
\noindent
{\emph{Step 1:}} First, for every $A\in\A(\Omega)$, $u\in L^0(\R^d,\R^m)$ with $u|_A\in SBV^{p(\cdot)}(A,\R^m)$ and for every $\sigma\in \Sigma$ we have
$\mathcal{E}''(u,A)\le \mathcal{E}^{\sigma}(u,A)$, whence by \eqref{eq:5.3dm} we immediately get
\begin{equation}\label{eq:5.4dm}
\mathcal{E}''(u,A)\le \mathcal{E}^0(u,A)\,.
\end{equation}
\noindent
{\emph{Step 2:}} We claim that
\begin{equation}\label{eq:5.5dm}
\mathcal{E}^0(u,A)\le \mathcal{E}'(u,A)
\end{equation}
for every $A \in \A(\Omega)$ and every $u\in L^\infty(\R^d,\R^m)$.

With fixed $A$ and $u$ as above, by $\Gamma$-convergence there exists a sequence $(u_j)$ converging to $u$ in $L^0(\R^d,\R^m)$ such that
\begin{equation}\label{eq:5.6dm}
\mathcal{E}'(u,A)=\liminf_{k\to+\infty}\mathcal{E}_j(u_j,A).
\end{equation}
Let us fix $\lambda>\|u\|_{L^\infty(\R^d\!,\,\R^m)}$ and $\sigma>0$.  
By Lemma~\ref{lem:cdmsz} there exist $\mu>\lambda$, independent of $j$, and a sequence $(v_j)\subset L^\infty(\R^d,\R^m)$, converging to $u$ in measure on bounded sets, such that for every $j$ we have
\begin{align}
&\|v_j\|_{L^\infty(\R^d\!,\,\R^m)}\le \mu,
\label{eq:5.7dm}
\\
&v_j=u_j \quad\mathcal{L}^d\hbox{-a.e.\ in }\{|u_j|\le\lambda\},
\label{eq:5.8dm}
\\
&\mathcal{E}_j(v_j,A)\le (1+\sigma) \mathcal{E}_j(u_j,A) + \beta\mathcal{L}^d(A\cap\{|u_j|\ge\lambda\}).
\label{eq:5.9dm}
\end{align}
If $\mathcal{E}_j(u_j,A)<+\infty$, by the lower bounds in \ref{eq:bound1}, \ref{eq:bound2bis}, and \eqref{eq:5.9dm} the function $v_j$ belongs to $GSBV^{p(\cdot)}(A,\R^m)$  and 
\begin{equation}\label{eq:5.10dm}
\mathcal{H}^{d-1}(J_{v_j}\cap A)\le  \frac{(1+\sigma)}{\alpha}\,\mathcal{E}_j(u_j,A)+  \frac{\beta}{\alpha} \,\mathcal{L}^d(A\cap\{|u_j|\ge\lambda\})\,.
\end{equation}
By \eqref{eq:perturben} and \eqref{eq:5.7dm} this implies that 
$$
\mathcal{E}^{\sigma}_j(v_j,A)\le \mathcal{E}_j(v_j,A)+2\sigma\mu\mathcal{H}^{d-1}(J_{v_j}\cap A),
$$ 
which, in its turn, by \eqref{eq:5.9dm} and \eqref{eq:5.10dm}, leads to
$$
\mathcal{E}^{\sigma}_j(v_j,A)\le (1+\sigma) \left(1 +\frac{2\sigma\mu}{\alpha}\right) \mathcal{E}_j(u_j,A) + \beta \left(1 +\frac{2\sigma\mu}{\alpha}\right) \mathcal{L}^d(A\cap\{|u_j|\ge\lambda\})\,.
$$
This inequality trivially holds also when $\mathcal{E}_j(u_j,A)=+\infty$.
Therefore, using  \eqref{eq:5.6dm} and the inequality $\|u\|_{L^\infty(\R^d\!,\,\R^m)}<\lambda$, by $\Gamma$-convergence we get
$$
\mathcal{E}^{\sigma}(u,A)\le (1+\sigma) \left(1 +\frac{2\sigma\mu}{\alpha}\right) \mathcal{E}'(u,A)
$$
for every $\sigma\in \Sigma$. By \eqref{eq:5.3dm}, passing to the limit as $\sigma\to0^+$ we obtain \eqref{eq:5.5dm} whenever $u\in L^\infty(\R^d,\R^m)$. \\
\noindent
{\emph{Step 3:}} We now prove that
\begin{equation}\label{eq:5.11dm}
\mathcal{E}''(u,A)\le \mathcal{E}^0(u,A) \quad\hbox{for every }u\in L^0(\R^d,\R^m)\hbox{ and every }A\in \A(\Omega).
\end{equation}
Let us fix $u$ and $A$. It is enough to prove the inequality when $u|_A\in GSBV^{p(\cdot)}(A,\R^m)$. By Lemma~\ref{lem:cdmsz} for every $\sigma>0$ and for every integer $j\ge 1$ there exists $u_j \in  L^\infty(\R^d,\R^m)$, with $u_j|_A\in SBV^{p(\cdot)}(A,\R^m)$, such that $u_j=u$ $\mathcal{L}^d$-a.e.\ in $\{|u|\le j\}$ and 
$$
\mathcal{E}^0(u_j,A)\le (1+\sigma) \mathcal{E}^0(u,A) + \beta\mathcal{L}^d(A\cap\{|u|\ge j\}).
$$
By \eqref{eq:5.4dm} we have
$
\mathcal{E}''(u_j,A)\le \mathcal{E}^0(u_j,A)
$, hence
$$
\mathcal{E}''(u_j,A)\le (1+\sigma) \mathcal{E}^0(u,A) + \beta\mathcal{L}^d(A\cap\{|u|\ge j\}).
$$
Since  $u_j\to u$ in measure on bounded sets, passing to the limit as $j\to +\infty$, by the lower semicontinuity of the $\Gamma$-limsup we deduce
$$
\mathcal{E}''(u,A)\le (1+\sigma) \mathcal{E}^0(u,A).
$$
Thus, letting $\sigma\to0^+$ we obtain \eqref{eq:5.11dm}.  \\
\noindent
{\emph{Step 4:}} We now prove that
\begin{equation}\label{eq:5.13dm}
\mathcal{E}^0(u,A)\le \mathcal{E}'(u,A) \quad\hbox{for every }u\in L^0(\R^d,\R^m)\hbox{ and every }A\in \A(\Omega)\,.
\end{equation}
Given an open set $A$, it is enough to prove the inequality for a function $u$ such that $u|_A\in GSBV^{p(\cdot)}(A,\R^m)$, since otherwise $\mathcal{E}'(u,A)=+\infty$ due to the lower bounds in \ref{eq:bound1} and \ref{eq:bound2bis}. By Lemma~\ref{lem:cdmsz} for every $\sigma>0$ and every integer $j\ge 1$ there exists $u_j \in  L^\infty(\R^d,\R^m)$, with $u_j|_A\in SBV^{p(\cdot)}(A,\R^m)$, such that 
\begin{equation}
\begin{split}
& u_j=u  \mbox{\,\, $\mathcal{L}^d$-a.e.\ in\,\,} \{|u|\le j\}\,, \\
& u^\pm_j=u^\pm \mbox{\,\, $\hs^{d-1}$-a.e.\ in \,\,} J_u\cap \{|u^\pm|\le j\}\,, \\
& \mathcal{E}'(u_j,A)\le (1+\sigma) \mathcal{E}'(u,A) + \beta\mathcal{L}^d(A\cap\{|u|\ge j\})\,.
\end{split}
\label{eq:5.13bisdm}
\end{equation}
By \eqref{eq:5.5dm} we have $\mathcal{E}^0(u_j,A)\le \mathcal{E}'(u_j,A)$, which combined with \eqref{eq:5.13bisdm} gives
$$
\int_{A\cap \{|u|\le j\}} \!\!\!\!\!\!\!\! f^0_\infty(x,\nabla u)\dx +
\int_{J_u\cap A\cap\{|u^+|\le j\}\cap \{|u^-|\le j\}} \!\!\!\! \!\!\!\!\!\!\!\!\!\!\!\!\!\!\!\!\!\!\!\!\!\!\!\!\!\!\!\!\!\!\!\!\!\!\!\! g^0_\infty(x,[u],\nu_u)\,\mathrm{d}\hs^{d-1}\le (1+\sigma) \mathcal{E}'(u,A) + \beta\mathcal{L}^d(A\cap\{|u|\ge j\})\,.
$$
Letting $j\to+\infty$ we get
$$
\mathcal{E}^0(u,A)= \int_A f^0_\infty(x,\nabla u)\dx +
\int_{J_u\cap A}
g^0_\infty(x,[u],\nu_u)\,d\hs^{d-1}\le (1+\sigma) \mathcal{E}'(u,A),
$$
and then sending $\sigma\to0^+$ we obtain \eqref{eq:5.13dm}. 

The $\Gamma$-convergence of $\mathcal{E}_j(\cdot,A)$ to $\mathcal{E}^0(\cdot,A)$ in  $L^0(\R^d,\R^m)$ follows from 
\eqref{eq:5.11dm} and \eqref{eq:5.13dm}. This concludes the proof. 
\end{proof}

\section*{Acknowledgements} 

The authors are members of Gruppo Nazionale per l'Analisi Matematica, la Probabilit\`a e le loro Applicazioni (GNAMPA) of INdAM.
The authors have been supported by the project STAR PLUS 2020 – Linea 1 (21‐UNINA‐EPIG‐172) ``New perspectives in the Variational modeling of Continuum Mechanics''. The work of FS is part of the project “Variational methods for stationary and evolution problems with singularities and interfaces” PRIN 2017 financed by the Italian Ministry of Education, University, and Research.

\bibliographystyle{siam}

\end{document}